\documentclass[a4paper,10pt]{article}
\usepackage[utf8x]{inputenc}
\usepackage{amsmath,amssymb}
\usepackage{tikz,pgf}
\usepackage{amsthm}
\usepackage{caption}

\usetikzlibrary{calc,through,intersections,arrows}

\newcommand{\ZZ}{\mathbb{Z}}

\DeclareMathOperator{\id}{id}
\DeclareMathOperator{\Aut}{Aut}

\renewcommand{\theenumii}{(\alph{enumii})}

\makeatletter
\renewcommand{\p@enumii}{}
\makeatother

\newtheorem{theorem}{Theorem}[section]
\newtheorem{lemma}[theorem]{Lemma}
\newtheorem{corollary}[theorem]{Corollary}
\newtheorem{proposition}[theorem]{Proposition}

\theoremstyle{definition}
\newtheorem{definition}{Definition}[section]
\newtheorem{construction}[definition]{Construction}
\newtheorem{remark}[theorem]{Remark}

\newcommand\liaisonp{liaison edges}
\newcommand\liaisons{liaison edge}

\newcounter{brokenenumi}

\title{Characterisation of symmetries of unlabelled triangulations and
	its applications\thanks{An extended abstract of this paper has been accepted for the proceedings of EUROCOMB 2015.}}
\author{Mihyun Kang\thanks{Supported by Austrian Science Fund (FWF): P27290 and W1230 II} and Philipp Sprüssel\textsuperscript{$\dagger$}}

\begin{document}

\maketitle

\begin{abstract}
  We give a full characterisation of the symmetries of unlabelled
  triangulations and derive a constructive decomposition of unlabelled
  triangulations depending on their symmetries. As an application of
  these results we can deduce a complete enumerative description of
  unlabelled cubic planar graphs.
\end{abstract}

\section{Introduction}\label{sec:intro}

One of the most studied problems in enumerative combinatorics has been
the enumeration of graphs \emph{embedded} or \emph{embeddable} on a
surface, in particular \emph{planar} graphs and triangulations.
Enumeration of labelled planar graphs, maps, and triangulations~\cite{BenderGaoWormald,GN,MSW,OPT03,Tutte-census,censusmaps63,Tutte-polyhedra},
properties of random labelled planar graphs like connectedness~\cite{GN,MSW}, degree distribution and maximum
degree~\cite{Dowden,DGN-planardegree,pana-maxdeg,GaoWormald,GMSW,MR,pana-deg},
containment of subgraphs~\cite{Dowden-evolution,FP-core,GN,MSW,PS-maximalsubgraph},
and random sampling~\cite{BGK-sampling,Fusy-sampler}
have been studied intensively. In contrast to this abundance of
results, many structural and enumerative problems concerning
\emph{unlabelled} (i.e.~non-isomorphic) graphs on a surface are still
open. In particular, the fundamental problem of determining the
asymptotic number of unlabelled planar graphs remains unsolved. The
best known partial results are enumerations of subfamilies of
unlabelled planar graphs such as outerplanar graphs~\cite{Outerplanar}
or series parallel graphs~\cite{subcritical}.

In his seminal work~\cite{censusmaps63}, Tutte conjectured that almost
all planar \emph{maps} (i.e.\ graphs \emph{embedded} on a sphere) are
asymmetric---a conjecture that was later proved by Richmond and
Wormald~\cite{RW95}. While this tells us that almost all planar maps
have no non-trivial automorphisms, the opposite is true for planar
\emph{graphs}: McDiarmid, Steger, and Welsh~\cite{MSW} showed that
almost all planar graphs have exponentially many automorphisms. Thus,
it is impossible to derive the asymptotic number of unlabelled planar
graphs from that of labelled planar graphs.

One of the fundamental tools for the enumeration of graphs and maps is
\emph{constructive decomposition}. The most prominent example is
Tutte's decomposition~\cite{censusmaps63}: 2-connected graphs can be
characterised by three disjoint subclasses of graphs, each of which is
decomposed into smaller building blocks, with 3-connected graphs as
one of the base cases, and vice versa, the building blocks construct
all possible 2-connected graphs. Constructive decompositions can be
interpreted as functional operations of generating functions that
encode the enumerative information of the class of graphs or maps that
is being decomposed. Following these lines, Chapuy et
al.~\cite{grammar} used the decomposition from~\cite{censusmaps63} to
derive a grammar that allows to transfer enumeration results for the
3-connected graphs in a given graph class $\mathcal{G}$ to the whole
class $\mathcal{G}$. As 3-connected planar graphs have a unique
embedding \emph{up to orientation} by Whitney's Theorem~\cite{Whitney}, the
problem of enumerating labelled planar graphs is reduced to the
enumeration of labelled 3-connected planar \emph{maps}.

For unlabelled 3-connected graphs, however, the two embeddings
provided by Whitney's Theorem are not necessarily distinct; whether we
have one or two distinct embeddings will depend on the symmetries of
the graph. A better understanding of the symmetries of 3-connected
planar graphs is therefore the key for the enumeration of unlabelled
planar graphs.

In this paper, we derive a complete description of the automorphisms
of unlabelled planar \emph{triangulations}, planar maps in which every
face boundary is a triangle (in other words, maximal planar maps). We also develop a constructive decomposition
depending on their symmetries. While triangulations are one of the
most fundamental classes of planar maps and thus their symmetries are
interesting in their own right, the results of this paper can be
extended even further: the \emph{duals} of triangulations are
precisely the 3-connected \emph{cubic} planar maps and thus the
constructive decomposition developed in this paper together with the
information about the symmetries of the maps in question can be used
to obtain an enumeration of unlabelled 3-connected cubic planar
\emph{graphs}. Using the grammar of~\cite{grammar}, this provides a
complete description of all unlabelled cubic planar graphs~\cite{cubicplanar}.
We believe that the insight gained in this work can be applied to study
the symmetry and component structure of unlabelled planar graphs, in
particular that of 3-connected unlabelled planar graphs, by carefully
characterising planar graphs with different types of symmetries.

The constructive decomposition of triangulations will consist of two
parts: the characterisation of the building blocks and the
construction of how the building blocks will be merged in order to
construct the triangulations. The building blocks will depend on the
type of symmetries a triangulation $T$ has: \emph{reflective}
symmetries, \emph{rotative} symmetries, or both. There will be three
classes of basic building blocks, called \emph{girdles}, \emph{fyke nets}, and
\emph{spindles}, each one corresponding to one of the three cases for
the symmetries of $T$. In each case, $T$ will contain a unique
subgraph $G$ from the respective class of base cases. Vice versa, we
will show that $T$ can be constructed from $G$ by \emph{inserting}
planar maps from some additional classes of maps into some of the faces
of $G$. The construction of inserting maps into faces is similar to the process used to obtain
\emph{stack triangulations}, objects that proved to have various
applications in geometry~\cite{stacktriangulations,stack-lattice,stack-4colorable,stack-schnyder}.

Part of our work is inspired by Tutte~\cite{Tutte-polyhedra}, who
derived decompositions of triangulations with reflective symmetries
and of triangulations with rotative symmetries (with the additional
property that the order of the automorphism is prime). For our
purposes, we need to consider \emph{all} possible symmetries of
triangulations and develop constructive decompositions for all three
cases (reflective, rotative, or both types of symmetries). Our
decomposition for the case of reflective symmetries will be very close
to Tutte's decomposition. Tutte's decomposition for rotative
symmetries, however, is not unique (not even when the order is prime)
and thus not a constructive decomposition. Our constructive
decomposition for rotative symmetries will only bear slight
resemblance to Tutte's decomposition. The case of both types of
symmetries has not been considered before.

This paper is organised as follows. After stating the necessary
notation and basic facts in Section~\ref{sec:preliminaries}, we
prove the aforementioned characterisation of symmetries as reflective
or rotative in Section~\ref{sec:symmetries}. In
Sections~\ref{sec:reflective} to~\ref{sec:both}, we then derive the
constructive decomposition of triangulations separately for
triangulations with reflective symmetries, with rotative symmetries,
and with both types of symmetries. We will then show in
Section~\ref{sec:constructions} how to construct the basic building
blocks and discuss the results obtained and the further work in
Section~\ref{sec:discussion}.

\section{Preliminaries}\label{sec:preliminaries}

All graphs and maps considered in this paper are
\emph{unlabelled} (i.e.\ are isomorphism classes of labelled graphs)
and \emph{simple} (i.e.\ no two edges have the same two end
vertices). Call a triangulation \emph{trivial} if it has at most four
vertices, so its underlying graph is a triangle or the complete
graph $K_4$ on four vertices. In view of the results of this paper,
these trivial triangulations represent degenerate cases of the
structures considered. In order to keep the results simple, we will
thus consider only non-trivial triangulations. Note that in a
non-trivial triangulation, no two faces have the same set of
vertices. For the rest of this paper, all triangulations are
considered to be non-trivial.

A \emph{face} of a planar map $G$ on a sphere $S$ is a
connected component of $S\setminus G$. We refer to the
vertices, edges, and faces of $G$ as its \emph{cells} of
dimension $0$, $1$, and $2$, respectively. Two cells of different
dimension are called \emph{incident} if one is contained in the
(topological) boundary of the other. Two cells of the same dimension
are \emph{adjacent} if there is a third cell incident with both.

An \emph{isomorphism} between planar maps $G,H$ is a bijective map
$\varphi\colon G\to H$ that maps each cell to a cell of the same
dimension and preserves incidencies. If $G=H$, then we call $\varphi$
an \emph{automorphism}. Note that for every isomorphism $\varphi$ of
planar maps, we can find a homeomorphism of the sphere that maps
every point in a cell $c$ to a point in the cell $\varphi(c)$. We can
therefore view isomorphisms of planar maps as special
homeomorphisms of the sphere. The automorphisms of a given
triangulation $T$ form a group which is denoted by $\Aut(T)$. A cell
$c$ is \emph{invariant} under a given automorphism $\varphi$ if
$\varphi(c) = c$. We also say that \emph{$\varphi$ fixes $c$}. A set
$A$ of cells is invariant if $\varphi(A)=A$, note that each element
of $A$ does not have to be invariant. The automorphisms under which
a given cell $c$ is invariant form a group; we denote it by $\Aut(c,T)$.

In enumerative combinatorics, triangulations are often considered
with a given \emph{rooting}; in other words, a certain cell---sometimes even
several cells---are required to be invariant under all automorphisms
that are considered. In this paper, all triangulations will have a
single cell $c_0$ as a root; we will thus consider only automorphisms
in $\Aut(c_0,T)$.

The most restrictive kind of rooting is the \emph{strong rooting}
consisting of a vertex, edge, and face that are mutually incident.
Isomorphisms between planar maps $G$ and $H$ with a strong rooting
are always supposed to map roots of $G$ to roots of $H$. We
will later see (Lemma~\ref{lem:identity}) that a triangulation with a
strong rooting has only the identity as an automorphism.

An explicit formula for the
number of triangulations with a strong rooting has been obtained by
Tutte~\cite{Tutte-census}. More generally, Brown~\cite{Brown} derived
a formula for the number of \emph{near-triangulations}. A planar map
$N$ with a strong rooting consisting of a face $f_N$, an edge $e_N$,
and a vertex $v_N$ is called a \emph{near-triangulation} if $f_N$ is
bounded by a cycle of any length $\ge 3$ while all other faces are bounded by
triangles (see Figure~\ref{fig:neartriangulation}). The root face
$f_N$ is called the \emph{outer face} of $N$, all vertices and edges
on its boundary---in particular the root vertex $v_N$ and the root
edge $e_N$---are called \emph{outer vertices} or \emph{outer edges} of
$N$, respectively. All other vertices, edges, and faces of $N$ are its
\emph{inner vertices}, \emph{inner edges}, or \emph{inner faces},
respectively. The number of near-triangulations with a strong
rooting with $m+3$ outer vertices and $n$ inner vertices is
\begin{equation*}
  A(n,m) = \frac{2(2m+3)!\,(4n+2m+1)!}{(m+2)!\,m!\,n!\,(3n+2m+3)!}\,.
\end{equation*}
The number of triangulations with $n+3$ vertices (i.e.\ $n$ inner vertices) and a strong rooting
is obviously given by $A(n,0)$.

\begin{figure}[htbp]
  \centering
  \begin{tikzpicture}
    \coordinate (A) at (280:2);
    \coordinate (B) at (330:2);
    \coordinate (C) at (25:2);
    \coordinate (D) at (70:2);
    \coordinate (E) at (125:2);
    \coordinate (F) at (180:2);
    \coordinate (G) at (230:2);

    \coordinate (H) at (45:.3);
    \coordinate (I) at (0:1.1);
    \coordinate (J) at (50:1.2);
    \coordinate (K) at (210:1.2);

    \draw[line width=4pt,black!30] (D) -- (E);
    \filldraw[black!30] (D) circle (4pt);
    \draw (A) -- (B) -- (C) -- (D) -- (E) -- (F) -- (G) -- cycle;
    \draw (A) -- (E);
    \draw (A) -- (H);
    \draw (A) -- (I);
    \draw (A) -- (K);
    \draw (B) -- (I);
    \draw (C) -- (I);
    \draw (C) -- (J);
    \draw (D) -- (H);
    \draw (D) -- (J);
    \draw (E) -- (H);
    \draw (E) -- (K);
    \draw (F) -- (K);
    \draw (G) -- (K);
    \draw (H) -- (I) -- (J) -- cycle;

    \filldraw (A) circle (1.5pt);
    \filldraw (B) circle (1.5pt);
    \filldraw (C) circle (1.5pt);
    \filldraw (D) circle (1.5pt);
    \filldraw (E) circle (1.5pt);
    \filldraw (F) circle (1.5pt);
    \filldraw (G) circle (1.5pt);
    \filldraw (H) circle (1.5pt);
    \filldraw (I) circle (1.5pt);
    \filldraw (J) circle (1.5pt);
    \filldraw (K) circle (1.5pt);

    \draw (A) node[anchor=110] {$u_4$};
    \draw (B) node[anchor=160] {$u_5$};
    \draw (C) node[anchor=190] {$u_6$};
    \draw (D) node[anchor=250] {$v_N$};
    \draw (E) node[anchor=290] {$u_1$};
    \draw (F) node[anchor=0] {$u_2$};
    \draw (G) node[anchor=60] {$u_3$};

    \draw ($(D)!.5!(E)$) node[anchor=280] {$e_N$};

    \draw (-2.75,1.25) node {$f_N$};

    \draw[white] (3,0) node {\Large $N$};
  \end{tikzpicture}
  \caption{A near-triangulation $N$ with root face $f_N$, root edge
    $e_N$, and root vertex $v_N$. The outer vertices of $N$ are
    $v_N,u_1,\dotsc,u_6$; the outer edges are
    $e_N=v_Nu_1,u_1u_2,\dotsc,u_5u_6,u_6v_N$.}
  \label{fig:neartriangulation}
\end{figure}
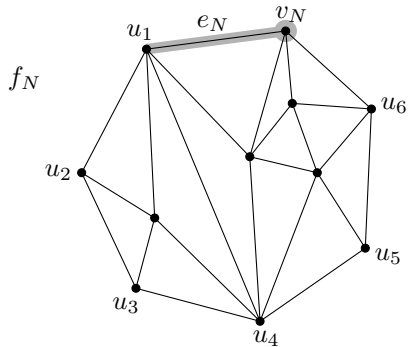

If a graph $G$ contains a cycle $C$ and an edge $e$ that does not
belong to $C$ but connects two vertices of $C$, then we call $e$ a
\emph{chord of $C$}. An inner edge of a near-triangulation $N$ is a
\emph{chord of $N$} if it is a chord of the cycle bounding the outer
face (e.g.\ the edge $u_1u_4$ in Figure~\ref{fig:neartriangulation} is a
chord).

Our goal is to provide a \emph{constructive decomposition} of triangulations.
The reverse direction of this decomposition will rely on the operation
of \emph{inserting near-triangulations} into faces of a given planar
map (see Figure~\ref{fig:inserting}).

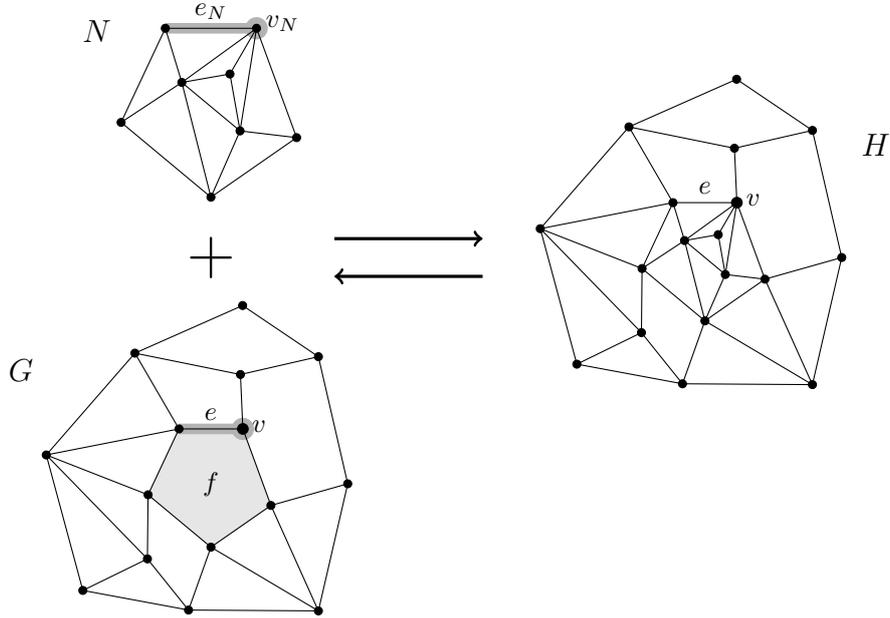
\begin{figure}[htbp]
  \centering
  \begin{tikzpicture}
    \coordinate (Center1) at (0,2);

    \coordinate (A) at (60:1.2);
    \coordinate (B) at (120:1.2);
    \coordinate (C) at (190:1.2);
    \coordinate (D) at (270:1.2);
    \coordinate (E) at (340:1.2);
    \coordinate (F) at (60:.5);
    \coordinate (G) at (140:.5);
    \coordinate (H) at (320:.5);

    \draw[line width=4pt,black!30] (Center1) +(A) -- +(B);
    \filldraw[black!30] (Center1) +(A) circle (4pt);
    \draw (Center1) +(A) -- +(B) -- +(C) -- +(D) -- +(E) -- cycle;
    \draw (Center1) +(A) -- +(F);
    \draw (Center1) +(A) -- +(G);
    \draw (Center1) +(A) -- +(H);
    \draw (Center1) +(B) -- +(G);
    \draw (Center1) +(C) -- +(G);
    \draw (Center1) +(D) -- +(G);
    \draw (Center1) +(D) -- +(H);
    \draw (Center1) +(E) -- +(H);
    \draw (Center1) +(F) -- +(G) -- +(H) -- cycle;
    
    \filldraw (Center1) +(A) circle (1.5pt);
    \filldraw (Center1) +(B) circle (1.5pt);
    \filldraw (Center1) +(C) circle (1.5pt);
    \filldraw (Center1) +(D) circle (1.5pt);
    \filldraw (Center1) +(E) circle (1.5pt);
    \filldraw (Center1) +(F) circle (1.5pt);
    \filldraw (Center1) +(G) circle (1.5pt);
    \filldraw (Center1) +(H) circle (1.5pt);
    
    \draw (Center1) +($(A)!.5!(B)$) node[anchor=south] {$e_N$};
    \draw (Center1) +(A) node[anchor=190] {$v_N$};
    \draw (Center1) +(-1.5,1) node{\large $N$};

    \draw (0,0) node {\Huge $+$};

    \coordinate (Center2) at (0,-3);

    \coordinate (Ap) at ($(0,0)!.7!(A)$);
    \coordinate (Bp) at ($(0,0)!.7!(B)$);
    \coordinate (Cp) at ($(0,0)!.7!(C)$);
    \coordinate (Dp) at ($(0,0)!.7!(D)$);
    \coordinate (Ep) at ($(0,0)!.7!(E)$);
    \coordinate (I) at (75:1.5);
    \coordinate (J) at (120:2);
    \coordinate (K) at (170:2.2);
    \coordinate (L) at (230:1.3);
    \coordinate (M) at (260:1.7);
    \coordinate (N) at (310:2.2);
    \coordinate (O) at (0:1.8);
    \coordinate (P) at (50:2.2);
    \coordinate (Q) at (80:2.4);
    \coordinate (R) at (220:2.2);

    \filldraw[black!10] (Center2) +(Ap) -- +(Bp) -- +(Cp) -- +(Dp) -- +(Ep) -- cycle;
    \draw[line width=4pt,black!30] (Center2) +(Ap) -- +(Bp);
    \filldraw[black!30] (Center2) +(Ap) circle (4pt);
    \draw (Center2) +(Ap) -- +(Bp) -- +(Cp) -- +(Dp) -- +(Ep) -- cycle;
    \draw (Center2) +(Ap) -- +(I);
    \draw (Center2) +(Bp) -- +(J);
    \draw (Center2) +(Bp) -- +(K);
    \draw (Center2) +(Cp) -- +(K);
    \draw (Center2) +(Cp) -- +(L);
    \draw (Center2) +(Dp) -- +(M);
    \draw (Center2) +(Dp) -- +(N);
    \draw (Center2) +(Ep) -- +(N);
    \draw (Center2) +(Ep) -- +(O);
    \draw (Center2) +(I) -- +(J) -- +(K) -- +(L) -- +(M) -- +(N) -- +(O) -- +(P) -- cycle;
    \draw (Center2) +(J) -- +(Q) -- +(P);
    \draw (Center2) +(K) -- +(R);
    \draw (Center2) +(L) -- +(R);
    \draw (Center2) +(M) -- +(R);
    
    \filldraw (Center2) +(Ap) circle (2pt);
    \filldraw (Center2) +(Bp) circle (1.5pt);
    \filldraw (Center2) +(Cp) circle (1.5pt);
    \filldraw (Center2) +(Dp) circle (1.5pt);
    \filldraw (Center2) +(Ep) circle (1.5pt);
    \filldraw (Center2) +(I) circle (1.5pt);
    \filldraw (Center2) +(J) circle (1.5pt);
    \filldraw (Center2) +(K) circle (1.5pt);
    \filldraw (Center2) +(L) circle (1.5pt);
    \filldraw (Center2) +(M) circle (1.5pt);
    \filldraw (Center2) +(N) circle (1.5pt);
    \filldraw (Center2) +(O) circle (1.5pt);
    \filldraw (Center2) +(P) circle (1.5pt);
    \filldraw (Center2) +(Q) circle (1.5pt);
    \filldraw (Center2) +(R) circle (1.5pt);
    
    \draw (Center2) +($(Ap)!.5!(Bp)$) node[anchor=south] {$e$};
    \draw (Center2) +(Ap) node[anchor=190] {$v$};
    \draw (Center2) node {$f$};
    \draw (Center2) +(-2.5,1.5) node {\large $G$};
    
    \coordinate (Center3) at (6.5,0);

    \coordinate (Fp) at ($(0,0)!.7!(F)$);
    \coordinate (Gp) at ($(0,0)!.7!(G)$);
    \coordinate (Hp) at ($(0,0)!.7!(H)$);

    \draw (Center3) +(Ap) -- +(Bp) -- +(Cp) -- +(Dp) -- +(Ep) -- cycle;
    \draw (Center3) +(Ap) -- +(I);
    \draw (Center3) +(Bp) -- +(J);
    \draw (Center3) +(Bp) -- +(K);
    \draw (Center3) +(Cp) -- +(K);
    \draw (Center3) +(Cp) -- +(L);
    \draw (Center3) +(Dp) -- +(M);
    \draw (Center3) +(Dp) -- +(N);
    \draw (Center3) +(Ep) -- +(N);
    \draw (Center3) +(Ep) -- +(O);
    \draw (Center3) +(I) -- +(J) -- +(K) -- +(L) -- +(M) -- +(N) -- +(O) -- +(P) -- cycle;
    \draw (Center3) +(J) -- +(Q) -- +(P);
    \draw (Center3) +(K) -- +(R);
    \draw (Center3) +(L) -- +(R);
    \draw (Center3) +(M) -- +(R);
    \draw (Center3) +(Ap) -- +(Fp);
    \draw (Center3) +(Ap) -- +(Gp);
    \draw (Center3) +(Ap) -- +(Hp);
    \draw (Center3) +(Bp) -- +(Gp);
    \draw (Center3) +(Cp) -- +(Gp);
    \draw (Center3) +(Dp) -- +(Gp);
    \draw (Center3) +(Dp) -- +(Hp);
    \draw (Center3) +(Ep) -- +(Hp);
    \draw (Center3) +(Fp) -- +(Gp) -- +(Hp) -- cycle;
    
    \filldraw (Center3) +(Ap) circle (2pt);
    \filldraw (Center3) +(Bp) circle (1.5pt);
    \filldraw (Center3) +(Cp) circle (1.5pt);
    \filldraw (Center3) +(Dp) circle (1.5pt);
    \filldraw (Center3) +(Ep) circle (1.5pt);
    \filldraw (Center3) +(Fp) circle (1.5pt);
    \filldraw (Center3) +(Gp) circle (1.5pt);
    \filldraw (Center3) +(Hp) circle (1.5pt);
    \filldraw (Center3) +(I) circle (1.5pt);
    \filldraw (Center3) +(J) circle (1.5pt);
    \filldraw (Center3) +(K) circle (1.5pt);
    \filldraw (Center3) +(L) circle (1.5pt);
    \filldraw (Center3) +(M) circle (1.5pt);
    \filldraw (Center3) +(N) circle (1.5pt);
    \filldraw (Center3) +(O) circle (1.5pt);
    \filldraw (Center3) +(P) circle (1.5pt);
    \filldraw (Center3) +(Q) circle (1.5pt);
    \filldraw (Center3) +(R) circle (1.5pt);
    
    \draw (Center3) +($(Ap)!.5!(Bp)$) node[anchor=south] {$e$};
    \draw (Center3) +(Ap) node[anchor=190] {$v$};
    \draw (Center3) +(2.25,1.5) node {\large $H$};
    
    \draw[very thick,->] (0,.25) +($(0,0)!.25!(Center3)$) -- +($(0,0)!.55!(Center3)$);
    \draw[very thick,<-] (0,-.25) +($(0,0)!.25!(Center3)$) -- +($(0,0)!.55!(Center3)$);
  \end{tikzpicture}
  \caption{Inserting a near-triangulation $N$ into the face $f$ of a
    graph $G$ resulting in a graph $H$; this is a
    constructive decomposition of $H$ into $(N,G)$.}
  \label{fig:inserting}
\end{figure}

To make this operation precise, let $N$ be a near-triangulation with
$m+3$ outer vertices and let $G$ be a planar map; denote by $S_N$ and
$S_G$ the spheres on which $N$ and $G$ are embedded, respectively. Suppose
that $f$ is a face of $G$ that is bounded by a cycle of length $m+3$;
let $e$ be an edge on the boundary of $f$ and let $v$ be one of the
end vertices of $e$. We obtain a new planar map $H$ as follows:
Deleting the outer face of $N$ from the sphere $S_N$ results in a
space $D_N$ homeomorphic to the unit disc; similarly, deleting $f$
from the sphere $S_G$ results in a space $D_G$ homeomorphic to the
unit disc. Note that by construction the boundary $C_N$ of $D_N$
(respectively the boundary $C_G$ of $D_G$) is the boundary of the
outer face of $N$ (respectively that of $f$) and thus the point set of
a cycle of length $m+3$. Let $\sigma\colon C_N\to C_G$ be a
homeomorphism that
\begin{itemize}
\item maps vertices to vertices;
\item maps the root vertex $v_N$ of $N$ to $v$; and
\item maps the (point set of the) root edge $e_N$ of $N$ to (the point
  set of) $e$.
\end{itemize}
The quotient space $(D_N \cup D_G)/\sigma$ obtained from the union
$D_N \cup D_G$ by identifying every point $x\in C_N$ with $\sigma(x)$
is a sphere on which a graph $H$ is embedded. We say that $H$ is obtained
from $G$ by \emph{inserting $N$ into $f$ at $v$ and $e$}. If $G$ is
rooted and $f$ is not its root face, then we consider $H$ to have the
rooting it inherits from $G$.

If $T$ is a triangulation and $G$ is a 2-connected subgraph of $T$, then $T$ can
always be obtained from $G$ by inserting near-triangulations into
several of its faces: suppose that for each face $f$ of $G$, we choose an edge
$e_f$ on its boundary and one of its end vertices $v_f$. Then the
near-triangulation $N_f$ that is inserted into $f$ at $v_f$ and $e_f$
in order to obtain $T$ is uniquely defined. We say that $N_f$ is the
near-triangulation \emph{induced by $(T,f)$ at $v_f$ and $e_f$}.

\section{Symmetries of triangulations}\label{sec:symmetries}

Throughout this paper, let $T$ be a triangulation and
choose a cell $c_0$ as the root of $T$. As mentioned before, we will
consider automorphisms in $\Aut(c_0,T)$, i.e. automorphisms of $T$
that fix the root $c_0$.

For every cell $c$ of $T$ of a given
dimension $d$ the numbers of incident cells of dimensions
$d+1\pmod{3}$ and $d+2\pmod{3}$ are the same. We call this number the
\emph{degree} of $c$ and denote it by $d(c)$. Clearly, for a vertex this
notion of degree equals the graph theoretical definition; every
edge has degree $2$; every face of $T$
has degree $3$. 
The \emph{distance} of two cells $c,c'$ is the smallest number $\ell$
for which there is a sequence of $\ell+1$ cells starting at $c$ and
ending at $c'$ such that every two consecutive cells in the sequence
are incident. Note that every two cells have a distance.

Given a cell $c$ of $T$, the set of cells incident with $c$ has a
cyclic order $(c_1,c_2,\dotsc,c_{2d(c)})$ in which two cells
are consecutive if and only if they are incident in the triangulation
(see Figure~\ref{fig:incidentcells1}). This order is unique up to
orientation. Two cells $c_{\alpha},c_{\beta}$ with
$\alpha,\beta\in\{1,2,\dotsc,2d(c)\}$ are said to \emph{lie opposite}
at $c$ if $|\alpha-\beta| = d(c)$. We observe that if $c$ is a face, then its boundary is a triangle and every
vertex $v$ of this triangle is opposite at $c$ to the edge of the triangle
that is not incident with $v$. If $c$ is an edge,
then its two incident faces lie opposite at $c$ and so do its end
vertices. If $c$ is a vertex, the situation
depends on the parity of $d(c)$: for even $d(c)$, every incident edge
lies opposite to another incident edge while every face lies opposite
to a face. For odd $d(c)$, every edge lies opposite to a face.

\begin{figure}[htbp]
  \centering
  \begin{tikzpicture}
    \coordinate (Centre1) at (2.75,-3.5);

    \coordinate (A) at (230:2);
    \coordinate (B) at (315:2);
    \coordinate (C) at (40:2.4);
    \coordinate (D) at (160:1.8);

    \filldraw[black!20] (Centre1) +(A) -- +(B) -- +(C) -- +(D) -- cycle;

    \draw (Centre1) +(A) -- +(B) -- +(C) -- +(D) -- cycle;
    \draw (Centre1) +(A) -- +(0,0) -- +(B);
    \draw (Centre1) +(C) -- +(0,0) -- +(D);
    \draw (Centre1) ++(A) -- +(.2,-.6);
    \draw (Centre1) ++(A) -- +(-.6,-.2);
    \draw (Centre1) ++(B) -- +(-.2,-.6);
    \draw (Centre1) ++(B) -- +(.6,.2);
    \draw (Centre1) ++(C) -- +(.6,-.2);
    \draw (Centre1) ++(C) -- +(.4,.4);
    \draw (Centre1) ++(C) -- +(-.5,.3);
    \draw (Centre1) ++(D) -- +(.2,.6);
    \draw (Centre1) ++(D) -- +(-.4,-.4);

    \filldraw[black!5] (Centre1) +(0,0) circle (2.5pt);
    \draw[line width=.75pt] (Centre1) +(0,0) node[right=2pt] {\Large $c$} circle (2.75pt);
    \filldraw (Centre1) +(A) circle (1.5pt);
    \filldraw (Centre1) +(B) circle (1.5pt);
    \filldraw (Centre1) +(C) circle (1.5pt);
    \filldraw (Centre1) +(D) circle (1.5pt);

    \draw (Centre1) +(1.1,0) node {$c_1$};
    \draw (Centre1) +($(C)!.5!(0,0)$) node[circle,fill=black!20] {$c_2$};
    \draw (Centre1) +(0,.6) node {$c_3$};
    \draw (Centre1) +($(D)!.5!(0,0)$) node[circle,fill=black!20] {$c_4$};
    \draw (Centre1) +(-1,-.3) node {$c_5$};
    \draw (Centre1) +($(A)!.5!(0,0)$) node[circle,fill=black!20] {$c_6$};
    \draw (Centre1) +(.05,-.9) node {$c_7$};
    \draw (Centre1) +($(B)!.5!(0,0)$) node[circle,fill=black!20] {$c_8$};

    \coordinate (Centre2) at (0,0);

    \coordinate (M) at (330:1.5);
    \coordinate (N) at (90:1.4);
    \coordinate (O) at (200:1.7);

    \filldraw[black!5] (Centre2) +(M) -- +(N) -- +(O) -- cycle;

    \draw[line width=3.5pt] (Centre2) +(M) -- +(N);
    \draw[line width=3.5pt] (Centre2) +(N) -- +(O);
    \draw[line width=3.5pt] (Centre2) +(O) -- +(M);
    \draw[line width=2.5pt,black!20] (Centre2) +(M) -- +(N);
    \draw[line width=2.5pt,black!20] (Centre2) +(N) -- +(O);
    \draw[line width=2.5pt,black!20] (Centre2) +(O) -- +(M);

    \draw (Centre2) ++(M) -- +(-.3,-.6);
    \draw (Centre2) ++(M) -- +(.7,.2);
    \draw (Centre2) ++(N) -- +(.8,-.1);
    \draw (Centre2) ++(N) -- +(-.6,.4);
    \draw (Centre2) ++(O) -- +(-.1,.8);
    \draw (Centre2) ++(O) -- +(-.8,-.1);
    \draw (Centre2) ++(O) -- +(.2,-.7);

    \draw (Centre2) +(-.1,.1) node {\Large $c$};
    \filldraw[black!30] (Centre2) +(M) circle (2.5pt);
    \draw[line width=.75pt] (Centre2) +(M) node[anchor=135] {$c_1$} circle (2.75pt);
    \draw (Centre2) +($(M)!.5!(N)$) node[anchor=190] {$c_2$};
    \filldraw[black!30] (Centre2) +(N) circle (2.5pt);
    \draw[line width=.75pt] (Centre2) +(N) node[anchor=255] {$c_3$} circle (2.75pt);
    \draw (Centre2) +($(N)!.5!(O)$) node[anchor=310] {$c_4$};
    \filldraw[black!30] (Centre2) +(O) circle (2.5pt);
    \draw[line width=.75pt] (Centre2) +(O) node[anchor=55] {$c_5$} circle (2.75pt);
    \draw (Centre2) +($(O)!.5!(M)$) node[anchor=80] {$c_6$};

    \coordinate (Centre3) at (6.5,.3);

    \coordinate (P) at (0:1.5);
    \coordinate (Q) at (110:1.6);
    \coordinate (R) at (180:1.7);
    \coordinate (S) at (265:1.7);

    \filldraw[black!20] (Centre3) +(P) -- +(Q) -- +(R) -- +(S) -- cycle;

    \draw[line width=3.5pt] (Centre3) +(P) -- +(R);
    \draw[line width=2.5pt,black!5] (Centre3) +(P) -- +(R);
    \draw (Centre3) +(P) -- +(Q) -- +(R) -- +(S) -- cycle;
    \draw (Centre3) ++(P) -- +(.4,-.6);
    \draw (Centre3) ++(P) -- +(.55,.5);
    \draw (Centre3) ++(Q) -- +(.8,-.1);
    \draw (Centre3) ++(Q) -- +(.3,.6);
    \draw (Centre3) ++(Q) -- +(-.7,.2);
    \draw (Centre3) ++(R) -- +(-.4,.5);
    \draw (Centre3) ++(R) -- +(-.8,.1);
    \draw (Centre3) ++(R) -- +(-.2,-.7);
    \draw (Centre3) ++(S) -- +(-.7,-.2);
    \draw (Centre3) ++(S) -- +(.6,-.3);

    \draw (Centre3) node[circle,fill=black!20] {\Large $c$};
    \filldraw[black!30] (Centre3) +(P) circle (2.5pt);
    \draw[line width=.75pt] (Centre3) +(P) node[anchor=170] {$c_1$} circle (2.75pt);
    \draw (Centre3) +($(0,0)!.5!(Q)$) node[circle,fill=black!20] {$c_2$};
    \filldraw[black!30] (Centre3) +(R) circle (2.5pt);
    \draw[line width=.75pt] (Centre3) +(R) node[anchor=35] {$c_3$} circle (2.75pt);
    \draw (Centre3) +($(0,0)!.5!(S)$) node[circle,fill=black!20] {$c_4$};
    \filldraw (Centre3) +(Q) circle (1.5pt);
    \filldraw (Centre3) +(S) circle (1.5pt);
  \end{tikzpicture}
  \caption{A cyclic order $(c_1,c_2,\dotsc,c_{2d(c)})$ of the cells incident with a cell $c$.}
  \label{fig:incidentcells1}
\end{figure}
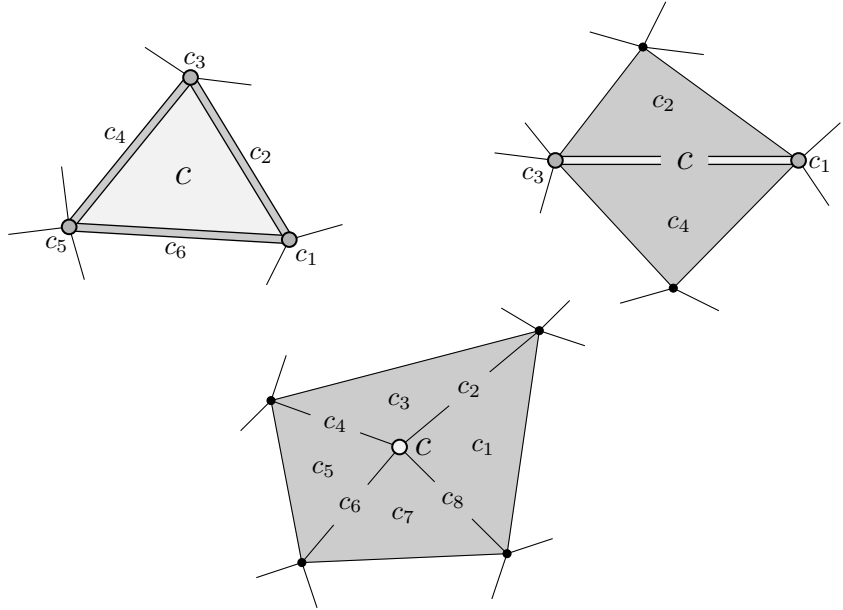

We first observe some basic properties of automorphisms in $\Aut(T)$
and $\Aut(c_0,T)$.

\begin{lemma}\label{lem:identity}
  If three mutually incident cells are invariant under an
  automorphism $\varphi\in\Aut(T)$, then $\varphi$ is the identity.
\end{lemma}

\begin{proof}
  Let $c$ be one of the cells from the statement and let 
  $(c_1,c_2,\dotsc,c_{2d(c)})$ be the cyclic order of its incident
  cells. The other two cells from the statement are part of this
  order since they are incident with $c$. Since they are incident
  with each other, they have consecutive positions in the order, $c_1$
  and $c_2$, say. Recall that the cyclic order is unique up to
  orientation; therefore, since $c_1$ and $c_2$ are invariant by
  assumption, all cells incident with $c$ are invariant.

  For every cell $c_i$ incident with $c$, the same holds: the cells
  $c$ and $c_{i+1}$ are invariant and consecutive in the cyclic
  order of the incident cells of $c_i$. Thus, all cells incident with
  $c_i$ are invariant. By induction over the distance to $c$, we
  obtain that all cells are invariant and therefore $\varphi$ is the
  identity.
\end{proof}

%

Lemma~\ref{lem:identity} in particular holds for automorphisms that
fix $c_0$: if an automorphism $\varphi\in\Aut(c_0,T)$ fixes two cells
that are incident with each other and with $c_0$, then $\varphi$ is
the identity. This immediately yields the following.

\begin{corollary}\label{cor:incident}
  An automorphism in $\Aut(c_0,T)$ is uniquely determined by its action
  on the cells incident with $c_0$.
\end{corollary}

Since an automorphism $\varphi\in\Aut(c_0,T)$ can only map cells of a
given dimension to cells of the same dimension and since the cyclic
order of the cells incident with $c_0$ is unique up to orientation, we
obtain the following.

\begin{corollary}\label{cor:dihedral}
  For every cell $c$ of $T$, $\Aut(c,T)$ is isomorphic to a subgroup of the dihedral group
  $D_{d(c)}$.
\end{corollary}

By definition, every automorphism $\varphi\in\Aut(c_0,T)$ fixes $c_0$.
But, is $c_0$ the \emph{only} invariant
cell under $\varphi$? It is not hard to prove that it is not:

\begin{lemma}\label{lem:invariant}
  For every $\varphi\in\Aut(c_0,T)$, there is at least one cell $c
  \not= c_0$ that is invariant under $\varphi$.
\end{lemma}

This can be proved by pure combinatorial means (see
e.g.~\cite{Tutte-polyhedra}), but there is also a simple topological
proof, which we provide below.

\begin{proof}
  We can find a cycle in the underlying graph of $T$ whose set of
  vertices and edges is invariant (see Figure~\ref{fig:invariantcycles}).
  Indeed, if $c_0$ is a face, then its boundary is such a cycle. If $c_0$
  is an edge, then the two faces incident with $c_0$ form an invariant
  set and hence the union of their boundaries, excluding the edge $c_0$, is
  the desired cycle. Finally, if $c_0$ is a vertex, then the vertices
  adjacent to it form the desired cycle together with all edges that lie
  opposite to $c_0$ at some face incident with $c_0$.

  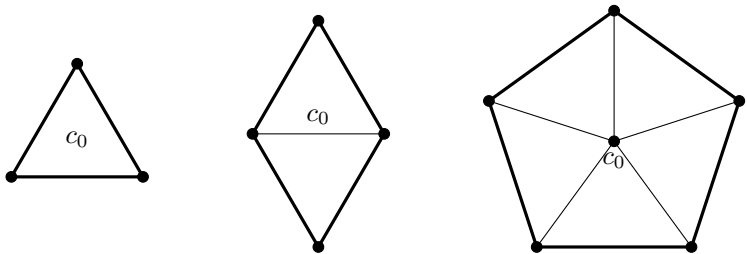
\begin{figure}[htbp]
    \centering
    \hfill
    \begin{tikzpicture}
      \draw[white] (-1,-1.5) -- (1,1.5);
      \coordinate (A) at (90:1);
      \coordinate (B) at (210:1);
      \coordinate (C) at (330:1);
      \draw[very thick] (A) -- (B) -- (C) -- cycle;
      \foreach \x in {A,B,C}
      {
        \filldraw (\x) circle (2pt);
      }
      \draw (0,0) node {$c_0$};
    \end{tikzpicture}
    \hfill
    \begin{tikzpicture}
      \coordinate (A) at (90:1);
      \coordinate (B) at (210:1);
      \coordinate (C) at (330:1);
      \coordinate (D) at (90:.5);
      \coordinate (E) at (270:.5);
      \coordinate (F) at (150:1);
      \coordinate (G) at (270:1);
      \coordinate (H) at (30:1);
      \draw (D) +(B) -- +(C);
      \draw[very thick] (D) +(C) -- +(A) -- +(B);
      \draw[very thick] (E) +(F) -- +(G) -- +(H);
      \foreach \x in {A,B,C}
      {
        \filldraw (D) +(\x) circle (2pt);
      }
      \filldraw (E) +(G) circle (2pt);
      \draw (0,0) node[anchor=south] {$c_0$};
    \end{tikzpicture}
    \hfill
    \begin{tikzpicture}
      \def\u{1.732}

      \coordinate (A) at (0,0);
      \coordinate (B) at (90:\u);
      \coordinate (C) at (162:\u);
      \coordinate (D) at (234:\u);
      \coordinate (E) at (306:\u);
      \coordinate (F) at (18:\u);
      \foreach \x in {B,C,D,E,F}
      {
        \draw (A) -- (\x);
      }
      \draw[very thick] (B) -- (C) -- (D) -- (E) -- (F) -- cycle;
      \foreach \x in {A,B,C,D,E,F}
      {
        \filldraw (\x) circle (2pt);
      }
      \draw (0,-.25) node {$c_0$};
    \end{tikzpicture}
    \hfill{}
    \caption{Finding an invariant cycle.}
    \label{fig:invariantcycles}
  \end{figure}

  The point set of this cycle divides the sphere into two discs, on
  both of which $\varphi$ induces a homeomorphism. By the Brouwer
  fixed-point theorem, both homeomorphisms have a fixed-point and
  hence the cells containing the fixed-points are invariant under
  $\varphi$. Since one of the two discs does not contain $c_0$, we have
  found the desired cell $c$.
\end{proof}

Corollary~\ref{cor:dihedral} provides a nice way of characterising
automorphisms $\varphi\in\Aut(c,T)$ for any given cell $c$: if $\varphi$ is not the
identity, then either
\begin{enumerate}
\item $\varphi$ changes the orientation of the cyclic order of the
  cells incident with $c$, in which case we call $\varphi$
  \emph{reflective at $c$}; or
\item $\varphi$ does not change the orientation of the cyclic order,
  in which case we call $\varphi$ \emph{rotative at $c$}.
\end{enumerate}
Note that the distinction between reflective and rotative
automorphisms is always up to the cell $c$ currently considered---if
an automorphism $\varphi$ of $T$ fixes two different cells $c,c'$,
then it is an element of $\Aut(c,T)$ as well as of $\Aut(c',T)$ and
the decision whether $\varphi$ is reflective or rotative at either
vertex is performed separately in each automorphism group. We will
later see (Corollary~\ref{cor:reflective}) that an automorphism cannot be reflective at one vertex and
rotative at another vertex, but we cannot use this implication yet.

The properties of the automorphism group $D_{d(c_0)}$ of a regular
$d(c_0)$-gon immediately implies the following characterisation of
reflective and rotative automorphisms.

\begin{lemma}\label{lem:fixedcells}
  Suppose that $\varphi\in\Aut(c_0,T)$ is not the identity. Then the
  following holds.
  \begin{enumerate}
  \item\label{fixedcells:refl}
    $\varphi$ is reflective at $c_0$ if and only if it fixes precisely two
    cells incident with $c_0$; these cells lie opposite at $c_0$.
  \item\label{fixedcells:rot}
    $\varphi$ is rotative at $c_0$ if and only if it fixes no cell incident
    with $c_0$.
  \end{enumerate}
\end{lemma}

We will distinguish whether $\Aut(c_0,T)$ contains
reflective automorphisms, rotative automorphisms, or both. Instead of
reflective and rotative automorphisms, we will sometimes shortly
speak of \emph{reflections} and \emph{rotations}.

Corollary~\ref{cor:dihedral} allows us to characterise $\Aut(c_0,T)$ by
the types of automorphisms it contains.

\begin{theorem}\label{thm:aut}
  For every subgroup $H$ of
  $\Aut(c_0,T)$ that contains at least one non-trivial automorphism,
  the following holds.
  \begin{enumerate}
  \item\label{aut:reflection}
    If $H$ contains a reflection but no rotation, then it
    is isomorphic to the 2-element group $\ZZ_2$.
  \item\label{aut:rotation}
    If $H$ contains $k\ge 1$ rotations but no reflection, then it is
    isomorphic to the cyclic group $\ZZ_{k+1}$ where $k+1$ is a
    divisor of $d(c_0)$.
  \item\label{aut:both}
    If $H$ contains both reflections and rotations, then it is
    isomorphic to a dihedral group $D_n$ where $n\ge 2$ is a
    divisor of $d(c_0)$.
  \end{enumerate}
\end{theorem}

\begin{proof}
  Claim~\ref{aut:reflection} follows since every reflection has order
  $2$ and there is only one reflection in $H$ since the composition of two
  distinct reflections would yield a rotation.
  Claims~\ref{aut:rotation} and~\ref{aut:both} follow directly from
  Corollary~\ref{cor:dihedral}.
\end{proof}

\section{Reflective symmetries}\label{sec:reflective}

In this section, suppose that $\Aut(c_0,T)$ contains a reflection
$\varphi$.

Our first lemma is a structural result that was first obtained by
Tutte~\cite{Tutte-polyhedra}. We include (a modified version of) its
proof for the sake of completeness.

\begin{lemma}\label{lem:pre-girdle}
  There is a cyclic sequence $(c_0,\dotsc,c_{\ell})$ of pairwise distinct cells such that for
  each cell $c$ in the sequence the following holds.
  \begin{enumerate}
  \item\label{pre-girdle:invariant}
    $c$ is invariant under $\varphi$;
  \item\label{pre-girdle:opposite}
    the predecessor and the successor of $c$ in the sequence are
    incident with $c$ and lie opposite at $c$; and
  \item\label{pre-girdle:notincident}
    no other cell in the sequence is incident with $c$.
  \end{enumerate}
\end{lemma}

\begin{proof}
  Let $I$ be the set of cells that are invariant under $\varphi$.
  Define an auxiliary graph $F$ with vertex set $I$ by joining two
  elements of $I$ by an edge whenever they are incident. Note that
  $\varphi$, although chosen as an element of $\Aut(c_0,T)$, is also
  an element of $\Aut(c,T)$ for every $c\in I$. Since $\varphi$ is not
  the identity, Lemma~\ref{lem:fixedcells} implies that every vertex
  in $F$ has degree $0$ or $2$ and thus, every component is a cycle or
  an isolated vertex. Since $\varphi$ is reflective at $c_0$ by
  assumption, $c_0$ has degree $2$ in $F$ and is thus contained in a
  cycle $C$ of $F$. The vertices of $C$---arranged in the order they
  appear on $C$---form the desired cyclic sequence: all cells
  are invariant under $\varphi$ by construction;
  Lemma~\ref{lem:fixedcells}\ref{fixedcells:refl} implies that the
  predecessor and the successor of a cell $c$ in the sequence lie
  opposite at $c$ and no other cell in the sequence is incident with
  $c$.
\end{proof}

For every edge in the sequence from Lemma~\ref{lem:pre-girdle}, its
predecessor and its successor are either its two end vertices or its
two incident faces. Every face $f$ in the sequence is preceded and
followed by a vertex and its opposite edge on the boundary of $f$.

The invariant cells from Lemma~\ref{lem:pre-girdle} play a central
role in the constructive decomposition of $T$ in the case of a
reflective symmetry: we will shortly see that these cells are the only
cells invariant under $\varphi$ and thus, they provide a way to define a
unique subgraph of $T$ that will be the basic building block in our
constructive decomposition.

\begin{definition}[Girdle]\label{def:girdle}
  Let $G$ be the planar map obtained by taking the union of all
  vertices and edges that either lie in the sequence from
  Lemma~\ref{lem:pre-girdle} or on the boundary of a face in this
  sequence. We call this subgraph of $T$ the \emph{girdle} with
  respect to $\varphi$. Its cells from the cyclic sequence are called
  \emph{central cells} of $G$, the other ones (which are only part of
  $G$ because they lie on the boundary of a face from the sequence)
  are called \emph{outer cells} of $G$. By construction, every face in
  the sequence is also a face of $G$ (and hence a central cell); the
  other faces of $G$ are called its \emph{sides}. For every face in
  the sequence, precisely one of the edges on its boundary is a
  central cell and so is the other face incident with this edge. The
  union of such two faces and their boundaries is called a
  \emph{diamond}. Note that every girdle has at least two central
  vertices; let $j(G)$ be the smallest index for which $c_{j(G)}$ is a
  vertex.
\end{definition}

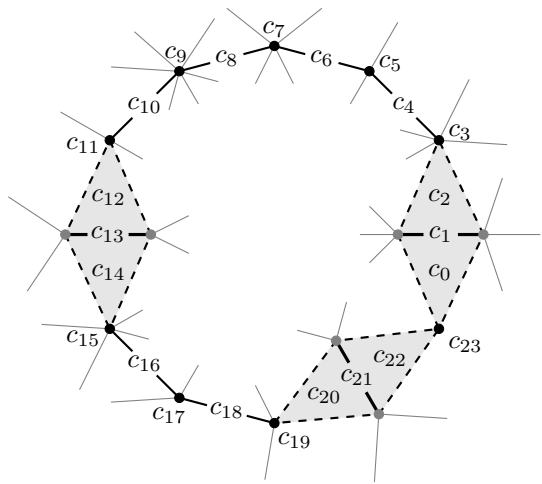
\begin{figure}[htbp]
  \centering
  \begin{tikzpicture}[scale=2.5]
    \coordinate (Ai) at (0:.65);
    \coordinate (Ao) at (0:1.1);
    \coordinate (A) at ($(Ai)!.5!(Ao)$);
    \coordinate (B) at (30:1);
    \coordinate (C) at (60:1);
    \coordinate (D) at (90:1);
    \coordinate (E) at (120:1);
    \coordinate (F) at (150:1);
    \coordinate (Gi) at (180:.65);
    \coordinate (Go) at (180:1.1);
    \coordinate (G) at ($(Gi)!.5!(Go)$);
    \coordinate (H) at (210:1);
    \coordinate (I) at (240:1);
    \coordinate (J) at (270:1);
    \coordinate (Ki) at (300:.65);
    \coordinate (Ko) at (300:1.1);
    \coordinate (K) at ($(Ki)!.5!(Ko)$);
    \coordinate (L) at (330:1);

    \draw[gray] (Ai) -- ++(-.15,.15);
    \draw[gray] (Ai) -- ++(-.2,0);
    \draw[gray] (Ai) -- ++(-.15,-.15);
    \draw[gray] (Ao) -- ++(.1,.3);
    \draw[gray] (Ao) -- ++(.3,0);
    \draw[gray] (Ao) -- ++(.1,-.3);
    \draw[gray,rotate=180] (Gi) -- ++(-.2,.1);
    \draw[gray,rotate=180] (Gi) -- ++(-.2,-.1);
    \draw[gray,rotate=180] (Go) -- ++(.2,.3);
    \draw[gray,rotate=180] (Go) -- ++(.3,-.2);
    \draw[gray,rotate=300] (Ki) -- ++(-.15,.15);
    \draw[gray,rotate=300] (Ki) -- ++(-.15,-.15);
    \draw[gray,rotate=300] (Ko) -- ++(.2,.3);
    \draw[gray,rotate=300] (Ko) -- ++(.3,-.2);
    \draw[gray,rotate=-60] (B) -- ++(.2,.3);
    \draw[gray,rotate=-60] (B) -- ++(-.2,.3);
    \draw[gray,rotate=-60] (B) -- ++(0,-.2);
    \draw[gray,rotate=-60] (B) -- ++(-.15,-.15);
    \draw[gray,rotate=-30] (C) -- ++(0,.3);
    \draw[gray,rotate=-30] (C) -- ++(0,-.2);
    \draw[gray] (D) -- ++(.25,.2);
    \draw[gray] (D) -- ++(-.25,.2);
    \draw[gray] (D) -- ++(.1,-.2);
    \draw[gray] (D) -- ++(-.1,-.2);
    \draw[gray,rotate=30] (E) -- ++(.3,.15);
    \draw[gray,rotate=30] (E) -- ++(-.1,.3);
    \draw[gray,rotate=30] (E) -- ++(-.3,.2);
    \draw[gray,rotate=30] (E) -- ++(.15,-.15);
    \draw[gray,rotate=30] (E) -- ++(0,-.2);
    \draw[gray,rotate=30] (E) -- ++(-.15,-.15);
    \draw[gray,rotate=60] (F) -- ++(0,.3);
    \draw[gray,rotate=60] (F) -- ++(0,-.2);
    \draw[gray,rotate=120] (H) -- ++(.2,.3);
    \draw[gray,rotate=120] (H) -- ++(-.2,.3);
    \draw[gray,rotate=120] (H) -- ++(0,-.2);
    \draw[gray,rotate=120] (H) -- ++(-.15,-.15);
    \draw[gray,rotate=150] (I) -- ++(.3,.2);
    \draw[gray,rotate=150] (I) -- ++(0,-.2);
    \draw[gray,rotate=180] (J) -- ++(.05,.3);
    \draw[gray,rotate=180] (J) -- ++(.1,-.2);

    \filldraw[black!10] (L) -- (Ai) -- (B) -- (Ao) -- cycle;
    \filldraw[black!10] (F) -- (Gi) -- (H) -- (Go) -- cycle;
    \filldraw[black!10] (J) -- (Ki) -- (L) -- (Ko) -- cycle;

    \draw[very thick] (Ai) -- (Ao);
    \draw[thick] (B) -- (C) -- (D) -- (E) -- (F);
    \draw[very thick] (Gi) -- (Go);
    \draw[thick] (H) -- (I) -- (J);
    \draw[very thick] (Ki) -- (Ko);

    \draw[thick,dashed] (L) -- (Ai) -- (B) -- (Ao) -- cycle;
    \draw[thick,dashed] (F) -- (Gi) -- (H) -- (Go) -- cycle;
    \draw[thick,dashed] (J) -- (Ki) -- (L) -- (Ko) -- cycle;

    \filldraw[black!50] (Ai) circle (.7pt);
    \filldraw[black!50] (Ao) circle (.7pt);
    \filldraw (B) circle (.7pt);
    \filldraw (C) circle (.7pt);
    \filldraw (D) circle (.7pt);
    \filldraw (E) circle (.7pt);
    \filldraw (F) circle (.7pt);
    \filldraw[black!50] (Gi) circle (.7pt);
    \filldraw[black!50] (Go) circle (.7pt);
    \filldraw (H) circle (.7pt);
    \filldraw (I) circle (.7pt);
    \filldraw (J) circle (.7pt);
    \filldraw[black!50] (Ki) circle (.7pt);
    \filldraw[black!50] (Ko) circle (.7pt);
    \filldraw (L) circle (.7pt);

    \filldraw[black!10] (A) circle (2.5pt);
    \draw (A) node {$c_1$};
    \draw ($(A)!.4!(B)$) node {$c_2$};
    \draw (B) node[anchor=200] {$c_3$};
    \filldraw[white] ($(B)!.5!(C)$) circle (2.5pt);
    \draw ($(B)!.5!(C)$) node {$c_4$};
    \draw (C) node[anchor=200] {$c_5$};
    \filldraw[white] ($(C)!.5!(D)$) circle (2.5pt);
    \draw ($(C)!.5!(D)$) node {$c_6$};
    \draw (D) node[anchor=270] {$c_7$};
    \filldraw[white] ($(D)!.5!(E)$) circle (2.5pt);
    \draw ($(D)!.5!(E)$) node {$c_8$};
    \draw (E) node[anchor=280] {$c_9$};
    \filldraw[white] ($(E)!.5!(F)$) circle (2.5pt);
    \draw ($(E)!.5!(F)$) node {$c_{10}$};
    \draw (F) node[anchor=10] {$c_{11}$};
    \draw ($(G)!.4!(F)$) node {$c_{12}$};
    \filldraw[black!10] (G) circle (3pt);
    \draw (G) node {$c_{13}$};
    \draw ($(G)!.4!(H)$) node {$c_{14}$};
    \draw (H) node[anchor=20] {$c_{15}$};
    \filldraw[white] ($(H)!.5!(I)$) circle (2.5pt);
    \draw ($(H)!.5!(I)$) node {$c_{16}$};
    \draw (I) node[anchor=60] {$c_{17}$};
    \filldraw[white] ($(I)!.5!(J)$) circle (3pt);
    \draw ($(I)!.5!(J)$) node {$c_{18}$};
    \draw (J) node[anchor=140] {$c_{19}$};
    \draw ($(K)!.4!(J)$) node {$c_{20}$};
    \filldraw[black!10] (K) circle (2.5pt);
    \draw (K) node {$c_{21}$};
    \draw ($(K)!.4!(L)$) node {$c_{22}$};
    \draw (L) node[anchor=150] {$c_{23}$};
    \draw ($(A)!.4!(L)$) node {$c_0$};
  \end{tikzpicture}
  \caption{The sequence of cells from Lemma~\ref{lem:pre-girdle}. The
    vertices in this picture, together with all black and all dashed
    edges, form the girdle of $T$ (see Definition~\ref{def:girdle}).
    The central cells of the girdle are the black vertices, the black
    edges, and the gray faces. The outer cells are the gray vertices
    and the dashed edges. The girdle has three diamonds.}
  \label{fig:pregirdle}
\end{figure}

\begin{lemma}\label{lem:girdle}
  The girdle $G$ has the following properties.
  \begin{enumerate}
  \item\label{girdle:sides}
    $G$ has exactly two sides $f_1,f_2$.
  \item\label{girdle:inside}
    Let $v_1=v_2:=c_{j(G)}$. If $c_{j(G)+1}$ is an edge, we let
    $e_1=e_2:=c_{j(G)+1}$; otherwise $c_{j(G)+1}$ is a face and we let
    $e_i$ be the unique edge on the boundary of $f_i$ that is incident
    with $c_{j(G)+1}$. Then $(T,f_i)$ induces a near-triangulation
    $N_i$ at $v_i$ and $e_i$.
  \item\label{girdle:isom}
    $\varphi$ is an isomorphism between $N_1$ and $N_2$.
  \item\label{girdle:invariant}
    The central cells of $G$ are precisely the cells that are
    invariant under $\varphi$.
  \end{enumerate}
\end{lemma}

\begin{proof}
  By Lemma~\ref{lem:pre-girdle}, two central cells of $G$ are
  incident if and only if they are consecutive in the cyclic sequence.
  We claim that every outer cell is contained in a unique diamond,
  which implies that the subspace of the sphere consisting of $G$ and
  the faces in its diamonds is contractible to a Jordan curve, which
  in turn implies~\ref{girdle:sides} by the Jordan curve theorem.
  Indeed, every outer cell of $G$ is a vertex or an edge that is
  contained in a diamond. If two diamonds share an outer edge, they
  also share an outer vertex $v$. Now $\varphi$ maps $v$ to the other
  outer vertex of each of the two diamonds, hence they also share
  their second outer vertex. But then the central edges contained in
  the two diamonds are distinct and have the same end vertices,
  contradicting the fact that $T$ has no double edges.

  We have thus proved~\ref{girdle:sides}. Since each side is bounded
  by a cycle,~\ref{girdle:inside} follows immediately. Let $c$ be a
  cell incident with $c_0$ that is not a central cell of $G$, then
  $c$ is contained in one of the sides of $G$ or lies on the boundary
  of precisely one side. The reflection $\varphi$ maps $c$ to a cell
  that is contained in (or lies on the boundary of) the other side of
  $G$, which yields~\ref{girdle:isom}. Finally,~\ref{girdle:invariant}
  follows directly from~\ref{girdle:isom}.
\end{proof}

\begin{figure}[htbp]
  \centering
  \begin{tikzpicture}[scale=.75]
    \clip (-4.1,-4) rectangle (3.7,3.5);
    \node (myfirstpic) at (0,0) {\includegraphics[height=6cm]{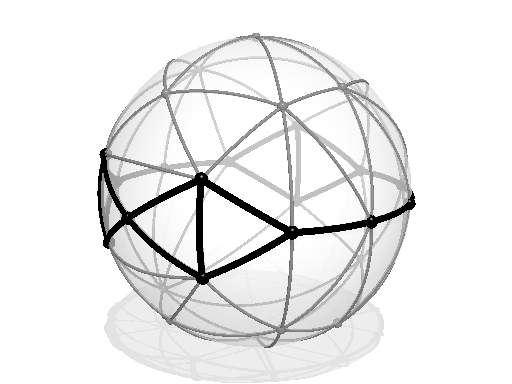}};
    \draw (-1.75,-.7) node {\large $c_0$};
    \draw (1.1,-1.2) node {\large $c_3$};
    \draw (1.6,-.5) node {\large $c_4$};
    \draw (-3,3) node {\LARGE $T$};
    \draw (-3.75,0) node {\LARGE $G$};
    \draw (3,2.2) node {\LARGE $N_1$};
    \draw (3.2,-2) node {\LARGE $N_2$};
  \end{tikzpicture}
  \caption{The girdle $G$ of a triangulation $T$ with $j(G)=3$. The
    near-triangulations $N_1$ and $N_2$ that $T$ induces on the sides
    of $G$ at $c_3$ and $c_4$ are isomorphic.}
  \label{fig:girdle}
\end{figure}
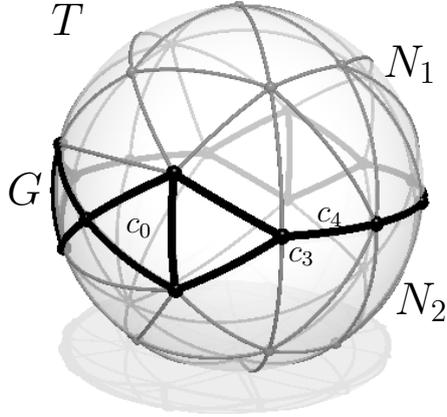

Note that Lemma~\ref{lem:girdle}\ref{girdle:invariant} implies that
no automorphism is reflective at one vertex and rotative at another
vertex:

\begin{corollary}\label{cor:reflective}
  If $\varphi\in\Aut(c_0,T)$ is reflective at $c_0$, then it is
  reflective at every cell $c$ that is invariant under $\varphi$.
\end{corollary}

By Lemma~\ref{lem:girdle}, we have a constructive decomposition of
$T$ into its girdle $G$ and two isomorphic near-triangulations
$N_1,N_2$. What other properties do $G$, $N_1$, and $N_2$ have to
satisfy? Clearly, each side of $G$ is bounded by a cycle whose length
matches the number of outer vertices of $N_1$ and $N_2$. We call this
number the \emph{length} of the girdle. The following lemma gives a
complete characterisation of the near-triangulations that can occur.

\begin{lemma}\label{lem:girdle:sides}
  Let $G$ be a graph that occurs as the girdle of some triangulation
  and let $N$ be a near-triangulation. There exists a triangulation
  $T$ with a reflective automorphism $\varphi$, $G$ as its girdle
  with respect to $\varphi$, and $N$
  as the near-triangulation from
  Lemma~\ref{lem:girdle} if and only if
  \begin{enumerate}
  \item\label{girdle:sides:length}
    the number of outer vertices of $N$ is the same as the length of
    $G$ and
  \item\label{girdle:sides:chords}
    every chord of $N$ has at least one end vertex that is an outer
    vertex of $G$.
  \end{enumerate}
\end{lemma}

\begin{proof}
  First assume that the triangulation $T$ exists.
  Property~\ref{girdle:sides:length} is immediate; in order to
  prove~\ref{girdle:sides:chords}, let $e=uv$ be a chord of $N$. If
  $u$ and $v$ are central vertices of $G$, then
  Lemma~\ref{lem:girdle}\ref{girdle:isom} would imply that $\varphi$
  maps $e$ to an edge $e'$ with the same end vertices. Since $e$ is
  not contained in $G$, Lemma~\ref{lem:girdle}\ref{girdle:invariant}
  shows that $e'\not= e$, contradicting the fact that there are no
  double edges.

  Now assume that $N$ and $G$ satisfy~\ref{girdle:sides:length}
  and~\ref{girdle:sides:chords}. Let $f_1,f_2$ be the sides of $G$ and
  let vertices $v_1,v_2$ and edges $e_1,e_2$ on the boundaries of
  $f_1$ and $f_2$, respectively, be defined as in
  Lemma~\ref{lem:girdle}. By~\ref{girdle:sides:length} we can insert
  $N$ into each side $f_i$ of $G$ at $v_i$ and $e_i$. The result
  of this operation does not have any double edges
  by~\ref{girdle:sides:chords}; since all its faces are triangular, it
  is the desired triangulation $T$.
\end{proof}

More details about the construction of graphs that can serve as
girdles and about the construction of triangulations with reflective
symmetry from their girdle and the near-triangulations characterised
by Lemma~\ref{lem:girdle:sides} will be given in Section~\ref{sec:constr:reflective}.

\section{Rotative symmetries}\label{sec:rotative}

In this section, suppose that $\Aut(c_0,T)$ contains a rotative
automorphism $\varphi$. Then the subgroup $H$ of $\Aut(c_0,T)$
generated by $\varphi$ contains no reflections and hence is isomorphic
to a cyclic group by Theorem~\ref{thm:aut}. We fix the group $H$ for
the rest of this section; let $m$ be its order.
For every cell $c$ incident with $c_0$, the cells $c,\varphi(c),
\dotsc,\varphi^{m-1}(c)$ are distinct since $\varphi,\dotsc,
\varphi^{m-1}$ are rotations and thus have no invariant cells incident
with $c_0$. Without loss of generality, we can choose $\varphi$ in
such a way that $c,\varphi(c),\dotsc,\varphi^{m-1}(c)$ are arranged
around $c_0$ in that order (in clockwise direction, say) for every
cell $c$ incident with $c_0$ (see Figure~\ref{fig:rotation}).

\begin{figure}[htbp]
  \centering
  \begin{tikzpicture}
    \coordinate (A) at (80:2);
    \coordinate (B) at (40:2.2);
    \coordinate (C) at (0:2.2);
    \coordinate (D) at (310:1.8);
    \coordinate (E) at (270:2);
    \coordinate (F) at (220:1.9);
    \coordinate (G) at (170:2);
    \coordinate (H) at (120:2.1);

    \draw (A) -- (B) -- (C) -- (D) -- (E) -- (F) -- (G) -- (H) -- cycle;
    \draw (A) -- (0,0) -- (B);
    \draw (C) -- (0,0) -- (D);
    \draw (E) -- (0,0) -- (F);
    \draw (G) -- (0,0) -- (H);
    \draw (A) -- +(30:.75);
    \draw (A) -- +(120:.75);
    \draw (B) -- +(340:.75);
    \draw (B) -- +(45:.75);
    \draw (B) -- +(100:.75);
    \draw (C) -- +(310:.75);
    \draw (C) -- +(40:.75);
    \draw (D) -- +(250:.75);
    \draw (D) -- +(315:.75);
    \draw (D) -- +(10:.75);
    \draw (E) -- +(220:.75);
    \draw (E) -- +(310:.75);
    \draw (F) -- +(160:.75);
    \draw (F) -- +(225:.75);
    \draw (F) -- +(280:.75);
    \draw (G) -- +(120:.75);
    \draw (G) -- +(210:.75);
    \draw (H) -- +(60:.75);
    \draw (H) -- +(125:.75);
    \draw (H) -- +(180:.75);

    \draw[very thick,white] (0,0) -- ($(0,0)!.25!(C)$);
    \filldraw (0,0) node[anchor=170] {\Large $c_0$} circle (2pt);
    \filldraw (A) circle (1.5pt);
    \filldraw (B) circle (1.5pt);
    \filldraw (C) circle (1.5pt);
    \filldraw (D) circle (1.5pt);
    \filldraw (E) circle (1.5pt);
    \filldraw (F) circle (1.5pt);
    \filldraw (G) circle (1.5pt);
    \filldraw (H) circle (1.5pt);

    \draw ($(B)!.5!(0,0)$) node[circle,fill=black!0] {$c$};
    \draw ($(D)!.5!(0,0)$) node[rectangle,fill=black!0] {$\varphi(c)$};
    \draw ($(F)!.5!(0,0)$) node[rectangle,fill=black!0] {$\varphi^2(c)$};
    \draw ($(H)!.5!(0,0)$) node[rectangle,fill=black!0] {$\varphi^3(c)$};
  \end{tikzpicture}
  \caption{The images of a cell $c$ incident with $c_0$ under a
    rotation $\varphi$ of order $4$.}
  \label{fig:rotation}
\end{figure}
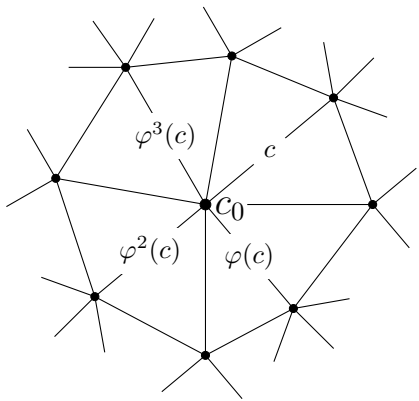

Lemma~\ref{lem:invariant} tells us that $c_0$ is not the only
invariant cell, so let $c_1$ be such a cell of shortest distance from
$c_0$ and consider a shortest path $P$ in $T$ from $c_0$ (or a vertex
incident with it---if $c_0$ is an edge or a face) to $c_1$ (or a
vertex incident with it).

\begin{lemma}\label{lem:pre-spindle}
  The paths $P,\varphi(P),\dotsc,\varphi^{m-1}(P)$ do not share any
  internal vertices. If $c_0$ is an edge or a face, all paths have
  distinct first vertices. The same is true for $c_1$ and the last
  vertices of the paths.
\end{lemma}

The special case of Lemma~\ref{lem:pre-spindle} in which $m$ is prime
has been proved by Tutte~\cite{Tutte-polyhedra}.

\begin{proof}
  First note that the paths $P,\varphi(P),\dotsc,\varphi^{m-1}(P)$ are
  distinct, because $\varphi^i(P)=\varphi^j(P)$ for $i\not= j$ would imply that
  $\varphi^i(c)=\varphi^j(c)$ for some cell $c$ incident with $c_0$,
  which we already saw to be impossible. The same argument shows that
  two paths can only share an end vertex if it is $c_0$ or $c_1$.

  Suppose two paths $\varphi^i(P)$, $\varphi^j(P)$ with $i\not= j$
  share an internal vertex. Its distance from the first vertex has to
  be the same in both paths, since otherwise the union of the two
  paths would contain a path from $c_0$ (or a vertex incident with it)
  to $c_1$ (or a vertex incident with it) shorter than $P$, a
  contradiction to the choice of $P$. Choose $i,j$ such that the
  distance of their first intersection $v$ from their first vertices
  is as small as possible. The union of the segments of the two paths
  from the first vertices to $v$ together with $c_0$ separates the
  sphere into two discs, one of which contains $c_1$ and all its
  incident cells. Any path $\varphi^k(P)$ starting in the other disc
  thus has to meet $\varphi^i(P)$ or $\varphi^j(P)$ at the latest in
  $v$. The minimal choice of $i,j$ implies that every such path goes
  through $v$.

  Therefore, there is a $k$ such that $\varphi^k(P)$ and
  $\varphi^{k+1}(P)$ meet in $v$. This means that $v$ is invariant
  under $\varphi$, contradicting the choice of $c_1$ as an invariant
  cell of minimal distance from $c_0$. This proves the lemma.
\end{proof}

Lemma~\ref{lem:pre-spindle} implies that, just like the girdle divides
the triangulation into two parts in the case of a reflective
automorphism, the paths $P,\varphi(P),\dotsc,\varphi^{m-1}(P)$
together with $c_0$ and $c_1$ divide the triangulation into $m$ parts.
The union of these paths and cells might thus serve as a building
block in our constructive decomposition.

\begin{definition}\label{def:spindle}
  Let $S$ be the union of $c_0$, $c_1$, their boundaries, and paths
  $P,\varphi(P),\dotsc,\varphi^{m-1}(P)$ satisfying the statement of
  Lemma~\ref{lem:pre-spindle} (see Figure~\ref{fig:spindle}). We call
  $S$ a \emph{spindle} of $T$ with respect to the group
  $H\subseteq\Aut(c_0,T)$. The cells $c_0$ and $c_1$ are called the \emph{north
  pole} and the \emph{south pole} of the spindle, respectively. A face
  of a spindle which is neither $c_0$ nor $c_1$ is called a
  \emph{segment} of the spindle.
  \begin{figure}[htbp]
    \centering
    \begin{tikzpicture}[scale=.75]
      \clip (-4.5,-3.9) rectangle (5.3,4.4);
      \node (myfirstpic) at (0,0) {\includegraphics[height=6cm]{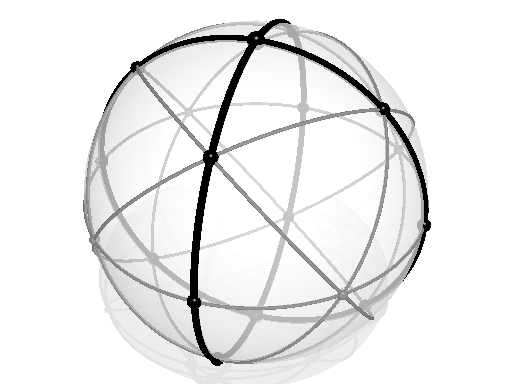}};
      \draw (-.2,3.5) node {\large $c_0$};
      \draw (-1.6,-.75) node {\Large $P$};
      \draw (-3.8,2.1) node {\Large $\varphi(P)$};
      \draw (1.5,4) node {\Large $\varphi^2(P)$};
      \draw (4.5,.2) node {\Large $\varphi^3(P)$};
      \draw (.2,-3) node {\large $c_1$};
    \end{tikzpicture}
    \caption{A triangulation and a spindle (bold) with respect to
      the group $H = \{\id,\varphi,\varphi^2,\varphi^3\}$ of
      automorphisms, in which the invariant cells $c_0,c_1$ are both
      vertices.}
    \label{fig:spindle}
  \end{figure}
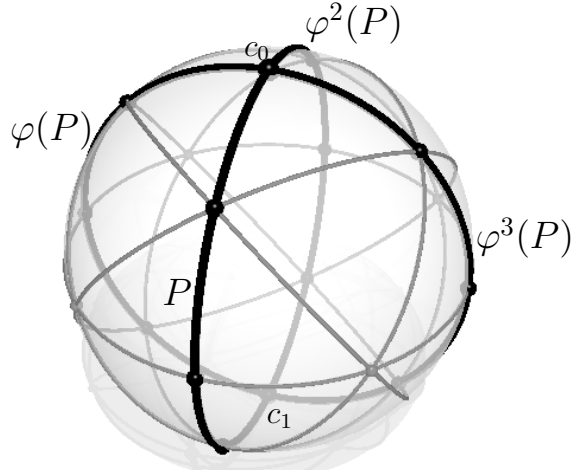
\end{definition}

Similarly to reflections, we immediately get the following result:

\begin{lemma}\label{lem:spindle}
  A spindle $S$ has the following properties:
  \begin{enumerate}
  \item $S$ has exactly $m$ segments $f_1,\dotsc,f_m$ with each $f_i$
    being bounded by a cycle containing $\varphi^{i-1}(P)$ and
    $\varphi^i(P)$;
  \item the intersection of $T$ with $f_i$ and its boundary is a
    near-triangulation $N_i$; and
  \item for every $i$, $\varphi$ is an isomorphism from $N_i$ to
    $N_{i+1}$.
  \end{enumerate}
\end{lemma}

A direct corollary of Lemma~\ref{lem:spindle} is that $c_0$ and $c_1$
are the only invariant cells under $\varphi$, even under each
$\varphi^i$ with $1\le i\le m-1$. We might thus refer to them as the
\emph{north and south pole of $T$}, not just of the spindle.

By Lemma~\ref{lem:spindle}, we can obtain all triangulations with
rotative symmetry by first constructing all possible spindles and
then inserting the same near-triangulation in each segment. However,
unlike the girdle, a spindle is \emph{not unique} since there might be
different choices for the path $P$. Figure~\ref{fig:spindle:nonunique}
shows two different spindles of the same triangulation. Since the
near-triangulation inserted in the segments in the first case is not
isomorphic to the one used in the second case, we \emph{do not} have a 1-1
correspondence between triangulations with rotative symmetries and
triangulations obtained by taking a spindle and inserting the same
near-triangulation in each segment.

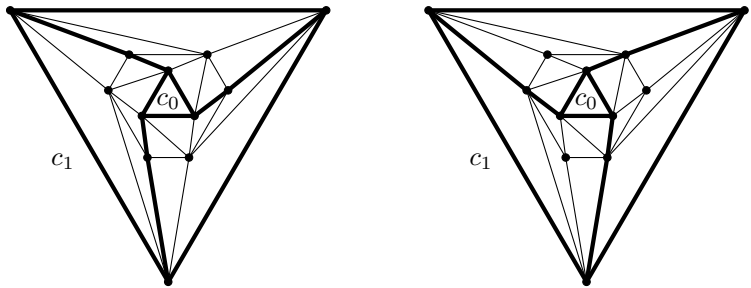
\begin{figure}[htbp]
  \centering
  \begin{tikzpicture}[scale=.4]
    \coordinate (A) at (210:1);
    \coordinate (B) at (330:1);
    \coordinate (C) at (90:1);
    \coordinate (D) at (250:2);
    \coordinate (E) at (290:2);
    \coordinate (F) at (10:2);
    \coordinate (G) at (50:2);
    \coordinate (H) at (130:2);
    \coordinate (I) at (170:2);
    \coordinate (J) at (270:6);
    \coordinate (K) at (30:6);
    \coordinate (L) at (150:6);

    \draw[ultra thick] (A) -- (B) -- (C) -- cycle;
    \draw (D) -- (E) -- (F) -- (G) -- (H) -- (I) -- cycle;
    \draw[ultra thick] (J) -- (K) -- (L) -- cycle;
    \draw[ultra thick] (A) -- (D) -- (J);
    \draw[ultra thick] (B) -- (F) -- (K);
    \draw[ultra thick] (C) -- (H) -- (L);
    \draw (A) -- (I) -- (L);
    \draw (B) -- (E) -- (J);
    \draw (C) -- (G) -- (K);
    \draw (A) -- (E) -- (K);
    \draw (B) -- (G) -- (L);
    \draw (C) -- (I) -- (J);
    \draw (0,0) node {$c_0$};
    \draw (210:4) node {$c_1$};

    \foreach \x in {A,B,C,D,E,F,G,H,I,J,K,L}
      \filldraw (\x) circle (3.5pt);
  \end{tikzpicture}
  \hspace{1cm}
  \begin{tikzpicture}[scale=.4]
    \coordinate (A) at (210:1);
    \coordinate (B) at (330:1);
    \coordinate (C) at (90:1);
    \coordinate (D) at (250:2);
    \coordinate (E) at (290:2);
    \coordinate (F) at (10:2);
    \coordinate (G) at (50:2);
    \coordinate (H) at (130:2);
    \coordinate (I) at (170:2);
    \coordinate (J) at (270:6);
    \coordinate (K) at (30:6);
    \coordinate (L) at (150:6);

    \draw[ultra thick] (A) -- (B) -- (C) -- cycle;
    \draw (D) -- (E) -- (F) -- (G) -- (H) -- (I) -- cycle;
    \draw[ultra thick] (J) -- (K) -- (L) -- cycle;
    \draw (A) -- (D) -- (J);
    \draw (B) -- (F) -- (K);
    \draw (C) -- (H) -- (L);
    \draw[ultra thick] (A) -- (I) -- (L);
    \draw[ultra thick] (B) -- (E) -- (J);
    \draw[ultra thick] (C) -- (G) -- (K);
    \draw (A) -- (E) -- (K);
    \draw (B) -- (G) -- (L);
    \draw (C) -- (I) -- (J);
    \draw (0,0) node {$c_0$};
    \draw (210:4) node {$c_1$};

    \foreach \x in {A,B,C,D,E,F,G,H,I,J,K,L}
      \filldraw (\x) circle (3.5pt);
  \end{tikzpicture}
  \caption{Two spindles (bold) of the same triangulation.}
  \label{fig:spindle:nonunique}
\end{figure}

In order to obtain a 1-1 correspondence, we thus have to refine the
definition of a spindle. To this end, we will first define a
substructure of a triangulation that will be part of our refined
spindles.

\begin{definition}
  A graph is called a \emph{cactus} if it is connected and every two
  cycles in it have at most one vertex in common. It is well known
  that cacti are \emph{outerplanar}, i.e.\ there is an embedding on the
  sphere for which all vertices lie on the boundary of a common face,
  the \emph{outer face}. Every \emph{block} of a cactus---a
  subgraph that cannot be disconnected by deleting a single
  vertex---is a cycle or an edge. If a cactus $G$ has a root vertex,
  this induces a natural order on its set of blocks, similar to a tree
  order: Consider the \emph{block graph} of $G$---the graph whose
  vertices are the vertices separating $G$ and the blocks of $G$ and
  in which a block is adjacent to all separating vertices it contains
  (see Figure~\ref{fig:blockorder}). This block graph is always a tree
  and if we choose its root to be 
  \begin{itemize}
  \item the root of $G$ if it is a separating vertex or otherwise
  \item the block of $G$ containing the root,
  \end{itemize}
  then this induces a tree order on the block graph and hence in
  particular an order on the set of blocks of $G$. In this order
  the blocks that contain the root are the minimal elements.

	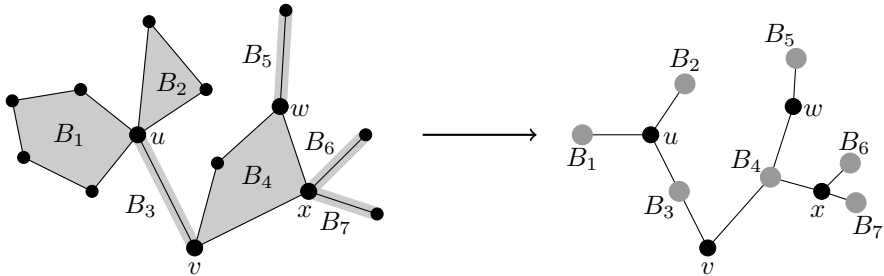
\begin{figure}[htbp]
	  \centering
	  \begin{tikzpicture}[scale=1.5]
	    \coordinate (A) at (0,0);
	    \coordinate (B) at (-.5,1);
	    \coordinate (C) at (.2,.75);
	    \coordinate (D) at (.75,1.25);
	    \coordinate (E) at (1,.5);
	    \coordinate (F) at (-.4,2);
	    \coordinate (G) at (.1,1.4);
	    \coordinate (H) at (-.9,.5);
	    \coordinate (I) at (-1.5,.8);
	    \coordinate (J) at (-1.6,1.3);
	    \coordinate (K) at (-1,1.4);
	    \coordinate (L) at (.8,2.1);
	    \coordinate (M) at (1.5,1);
	    \coordinate (N) at (1.6,.3);
	    
	    \draw[black!20,line width=5pt] (A) -- (B);
	    \filldraw[black!20] (A) -- (C) -- (D) -- (E) -- cycle;
	    \filldraw[black!20] (B) -- (F) -- (G) -- cycle;
	    \filldraw[black!20] (B) -- (H) -- (I) -- (J) -- (K) -- cycle;
	    \draw[black!20,line width=5pt] (D) -- (L);
	    \draw[black!20,line width=5pt] (E) -- (M);
	    \draw[black!20,line width=5pt] (E) -- (N);
	
	    \draw (A) -- (B);
	    \draw (A) -- (C) -- (D) -- (E) -- cycle;
	    \draw (B) -- (F) -- (G) -- cycle;
	    \draw (B) -- (H) -- (I) -- (J) -- (K) -- cycle;
	    \draw (D) -- (L);
	    \draw (E) -- (M);
	    \draw (E) -- (N);
	    
	    \filldraw (A) circle (2pt);
	    \filldraw (B) circle (2pt);
	    \filldraw (C) circle (1.5pt);
	    \filldraw (D) circle (2pt);
	    \filldraw (E) circle (2pt);
	    \filldraw (F) circle (1.5pt);
	    \filldraw (G) circle (1.5pt);
	    \filldraw (H) circle (1.5pt);
	    \filldraw (I) circle (1.5pt);
	    \filldraw (J) circle (1.5pt);
	    \filldraw (K) circle (1.5pt);
	    \filldraw (L) circle (1.5pt);
	    \filldraw (M) circle (1.5pt);
	    \filldraw (N) circle (1.5pt);
	    
	    \draw[thick,->] (2,1) -- (3,1);
	    
	    \coordinate (Center) at (4.5,0);
	    \coordinate (AB) at ($(A)!.5!(B)$);
	    \coordinate (ACDE) at (.55,.625);
	    \coordinate (DL) at ($(D)!.5!(L)$);
	    \coordinate (EM) at ($(E)!.5!(M)$);
	    \coordinate (EN) at ($(E)!.5!(N)$);
	    \coordinate (BFG) at (-.2,1.45);
	    \coordinate (BHIJK) at (-1.1,1);
	    
	    \draw (B) +(5pt,-1pt) node {$u$};
	    \draw (A) +(0,-5pt) node {$v$};
	    \draw (D) +(5pt,-1pt) node {$w$};
	    \draw (E) +(-1pt,-5pt) node {$x$};
	    \draw (BHIJK) node {$B_1$};
	    \draw (BFG) node {$B_2$};
	    \draw (AB) node[anchor=30] {$B_3$};
	    \draw (ACDE) node {$B_4$};
	    \draw (DL) node[anchor=355] {$B_5$};
	    \draw (EM) node[anchor=310] {$B_6$};
	    \draw (EN) node[anchor=70] {$B_7$};
	    
	    \draw (Center) +(A) -- +(AB);
	    \draw (Center) +(A) -- +(ACDE);
	    \draw (Center) +(AB) -- +(B);
	    \draw (Center) +(B) -- +(BFG);
	    \draw (Center) +(B) -- +(BHIJK);
	    \draw (Center) +(ACDE) -- +(D);
	    \draw (Center) +(ACDE) -- +(E);
	    \draw (Center) +(D) -- +(DL);
	    \draw (Center) +(E) -- +(EM);
	    \draw (Center) +(E) -- +(EN);
	    
	    \filldraw (Center) +(A) circle (2pt);
	    \filldraw (Center) +(B) circle (2pt);
	    \filldraw (Center) +(D) circle (2pt);
	    \filldraw (Center) +(E) circle (2pt);
	    \filldraw[black!40] (Center) +(AB) circle (2.5pt);
	    \filldraw[black!40] (Center) +(ACDE) circle (2.5pt);
	    \filldraw[black!40] (Center) +(DL) circle (2.5pt);
	    \filldraw[black!40] (Center) +(EM) circle (2.5pt);
	    \filldraw[black!40] (Center) +(EN) circle (2.5pt);
	    \filldraw[black!40] (Center) +(BFG) circle (2.5pt);
	    \filldraw[black!40] (Center) +(BHIJK) circle (2.5pt);
	    
	    \draw (Center) ++(B) +(5pt,-1pt) node {$u$};
	    \draw (Center) ++(A) +(0,-5pt) node {$v$};
	    \draw (Center) ++(D) +(5pt,-1pt) node {$w$};
	    \draw (Center) ++(E) +(-1pt,-5pt) node {$x$};
	    \draw (Center) ++(BHIJK) +(0,-6pt) node {$B_1$};
	    \draw (Center) ++(BFG) +(0,6pt) node {$B_2$};
	    \draw (Center) ++(AB) +(-5pt,-4pt) node {$B_3$};
	    \draw (Center) ++(ACDE) +(-6pt,4pt) node {$B_4$};
	    \draw (Center) ++(DL) +(-4pt,6pt) node {$B_5$};
	    \draw (Center) ++(EM) +(1pt,5pt) node {$B_6$};
	    \draw (Center) ++(EN) +(3pt,-6pt) node {$B_7$};
	  \end{tikzpicture}
	  \caption{A cactus and its block graph.}
	  \label{fig:blockorder}
	\end{figure}

  Let $k\ge 2$ and let $G$ be an outerplanar subgraph of $T$ for
  which the north pole $c_0$ lies in its outer face. We call $G$ a
  \emph{plane symmetric cactus of order $m=|H|$} if it is a cactus and
  invariant under all elements of the group $H\subseteq\Aut(c_0,T)$. Clearly, the outer
  face of $G$ is invariant under these automorphisms and by
  Lemma~\ref{lem:invariant}, $G$ has another invariant cell. If this
  cell is a cell of $T$, then it is invariant under rotations and
  hence the south pole $c_1$ of $T$, in which case we call $G$
  \emph{antarctic} (see Figure~\ref{fig:cactus:antarctic}). If it is
  not a cell of $T$, then it is a face of $G$ whose boundary is a
  cycle and hence contains an invariant cell of $T$ by the Brouwer
  fixed-point theorem. Again, this cell is $c_1$.
  Either way, we obtain that $G$ has a \emph{unique} invariant cell
  $c$ which is not its outer face. The subgraph of $G$ consisting of
  $c$ (if $c$ is a vertex or an edge) and all vertices and edges on
  the boundary of $c$ is called the \emph{centre} of $G$. The
  maximal connected subgraphs of $G$ that share precisely one vertex
  with the centre are called \emph{branches} of $G$ (see
  Figure~\ref{fig:cactus:general}); the vertex of a branch $B$ that
  lies in the centre of $G$ is called the \emph{base} of $B$. Note that if $G$ is antarctic, then
  it has precisely $1$, $2$, or $3$ branches, depending on whether
  $c_1$ is a vertex, an edge, or a face (see Figure~\ref{fig:cactus:antarctic}).

  \begin{figure}[htbp]
    \centering
    \begin{tikzpicture}[scale=1.5]
      \coordinate (Centre1) at (5,0);
      \coordinate (A) at (0:.75);
      \coordinate (B) at (120:.75);
      \coordinate (C) at (240:.75);
      \coordinate (AD) at (45:.5);
      \coordinate (AE) at (0:.5);
      \coordinate (AF) at (315:.5);
      \coordinate (BD) at (165:.5);
      \coordinate (BE) at (120:.5);
      \coordinate (BF) at (75:.5);
      \coordinate (CD) at (285:.5);
      \coordinate (CE) at (240:.5);
      \coordinate (CF) at (195:.5);
      
      \draw (Centre1) +(A) -- +(B) -- +(C) -- cycle;
      \draw (Centre1) ++(A) -- +(AD);
      \draw (Centre1) ++(A) -- +(AE) -- +(AF) -- cycle;
      \draw (Centre1) ++(B) -- +(BD);
      \draw (Centre1) ++(B) -- +(BE) -- +(BF) -- cycle;
      \draw (Centre1) ++(C) -- +(CD);
      \draw (Centre1) ++(C) -- +(CE) -- +(CF) -- cycle;
      
      \filldraw (Centre1) +(A) circle (1.5pt);
      \filldraw (Centre1) +(B) circle (1.5pt);
      \filldraw (Centre1) +(C) circle (1.5pt);
      \filldraw (Centre1) ++(A) +(AD) circle (1.5pt);
      \filldraw (Centre1) ++(A) +(AE) circle (1.5pt);
      \filldraw (Centre1) ++(A) +(AF) circle (1.5pt);
      \filldraw (Centre1) ++(B) +(BD) circle (1.5pt);
      \filldraw (Centre1) ++(B) +(BE) circle (1.5pt);
      \filldraw (Centre1) ++(B) +(BF) circle (1.5pt);
      \filldraw (Centre1) ++(C) +(CD) circle (1.5pt);
      \filldraw (Centre1) ++(C) +(CE) circle (1.5pt);
      \filldraw (Centre1) ++(C) +(CF) circle (1.5pt);
      
      \draw (Centre1) +(0,0) node {\large $c_1$};
      
      \coordinate (Centre2) at (2.5,0);
      \coordinate (G) at (0,.5);
      \coordinate (H) at (0,-.5);
      \coordinate (GI) at (135:.5);
      \coordinate (GJ) at (90:.5);
      \coordinate (GK) at (45:.5);
      \coordinate (HI) at (315:.5);
      \coordinate (HJ) at (270:.5);
      \coordinate (HK) at (225:.5);
      
      \draw (Centre2) +(G) -- +(H);
      \draw (Centre2) ++(G) -- +(GI);
      \draw (Centre2) ++(G) -- +(GJ) -- +(GK) -- cycle;
      \draw (Centre2) ++(H) -- +(HI);
      \draw (Centre2) ++(H) -- +(HJ) -- +(HK) -- cycle;
      
      \filldraw (Centre2) ++(G) circle (1.5pt);
      \filldraw (Centre2) ++(H) circle (1.5pt);
      \filldraw (Centre2) ++(G) +(GI) circle (1.5pt);
      \filldraw (Centre2) ++(G) +(GJ) circle (1.5pt);
      \filldraw (Centre2) ++(G) +(GK) circle (1.5pt);
      \filldraw (Centre2) ++(H) +(HI) circle (1.5pt);
      \filldraw (Centre2) ++(H) +(HJ) circle (1.5pt);
      \filldraw (Centre2) ++(H) +(HK) circle (1.5pt);
      
      \draw (Centre2) +($(G)!.5!(H)$) node[anchor=east] {\large $c_1$};
      
      \coordinate (Centre3) at (0,0);
      \coordinate (L1) at (105:1);
      \coordinate (M1) at (90:1);
      \coordinate (N1) at (75:1);
      \coordinate (L2) at (195:1);
      \coordinate (M2) at (180:1);
      \coordinate (N2) at (165:1);
      \coordinate (L3) at (285:1);
      \coordinate (M3) at (270:1);
      \coordinate (N3) at (255:1);
      \coordinate (L4) at (15:1);
      \coordinate (M4) at (0:1);
      \coordinate (N4) at (345:1);
      
      \draw (Centre3) -- +(L1);
      \draw (Centre3) -- +(M1) -- +(N1) -- cycle;
      \draw (Centre3) -- +(L2);
      \draw (Centre3) -- +(M2) -- +(N2) -- cycle;
      \draw (Centre3) -- +(L3);
      \draw (Centre3) -- +(M3) -- +(N3) -- cycle;
      \draw (Centre3) -- +(L4);
      \draw (Centre3) -- +(M4) -- +(N4) -- cycle;
      
      \filldraw (Centre3) circle (1.5pt);
      \filldraw (Centre3) +(L1) circle (1.5pt);
      \filldraw (Centre3) +(M1) circle (1.5pt);
      \filldraw (Centre3) +(N1) circle (1.5pt);
      \filldraw (Centre3) +(L2) circle (1.5pt);
      \filldraw (Centre3) +(M2) circle (1.5pt);
      \filldraw (Centre3) +(N2) circle (1.5pt);
      \filldraw (Centre3) +(L3) circle (1.5pt);
      \filldraw (Centre3) +(M3) circle (1.5pt);
      \filldraw (Centre3) +(N3) circle (1.5pt);
      \filldraw (Centre3) +(L4) circle (1.5pt);
      \filldraw (Centre3) +(M4) circle (1.5pt);
      \filldraw (Centre3) +(N4) circle (1.5pt);
      
      \draw (Centre3) node[anchor=315] {\large $c_1$};
    \end{tikzpicture}
    \caption{The three types of antarctic plane symmetric cacti. Note
      that if $c_1$ is a vertex, then the centre of the antarctic
      cactus $C$ is $c_1$; thus, the only branch of $C$ is $C$ itself.}
    \label{fig:cactus:antarctic}
  \end{figure}
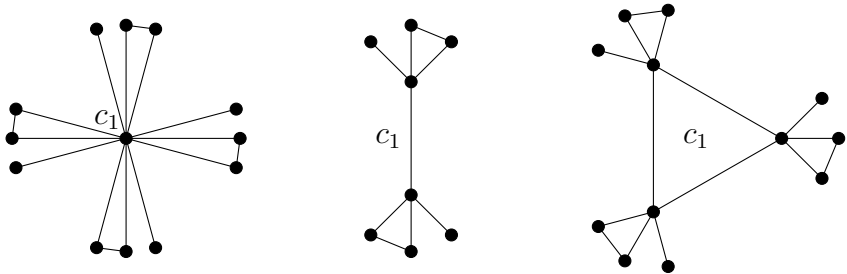

	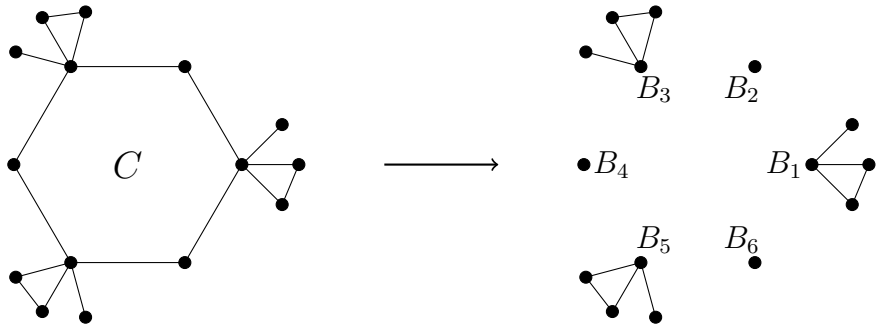
\begin{figure}[htbp]
	  \centering
	  \begin{tikzpicture}[scale=1.5]
	    \coordinate (A) at (0:1);
	    \coordinate (B) at (60:1);
	    \coordinate (C) at (120:1);
	    \coordinate (D) at (180:1);
	    \coordinate (E) at (240:1);
	    \coordinate (F) at (300:1);
	    \coordinate (AG) at (45:.5);
	    \coordinate (AH) at (0:.5);
	    \coordinate (AI) at (315:.5);
	    \coordinate (CG) at (165:.5);
	    \coordinate (CH) at (120:.5);
	    \coordinate (CI) at (75:.5);
	    \coordinate (EG) at (285:.5);
	    \coordinate (EH) at (240:.5);
	    \coordinate (EI) at (195:.5);
	    
	    \draw (A) -- (B) -- (C) -- (D) -- (E) -- (F) -- cycle;
	    \draw (A) -- +(AG);
	    \draw (A) -- +(AH) -- +(AI) -- cycle;
	    \draw (C) -- +(CG);
	    \draw (C) -- +(CH) -- +(CI) -- cycle;
	    \draw (E) -- +(EG);
	    \draw (E) -- +(EH) -- +(EI) -- cycle;
	    
	    \filldraw (A) circle (1.5pt);
	    \filldraw (B) circle (1.5pt);
	    \filldraw (C) circle (1.5pt);
	    \filldraw (D) circle (1.5pt);
	    \filldraw (E) circle (1.5pt);
	    \filldraw (F) circle (1.5pt);
	    \filldraw (A) +(AG) circle (1.5pt);
	    \filldraw (A) +(AH) circle (1.5pt);
	    \filldraw (A) +(AI) circle (1.5pt);
	    \filldraw (C) +(CG) circle (1.5pt);
	    \filldraw (C) +(CH) circle (1.5pt);
	    \filldraw (C) +(CI) circle (1.5pt);
	    \filldraw (E) +(EG) circle (1.5pt);
	    \filldraw (E) +(EH) circle (1.5pt);
	    \filldraw (E) +(EI) circle (1.5pt);
	    
	    \draw (0,0) node {\Large $C$};
	    
	    \draw[thick,->] (2.25,0) -- (3.25,0);
	    
	    \coordinate (Center) at (5,0);
	    
	    \draw (Center) ++(A) -- +(AG);
	    \draw (Center) ++(A) -- +(AH) -- +(AI) -- cycle;
	    \draw (Center) ++(C) -- +(CG);
	    \draw (Center) ++(C) -- +(CH) -- +(CI) -- cycle;
	    \draw (Center) ++(E) -- +(EG);
	    \draw (Center) ++(E) -- +(EH) -- +(EI) -- cycle;
	    
	    \filldraw (Center) ++(A) node[anchor=0] {\large $B_1$} circle (1.5pt);
	    \filldraw (Center) ++(B) node[anchor=60] {\large $B_2$} circle (1.5pt);
	    \filldraw (Center) ++(C) node[anchor=120] {\large $B_3$} circle (1.5pt);
	    \filldraw (Center) ++(D) node[anchor=180] {\large $B_4$} circle (1.5pt);
	    \filldraw (Center) ++(E) node[anchor=240] {\large $B_5$} circle (1.5pt);
	    \filldraw (Center) ++(F) node[anchor=300] {\large $B_6$} circle (1.5pt);
	    \filldraw (Center) ++(A) +(AG) circle (1.5pt);
	    \filldraw (Center) ++(A) +(AH) circle (1.5pt);
	    \filldraw (Center) ++(A) +(AI) circle (1.5pt);
	    \filldraw (Center) ++(C) +(CG) circle (1.5pt);
	    \filldraw (Center) ++(C) +(CH) circle (1.5pt);
	    \filldraw (Center) ++(C) +(CI) circle (1.5pt);
	    \filldraw (Center) ++(E) +(EG) circle (1.5pt);
	    \filldraw (Center) ++(E) +(EH) circle (1.5pt);
	    \filldraw (Center) ++(E) +(EI) circle (1.5pt);
	  \end{tikzpicture}
	  \caption{A plane symmetric cactus of order $3$ with centre $C$
	    and branches $B_1,\dotsc,B_6$.}
	  \label{fig:cactus:general}
	\end{figure}

  If $G$ is not antarctic and in addition the boundary of the south pole $c_1$ of $T$
  meets the boundary of the centre of $G$, then the south pole has to
  be a face or an edge and by symmetry all vertices on its boundary
  lie in the centre of $G$. In this case, we call $G$
  \emph{pseudo-antarctic} (see Figure~\ref{fig:cactus:pseudoantarctic}).

	\begin{figure}[htbp]
	  \centering
	  \begin{tikzpicture}[scale=1.5]
	    \coordinate (A) at (0:1);
	    \coordinate (B) at (60:1);
	    \coordinate (C) at (120:1);
	    \coordinate (D) at (180:1);
	    \coordinate (E) at (240:1);
	    \coordinate (F) at (300:1);
	    \coordinate (AG) at (45:.5);
	    \coordinate (AH) at (0:.5);
	    \coordinate (AI) at (315:.5);
	    \coordinate (CG) at (165:.5);
	    \coordinate (CH) at (120:.5);
	    \coordinate (CI) at (75:.5);
	    \coordinate (EG) at (285:.5);
	    \coordinate (EH) at (240:.5);
	    \coordinate (EI) at (195:.5);
	    
	    \draw[dashed] (B) -- (D) -- (F) -- cycle;
	    
	    \draw (A) -- (B) -- (C) -- (D) -- (E) -- (F) -- cycle;
	    \draw (A) -- +(AG);
	    \draw (A) -- +(AH) -- +(AI) -- cycle;
	    \draw (C) -- +(CG);
	    \draw (C) -- +(CH) -- +(CI) -- cycle;
	    \draw (E) -- +(EG);
	    \draw (E) -- +(EH) -- +(EI) -- cycle;
	    
	    \filldraw (A) circle (1.5pt);
	    \filldraw (B) circle (1.5pt);
	    \filldraw (C) circle (1.5pt);
	    \filldraw (D) circle (1.5pt);
	    \filldraw (E) circle (1.5pt);
	    \filldraw (F) circle (1.5pt);
	    \filldraw (A) +(AG) circle (1.5pt);
	    \filldraw (A) +(AH) circle (1.5pt);
	    \filldraw (A) +(AI) circle (1.5pt);
	    \filldraw (C) +(CG) circle (1.5pt);
	    \filldraw (C) +(CH) circle (1.5pt);
	    \filldraw (C) +(CI) circle (1.5pt);
	    \filldraw (E) +(EG) circle (1.5pt);
	    \filldraw (E) +(EH) circle (1.5pt);
	    \filldraw (E) +(EI) circle (1.5pt);
	    
	    \draw (0,0) node {\large $c_1$};
	    
	    \coordinate (Center) at (4,0);
	    
	    \draw[dashed] (Center) +(A) -- node[anchor=south] {\large $c_1$} +(D);
	    
	    \draw (Center) +(A) -- +(B) -- +(C) -- +(D) -- +(E) -- +(F) -- cycle;
	    \draw (Center) ++(A) -- +(AG);
	    \draw (Center) ++(A) -- +(AH) -- +(AI) -- cycle;
	    \draw (Center) ++(C) -- +(CG);
	    \draw (Center) ++(C) -- +(CH) -- +(CI) -- cycle;
	    \draw (Center) ++(E) -- +(EG);
	    \draw (Center) ++(E) -- +(EH) -- +(EI) -- cycle;
	    
	    \filldraw (Center) ++(A) circle (1.5pt);
	    \filldraw (Center) ++(B) circle (1.5pt);
	    \filldraw (Center) ++(C) circle (1.5pt);
	    \filldraw (Center) ++(D) circle (1.5pt);
	    \filldraw (Center) ++(E) circle (1.5pt);
	    \filldraw (Center) ++(F) circle (1.5pt);
	    \filldraw (Center) ++(A) +(AG) circle (1.5pt);
	    \filldraw (Center) ++(A) +(AH) circle (1.5pt);
	    \filldraw (Center) ++(A) +(AI) circle (1.5pt);
	    \filldraw (Center) ++(C) +(CG) circle (1.5pt);
	    \filldraw (Center) ++(C) +(CH) circle (1.5pt);
	    \filldraw (Center) ++(C) +(CI) circle (1.5pt);
	    \filldraw (Center) ++(E) +(EG) circle (1.5pt);
	    \filldraw (Center) ++(E) +(EH) circle (1.5pt);
	    \filldraw (Center) ++(E) +(EI) circle (1.5pt);
	  \end{tikzpicture}
	  \caption{The two possibilities for a pseudo-antarctic plane
	    symmetric cactus.}
	  \label{fig:cactus:pseudoantarctic}
	\end{figure}
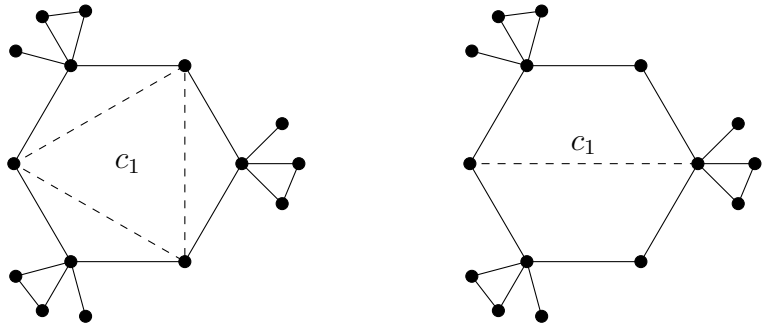
\end{definition}

Note that the above definition allows the case that the branches of a plane
symmetric cactus are just the vertices of its centre, in particular
every invariant cycle is a plane symmetric cactus. Furthermore, a
plane symmetric cactus of order $k$ is also a plane symmetric cactus
of order $\ell$ for every divisor $\ell\ge2$ of $k$.

Plane symmetric cacti appear in a natural way when we move from the
north pole towards the south pole of the triangulation:

\begin{lemma}\label{lem:cactus}
  Let $C$ be a cycle in $T$ that is invariant under $\varphi$ (and
  thus a plane symmetric cactus of order $m$). Suppose that $C$ is
  neither antarctic nor pseudo-antarctic and let $f$ be the face of
  $C$ that contains the south pole. Denote by $\mathcal{F}$ the set
  of all faces of $T$ that are contained in $f$ and whose boundaries
  meet $C$. Let $F$ be the subgraph of $T$ consisting of all vertices
  and edges that lie on the boundary of a face $f'\in\mathcal{F}$ but
  do not lie in $C$ or have an incident vertex in $C$. Then $F$ has a
  component that is a plane symmetric cactus of order $m$.
\end{lemma}

\begin{proof}
  By construction, $F$ is outerplanar and all its edges lie on the
  boundary of its outer face. Thus, no two of its cycles can meet in
  more than one vertex, showing that all components of $F$ are cacti.
  The south pole $c_1$ is not contained in the outer face of $F$ by
  construction, therefore there is a component $F_1$ of $F$ such that
  either
  \begin{itemize}
  \item
    $c_1$ is contained in $F_1$ or
  \item
    $c_1$ is contained in a face of $F_1$ that is not its outer face.
  \end{itemize}
  In either case, $F_1$ is invariant under $\varphi$ (and hence under
  all elements of $H$) and thus a plane symmetric cactus of order $m$.
\end{proof}

Repeated application of Lemma~\ref{lem:cactus} gives rise to a finite
sequence $F_0,\dotsc,F_k$ of plane symmetric cacti in $T$ as follows.
We start by letting $F_0$ be the invariant cycle ``closest'' to $c_0$
like in Figure~\ref{fig:invariantcycles}: if $c_0$ is a face, let
$F_0$ be its boundary. If $c_0$ is an edge, let $F_0$ consist of all
vertices and edges, apart from $c_0$ itself, that lie on the boundary
of a face incident with $c_0$. Finally, if $c_0$ is a vertex, let
$F_0$ consist of all vertices adjacent to $c_0$ and all edges that lie
opposite to $c_0$ at some face incident with $c_0$. Note that in
either case, $F_0$ is a cycle whose length is a multiple of $m$.

If $F_0$ is antarctic or pseudo-antarctic, the sequence ends with
$k=0$; otherwise, by applying Lemma~\ref{lem:cactus} with $C=F_0$, we
obtain a plane symmetric cactus $F_1$ of order $m$. If $F_1$ is
antarctic or pseudo-antarctic, we stop; otherwise, we apply
Lemma~\ref{lem:cactus} with $C$ being the centre of
$F_1$ to obtain another plane symmetric cactus $F_2$ of order $m$.
We continue this way until we obtain an antarctic or pseudo-antarctic
plane symmetric cactus $F_k$. We call the graphs $F_0,\dotsc,F_k$ the
\emph{levels} of $T$ (see Figure~\ref{fig:spindle:levels}) and denote
their centres by $C_0,\dotsc,C_k$.

\begin{figure}[htbp]
  \captionsetup{singlelinecheck=off}
  \centering
  \begin{tikzpicture}
    \begin{scope}[scale=.7,yshift=1.5cm]
    \coordinate (A) at (210:.8);
    \coordinate (B) at (330:.8);
    \coordinate (C) at (90:.8);
    \coordinate (D) at (250:2.5);
    \coordinate (E) at (290:2.5);
    \coordinate (F) at (10:2.5);
    \coordinate (G) at (50:2.5);
    \coordinate (H) at (130:2.5);
    \coordinate (I) at (170:2.5);
    \coordinate (J) at (280:1);
    \coordinate (K) at (296.5:1.12);
    \coordinate (L) at (290:1.8);
    \coordinate (M) at (40:1);
    \coordinate (N) at (56.5:1.12);
    \coordinate (O) at (50:1.8);
    \coordinate (P) at (160:1);
    \coordinate (Q) at (176.5:1.12);
    \coordinate (R) at (170:1.8);
    \coordinate (S) at (270:6);
    \coordinate (T) at (30:6);
    \coordinate (U) at (150:6);
    \coordinate (V) at (270:4);
    \coordinate (W) at (30:4);
    \coordinate (X) at (150:4);

    \filldraw[black!20] (A) -- (B) -- (C) -- cycle;
    \draw[ultra thick] (A) -- (B) -- (C) -- cycle;
    \draw[ultra thick] (D) -- (E) -- (F) -- (G) -- (H) -- (I) -- cycle;
    \draw[ultra thick] (S) -- (T) -- (U) -- cycle;
    \draw (A) -- (J);
    \draw (A) -- (L);
    \draw (B) -- (J);
    \draw (B) -- (K);
    \draw (B) -- (L);
    \draw[ultra thick] (J) -- (K) -- (L) -- cycle;
    \draw[ultra thick] (L) -- (E);
    \draw (B) -- (M);
    \draw (B) -- (O);
    \draw (C) -- (M);
    \draw (C) -- (N);
    \draw (C) -- (O);
    \draw[ultra thick] (M) -- (N) -- (O) -- cycle;
    \draw[ultra thick] (O) -- (G);
    \draw (C) -- (P);
    \draw (C) -- (R);
    \draw (A) -- (P);
    \draw (A) -- (Q);
    \draw (A) -- (R);
    \draw[ultra thick] (P) -- (Q) -- (R) -- cycle;
    \draw[ultra thick] (R) -- (I);
    \draw (A) -- (D) -- (S);
    \draw (B) -- (F) -- (T);
    \draw (C) -- (H) -- (U);
    \draw (A) -- (I) -- (U);
    \draw (B) -- (E) -- (S);
    \draw (C) -- (G) -- (T);
    \draw (A) -- (E) -- (T);
    \draw (B) -- (G) -- (U);
    \draw (C) -- (I) -- (S);
    \draw (D) -- (V) -- (E);
    \draw[ultra thick] (V) -- (S);
    \draw (F) -- (W) -- (G);
    \draw[ultra thick] (W) -- (T);
    \draw (H) -- (X) -- (I);
    \draw[ultra thick] (X) -- (U);
    \draw (0,0) node {$c_0$};
    \draw (210:4) node {$c_1$};
    \draw (330:1.2) node {{\boldmath$F_0$}};
    \draw (30:2.7) node {{\boldmath$F_1$}};
    \draw (90:3.3) node {{\boldmath$F_2$}};
    \draw[black!0] (0,-7.5) circle (1pt);
    \draw (0,-6.75) node {(i)};
    \end{scope}
    \begin{scope}[scale=.4,xshift=13cm]
    \coordinate (A) at (0,0);
    \coordinate (B) at (0:3);
    \coordinate (C) at (45:1);
    \coordinate (D) at (90:5);
    \coordinate (E) at (135:1);
    \coordinate (F) at (180:3);
    \coordinate (G) at (225:1);
    \coordinate (H) at (270:5);
    \coordinate (I) at (315:1);
    \coordinate (J) at (45:2);
    \coordinate (K) at (225:2);
    \coordinate (W) at (0:5);
    \coordinate (X) at (90:8);
    \coordinate (Y) at (180:5);
    \coordinate (Z) at (270:8);

    \draw (A) -- (B);
    \draw (A) -- (C);
    \draw (A) -- (D);
    \draw (A) -- (E);
    \draw (A) -- (F);
    \draw (A) -- (G);
    \draw (A) -- (H);
    \draw (A) -- (I);
    \draw[ultra thick] (B) -- (C) -- (D);
    \draw[ultra thick] (D) -- (E) -- (F);
    \draw[ultra thick] (F) -- (G) -- (H);
    \draw[ultra thick] (H) -- (I) -- (B);
    \draw (B) -- (J);
    \draw (C) -- (J);
    \draw (D) -- (J);
    \draw (F) -- (K);
    \draw (G) -- (K);
    \draw (H) -- (K);
    \draw (B) -- (D) -- (F) -- (H) -- cycle;
    \draw (B) -- (W);
    \draw (D) -- (X);
    \draw (F) -- (Y);
    \draw (H) -- (Z);
    \draw (D) -- (W);
    \draw (D) -- (Y);
    \draw (H) -- (Y);
    \draw (H) -- (W);
    \draw[ultra thick] (W) -- (X) -- (Y) -- (Z) -- cycle;
    \draw (Z) arc (-90:90:8);
    \filldraw (A) circle (7.5pt);
    \filldraw[black!20] (A) circle (5pt);
    \draw[thick,black!0] (.3,0) -- (1.2,0);
    \draw (.75,0) node {$c_0$};
    \draw (7.25,0) node {$c_1$};
    \draw (135:1.75) node {{\boldmath$F_0$}};
    \draw (120:5.5) node {{\boldmath$F_1$}};
    \draw (0,-9.25) node {(ii)};
    \end{scope}
  \end{tikzpicture}
  \caption[levels]{Two triangulations and their levels.
    \begin{enumerate}
    \item A triangulation with three levels $F_0,F_1,F_2$ (bold), each of
      which is a plane symmetric cactus of order $3$. The last level
      $F_2$ is antarctic.
    \item A triangulation with two levels $F_0,F_1$ (bold), both plane
      symmetric cacti of order $2$. The last level $F_1$ is
      pseudo-antarctic.
    \end{enumerate}}
  \label{fig:spindle:levels}
\end{figure}
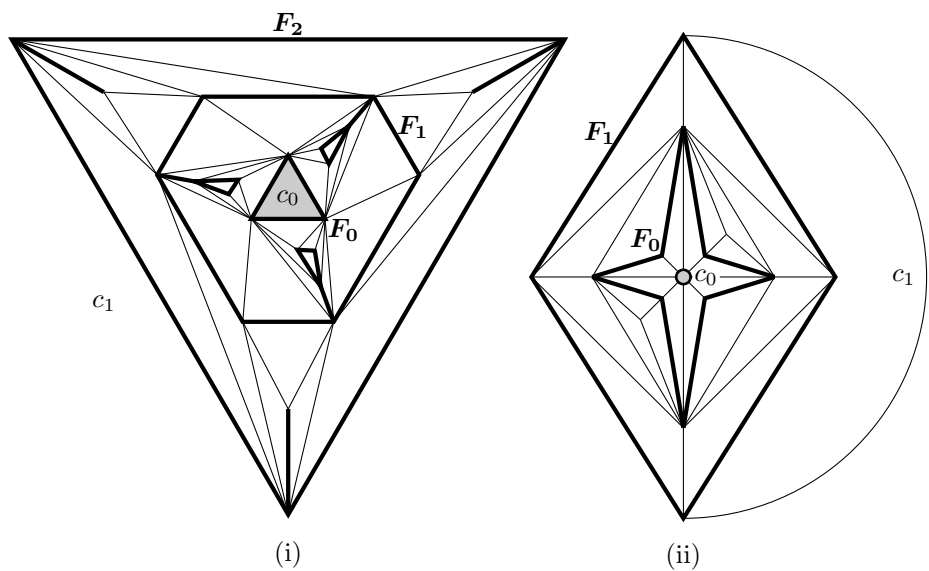

The idea behind our refined version of a spindle will be as follows:
for a constructive decomposition, we shall need a \emph{unique}
substructure of $T$; something that the spindle was not able to
provide, since the path $P$ was chosen arbitrarily. Instead of
connecting the north pole and the south pole by paths, we will base
our construction on the levels of $T$ and connect them by edges.
Those edges have to be chosen in a unique way, which we will
guarantee by always picking the `leftmost' edge from a given vertex
to the next level---a construction that will be made precise shortly.
Moreover, it will not always be enough to have $m$ edges from each
level to the next. Indeed, if the north pole $c_0$ is a vertex, then
its degree might be a multiple of $m$ and there is no criterion which
of the $d(c_0)$ edges we should choose. We thus have to start with all
these edges.

The starting point of our construction will be vertices
$u_0,\dotsc,u_{am-1}$ on $F_0=C_0$ (precise construction follows in
Construction~\ref{constr:liaison}). We would then like to choose an
edge from each $u_j$ to the level $F_1$. However, not every vertex
$u_j$ necessarily has a neighbour in $F_1$. We will thus walk along
the cycle $C_0$ in clockwise direction from each $u_j$ until we find a
vertex $v_j$ that has a neighbour in $F_1$. In order to decide which
edge from $v_j$ to $F_1$ we will pick, let $e$ be one of the two edges
of $C_0$ at $v_j$ and let $e_j=v_jw_j$ be the first edge in clockwise
direction around $v_j$, starting at $e$, with $w_j\in F_1$. Note that
this definition does not depend on which edge of $C_0$ we choose as
$e$. We call $e_j$ the \emph{leftmost edge} from $v_j$ to $F_1$ and
$w_j$ the \emph{leftmost neighbour} of $v_j$ in $F_1$. We then
continue the construction in $F_1$ by first going to the base of the
branch that contains $w_j$, then walk along the cycle $C_1$ until we
find a vertex that has a neighbour in $F_2$ and so on. We will now
make this construction precise.

\begin{construction}[\liaisonp, sources, targets]\label{constr:liaison}
  We begin our construction by choosing vertices $u_0,\dotsc,u_{am-1}$
  on $F_0=C_0$ as follows (see also Figure~\ref{fig:spindle:start}):
  if $c_0$ is a vertex, let $a:=d(c_0)/m$ and let
  $u_0,\dotsc,u_{am-1}$ be all vertices of $F_0$, where the
  enumeration is in clockwise direction around the north pole. If
  $c_0$ is an edge, let $a:=1$ and let $u_0$ and $u_1$ be the vertices
  of $F_0$ that are not end vertices of $c_0$. Finally, if $c_0$ is a
  face, let $a:=1$ and let $u_0,u_1,u_2$ be the vertices on its
  boundary in clockwise direction. Note that by the choice of $u_0,\dotsc,u_{am-1}$, we have
  $\varphi(u_j) = u_{j+a \pmod{am}}$ for every $j$. With a slight
  abuse of notation, we will omit the modulo term in the index and
  simply write $u_i$ instead of $u_{i\pmod{am}}$. We will use this
  notation also for all other cyclic sequences of vertices throughout
  this section.

  \begin{figure}[htbp]
    \centering
      \hfill
      \begin{tikzpicture}
        \def\u{1.732}
  
        \coordinate (A) at (0,0);
        \coordinate (B) at (45:\u);
        \coordinate (C) at (135:\u);
        \coordinate (D) at (225:\u);
        \coordinate (E) at (315:\u);
        \foreach \x in {B,C,D,E}
        {
          \draw (A) -- (\x);
        }
        \draw[ultra thick] (B) -- (C) -- (D) -- (E) -- cycle;
        \foreach \x in {B,C,D,E}
        {
          \filldraw (\x) circle (2pt);
        }
        \filldraw (A) circle (3.25pt);
        \filldraw[black!20] (A) circle (2.5pt);
        \draw (.05,-.05) node[anchor=west] {$c_0$};
        \draw (B) node[anchor=225] {$u_0$};
        \draw (E) node[anchor=135] {$u_1$};
        \draw (D) node[anchor=45] {$u_2$};
        \draw (C) node[anchor=315] {$u_3$};
        \draw[white] (0,1.5) node[anchor=south] {$u_0$};
        \draw[white] (0,-1.5) node[anchor=north] {$u_1$};
        \draw (0,-2.25) node {(i)};
      \end{tikzpicture}
      \hfill
      \begin{tikzpicture}
        \coordinate (A) at (90:1);
        \coordinate (B) at (210:1);
        \coordinate (C) at (330:1);
        \coordinate (D) at (90:.5);
        \coordinate (E) at (270:.5);
        \coordinate (F) at (150:1);
        \coordinate (G) at (270:1);
        \coordinate (H) at (30:1);
        \draw[line width=3.5pt] (D) +(B) -- +(C);
        \draw[black!20,line width=2pt] (D) +(B) -- +(C);
        \draw[ultra thick] (D) +(C) -- +(A) -- +(B);
        \draw[ultra thick] (E) +(F) -- +(G) -- +(H);
        \foreach \x in {A,B,C}
        {
          \filldraw (D) +(\x) circle (2pt);
        }
        \filldraw (E) +(G) circle (2pt);
        \draw (0,0) node[anchor=south] {$c_0$};
        \draw (D) +(A) node[anchor=south] {$u_0$};
        \draw (E) +(G) node[anchor=north] {$u_1$};
        \draw (0,-2.25) node {(ii)};
      \end{tikzpicture}
      \hfill
      \begin{tikzpicture}
        \draw[white] (-1,-1.5) -- (1,1.5);
        \coordinate (A) at (90:1);
        \coordinate (B) at (210:1);
        \coordinate (C) at (330:1);
        \filldraw[black!20] (A) -- (B) -- (C) -- cycle;
        \draw[ultra thick] (A) -- (B) -- (C) -- cycle;
        \foreach \x in {A,B,C}
        {
          \filldraw (\x) circle (2pt);
        }
        \draw (0,0) node {$c_0$};
        \draw (A) node[anchor=270] {$u_0$};
        \draw (C) node[anchor=150] {$u_1$};
        \draw (B) node[anchor=30] {$u_2$};
        \draw[white] (0,1.75) node[anchor=south] {$u_0$};
        \draw[white] (0,-1.25) node[anchor=north] {$u_1$};
        \draw (0,-2.25) node {(iii)};
      \end{tikzpicture}
      \hfill{}
    \caption{The vertices $u_0,\dotsc,u_{am-1}$ for the north pole $c_0$
      being (i)~a vertex, (ii)~an edge, (iii)~a face. Note that in
      Case~(i), we can either have $m=4$, $a=1$ or $m=a=2$.}
    \label{fig:spindle:start}
  \end{figure}
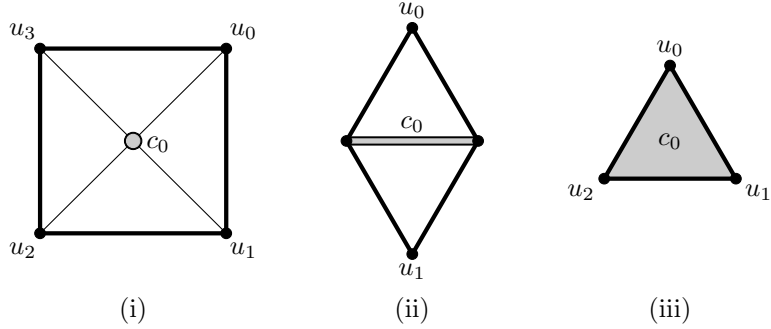

  For each $j=0,\dotsc,am-1$, we define the vertices
  $u_j^0:=u_j,u_j^1,\dotsc,u_j^k$, $v_j^0,\dotsc,v_j^{k-1}$, and
  $w_j^1,\dotsc,w_j^k$ as follows (see
  Figure~\ref{fig:spindle:sourcestargetsbases}): recursively for
  $0\le i\le k-1$
  \begin{enumerate}
  \item\label{spindle:source}
    let $v_j^i$ be the first vertex starting from $u_j^i$ along the
    cycle $C_i$ in clockwise direction around the north pole that has an
    edge to $F_{i+1}$;
  \item\label{spindle:target}
    let $v_j^iw_j^{i+1}$ be the leftmost neighbour of $v_j^i$ in
    $F_{i+1}$; and
  \item\label{spindle:base}
    let $u_j^{i+1}$ be the base of the branch of $F_{i+1}$ that contains
    $w_j^{i+1}$.
  \end{enumerate}
  The vertices $u_j^0,\dotsc,u_j^k$, $v_j^0,\dotsc,v_j^{k-1}$, and
  $w_j^1,\dotsc,w_j^k$ are uniquely defined by
  \ref{spindle:source}-\ref{spindle:base}. We have $\varphi(u_j^i) =
  u_{j+a}^i$, $\varphi(v_j^i) = v_{j+a}^i$, and $\varphi(w_j^i) =
  w_{j+a}^i$ for all $i,j$ by the symmetry of $T$ and the fact that
  $\varphi(u_j^0) = u_{j+a}^0$. Note that
  in~\ref{spindle:source}, we encounter $v_j^i$ before we reach
  $u_{j+a}^i$: indeed, if the subpath of $C_i$ from $u_j^i$ to
  $u_{j+a}^i$ contains no vertex that has a neighbour in $F_{i+1}$,
  then by the fact that $\varphi(u_j^i)=u_{j+a}^i$, no vertex of $C_i$
  has a neighbour in $F_{i+1}$, a contradiction to the definition of
  $F_{i+1}$.
  
  \begin{figure}[htbp]
    \centering
    \begin{tikzpicture}
      \coordinate (UI) at (100:10);
      \coordinate (UIa) at (75:10);
      \coordinate (VI) at (90:10);
      \coordinate (WIp) at (95:11.5);
      \coordinate (UIp) at (95:14);
      
      \coordinate (AI) at (107:9);
      \coordinate (BI) at (103:9);
      \coordinate (CI) at (82:9);
      \coordinate (DI) at (78:9);
      
      \coordinate (AIp) at (98:12.75);
      \coordinate (BIp) at (94:12.75);
      \coordinate (CIp) at (90:12.75);
      \coordinate (DIp) at (91:11.5);
  
      \draw[black!30,line width=5pt] (VI) -- (WIp);
      
      \draw[ultra thick] (72:10) arc (72:108:10);
      \draw[ultra thick] (72:14) arc (72:108:14);
      \draw[ultra thick] (AI) -- (BI) -- (105:10) -- cycle;
      \draw[ultra thick] (CI) -- (DI) -- (80:10) -- cycle;
      \draw[ultra thick] (AIp) -- (UIp);
      \draw[ultra thick] (BIp) -- (CIp) -- (UIp) -- cycle;
      \draw[ultra thick] (WIp) -- (DIp) -- (BIp) -- cycle;
  
      \draw (VI) to[out=150,in=40] (UI);    
      \draw (VI) to[out=135,in=40] (105:10);    
      \draw[thick] (VI) -- (WIp);
      \draw (VI) -- (DIp);
      \draw (VI) to[out=70,in=330] (BIp);
      
      \filldraw (AI) circle (2pt);
      \filldraw (BI) circle (2pt);
      \filldraw (CI) circle (2pt);
      \filldraw (DI) circle (2pt);
      \filldraw (AIp) circle (2pt);
      \filldraw (BIp) circle (2pt);
      \filldraw (CIp) circle (2pt);
      \filldraw (DIp) circle (2pt);
      \foreach \x in {75,80,...,105}
      {
        \filldraw (\x:10) circle (2pt);
        \filldraw (\x:14) circle (2pt);
      }
      \filldraw (UI) circle (3.5pt);
      \filldraw[black!30] (UI) circle (2.5pt);
      \filldraw (VI) circle (3.5pt);
      \filldraw[black!30] (VI) circle (2.5pt);
      \filldraw (WIp) circle (3.5pt);
      \filldraw[black!30] (WIp) circle (2.5pt);
      \filldraw (UIp) circle (3.5pt);
      \filldraw[black!30] (UIp) circle (2.5pt);
      
      \draw (110:10) node {\large $\mathbf{F_i}$};
      \draw (110:14) node {\large $\mathbf{F_{i+1}}$};
      \draw (UI) node[anchor=north] {$u_j^i$};
      \draw (UIa) node[anchor=north] {$u_{j+a}^i$};
      \draw (VI) node[anchor=north] {$v_j^i$};
      \draw (WIp) node[anchor=east] {$w_j^{i+1}$};
      \draw (UIp) node[anchor=south] {$u_j^{i+1}$};
      \draw (95:10) node[anchor=north] {$x$};
    \end{tikzpicture}
    \caption{Constructing the sources $v_j^i$, targets $w_j^{i+1}$, and
      bases $u_j^i$. Note that if $x$ is a base, say $x=u_{j'}^i$, then
      the construction yields $v_j^i=v_{j'}^i$.}
    \label{fig:spindle:sourcestargetsbases}
  \end{figure}
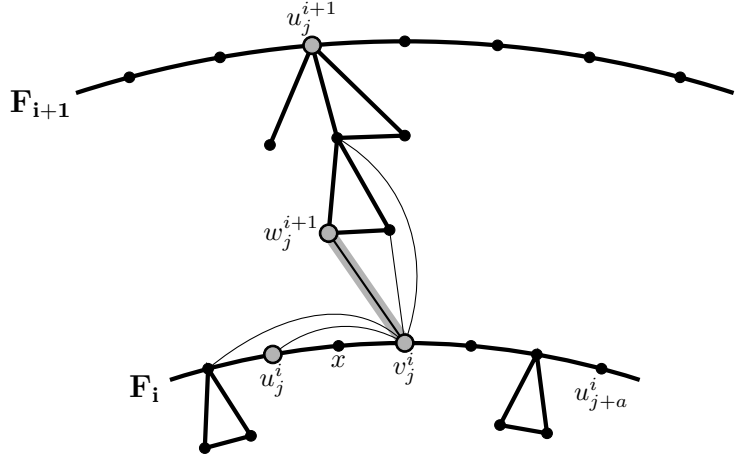

  The edges
  $v_j^iw_j^{i+1}$ are called \emph{\liaisonp}. For every
  \liaisons, we call $v_j^i$ its \emph{source} and $w_j^{i+1}$
  its \emph{target}. Note that sources, targets, and
  bases do not have to be distinct. Clearly, two targets that lie in
  the same branch will always result in the same base, but also two
  bases will result in the same source if there is no eligible choice
  for a source between them on the cycle, and two sources may result
  in the same target if their leftmost edges lead to the same vertex.
  It is important to note that the sources $v_j^i,v_{j+a}^i,\dotsc,
  v_{j+a(m-1)}^i$ are always distinct since they form an orbit under
  $\varphi$ by the symmetry of the construction. The same holds for
  targets and bases up to the ($k-1$)-st level.
\end{construction}

With the levels $F_0,\dotsc,F_k$ and the \liaisonp\ $v_j^iw_j^{i+1}$,
we are now able to define our refined spindles, called \emph{fyke
nets}.

\begin{definition}[Fyke net]\label{def:extendedspindle}
  Let $\tilde F$ be the union of
  \begin{itemize}
  \item the levels $F_0,\dotsc,F_k$ of $T$,
  \item all \liaisonp\ $v_j^iw_j^{i+1}$,
  \item the north pole $c_0$ of $T$,
  \item all edges from $c_0$ to $F_0$ (if $c_0$ is a vertex), and
  \item the south pole $c_1$ and its boundary (if the last level $F_k$
    is pseudo-antarctic).
  \end{itemize}
  The \emph{fyke net} of $T$ with respect to the group
  $H\subseteq\Aut(c_0,T)$ is the maximal $2$-connected subgraph $F$
  of $\tilde F$ that contains both poles $c_0,c_1$.

  \begin{figure}[htbp]
    \centering
    \begin{tikzpicture}[scale=.75]
      \clip (-4.1,-4) rectangle (3.7,3.8);
      \node (myfirstpic) at (0,0) {\includegraphics[height=6cm]{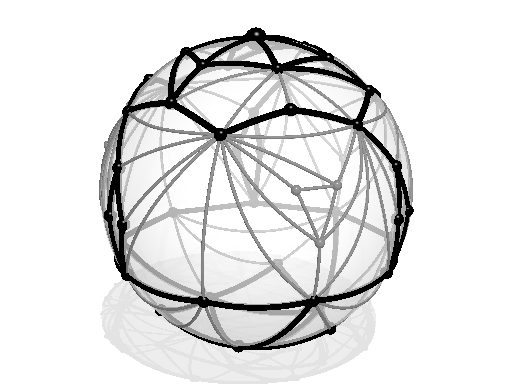}};
      \draw (0,3.6) node {\large $c_0$};
      \draw (-.05,-2.85) node {\large $c_1$};
    \end{tikzpicture}
    \caption{A triangulation its fyke net (bold) with respect to the
      group $H=\Aut(c_0,T)$, in which
      the north pole $c_0$ is a vertex and the south pole $c_1$ is a
      face.}
    \label{fig:fykenet}
  \end{figure}
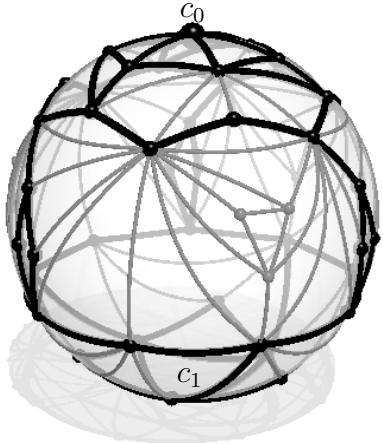

  The intersection of the fyke net with the level $F_i$ is
  its \emph{$i$th layer} and denoted by $H_i$. Note that every layer
  is a plane symmetric cactus of order $m$ by the symmetry of the
  construction. The fyke net has up to five different types of faces:
  \begin{enumerate}
  \item faces at the north pole $c_0$: either $c_0$ itself (if it is
    a face) or all faces of $T$ that are incident with $c_0$;
  \item faces that are bounded by cycles in a branch of a layer; we
    call such faces \emph{leaves};
  \item faces bounded by two consecutive \liaisonp\ and two
    subpaths of the two layers connecting their sources and their
    targets; we call such faces \emph{segments};
  \item if the last layer $H_k$ is pseudo-antarctic, $m$ faces that
    are bounded by a subpath of the centre of $H_k$ and the south pole
    $c_1$ (if it is an edge) or one of its incident edges (if it is a
    face); we call such faces \emph{pseudo-antarctic};
  \item the south pole $c_1$ (if it is a face).
  \end{enumerate}
\end{definition}

The following properties of the fyke net are easy to show, using
the 2-connectedness of the fyke net and the structure of the
levels of $T$.

\begin{proposition}\label{prop:spindle}
  Let $F$ be the fyke net of a triangulation $T$. Then every segment
  of $F$ is bounded by a cycle. An edge $e$ of $T$ lies in $F$ if and
  only if
  \begin{enumerate}
  \item $e$ is a \liaisons,
  \item $e$ lies in the centre of a level $F_i$, or
  \item $e$ and a target $w_j^i$ are contained in the same branch $B$
    of a level $F_i$ and $e$ lies on a path from $w_j^i$ to $u_j^i$ in
    $B$. Equivalently, the block $B(e)$ of $B$ containing $e$ and the
    smallest block (in the tree order on the block graph induced by
    choosing the base $u_j^i$ of $B$ as its root) $B(w_j^i)$
    containing $w_j^i$ satisfy $B(w_j^i)\ge B(e)$ (in said tree order).
  \end{enumerate}
\end{proposition}

Note that the orbits of $\varphi$ partition the sets of leaves and segments
into sets of size $m$. In particular, the near-triangulations that
have to be inserted into the leaves and segments in order to re-obtain
(i.e.\ construct) $T$ are isomorphic if the corresponding faces of $F$ are in
the same orbit.

Unlike spindles, the fyke net is \emph{unique} and thus, we can
obtain all triangulations with rotative symmetry by first choosing
a fyke net and then the near-triangulations that are to be
pasted into the leaves and segments.

As in the case of reflective symmetries, we have to be more specific
on which near-triangulations we are allowed to paste into leaves and
segments.

\begin{lemma}\label{lem:spindle:segments}
  Let $f$ be a segment of the fyke net of $T$. Then there exist
  \emph{unique} indices $i,j$ satisfying the following properties
  (see Figure~\ref{fig:spindle:segment}).
  \begin{enumerate}
  \item\label{spindle:segment:boundary}
    The boundary $C$ of $f$ consists of two \liaisonp\
    $v_j^iw_j^{i+1}$, $v_{j+1}^iw_{j+1}^{i+1}$ and paths $P_i$,
    $P_{i+1}$, where $P_i$ is a path in the centre of the layer $H_i$ from
    $v_j^i$ to $v_{j+1}^i$ and $P_{i+1}$ is a path in the layer $H_{i+1}$ from
    $w_j^{i+1}$ to $w_{j+1}^{i+1}$.
  \item\label{spindle:segment:paths}
    The path $P_i$ has at least one edge and runs along the centre of
    $H_i$ in clockwise direction around the north pole.
  \item\label{spindle:segment:base}
    The base $u_{j+1}^i$ is a vertex on $P_i\setminus\{v_j^i\}$.
    \setcounter{brokenenumi}{\value{enumi}}
  \end{enumerate}
  We write $f_j^i=f$ and denote by $N_j^i$ the near-triangulation that
  $(T,f_j^i)$ induces at the vertex $v_j^i$ and the edge
  $v_j^iw_j^{i+1}$. The near-triangulation $N_j^i$ has the following
  properties.
  \begin{enumerate}\setcounter{enumi}{\value{brokenenumi}}
  \item\label{spindle:segment:rightmost}
    The edge $v_{j+1}^iw_{j+1}^{i+1}$ is part of the boundary of a face
    of $N_j^i$ whose third vertex $x$ lies in the subpath of $P_i$ from
    $v_j^i$ to the predecessor of $u_{j+1}^i$.
  \item\label{spindle:segment:faces}
    Every edge of $P_{i+1}$ is part of the boundary of a face of
    $N_j^i$ whose third vertex is in $P_i$.
  \item\label{spindle:segment:nochords}
    No two vertices in $P_{i+1}$ are connected by a chord in $N_j^i$.
  \item\label{spindle:segment:order2}
    If $m=2$ and $\varphi(v_j^i)=v_{j+1}^i$, then there is no edge in
    $N_j^i$ from $v_j^i$ to $v_{j+1}^i$.
  \end{enumerate}
\end{lemma}

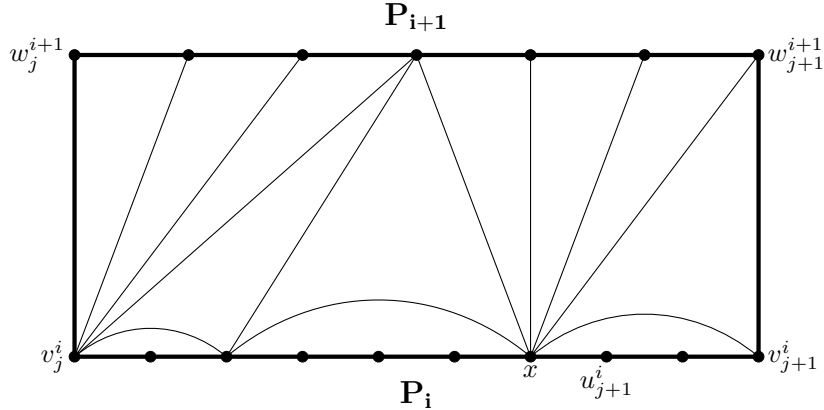
\begin{figure}[htbp]
  \centering
  \begin{tikzpicture}
    \draw[ultra thick] (0,0) rectangle (9,4);
    
    \foreach \x in {0,...,9}
      \filldraw (\x,0) circle (2pt);
    \foreach \x in {0,1.5,...,9}
      \filldraw (\x,4) circle (2pt);
      
    \draw (9,4) -- (6,0) to[out=40,in=140] (9,0);
    \draw (7.5,4) -- (6,0);
    \draw (6,4) -- (6,0);
    \draw (4.5,4) -- (6,0) to[out=140,in=40] (2,0);
    \draw (4.5,4) -- (2,0) to[out=140,in=40] (0,0);
    \draw (4.5,4) -- (0,0);
    \draw (3,4) -- (0,0);
    \draw (1.5,4) -- (0,0);
      
    \draw (0,0) node[anchor=0] {$v_j^i$};
    \draw (9,0) node[anchor=180] {$v_{j+1}^i$};
    \draw (0,4) node[anchor=0] {$w_j^{i+1}$};
    \draw (9,4) node[anchor=180] {$w_{j+1}^{i+1}$};
    \draw (6,0) node[anchor=90] {$x$};
    \draw (7,0) node[anchor=90] {$u_{j+1}^i$};
    \draw (4.5,-.5) node {\large $\mathbf{P_i}$};
    \draw (4.5,4.5) node {\large $\mathbf{P_{i+1}}$};
  \end{tikzpicture}
  \caption{The structure of the near-triangulation in a segment of
    the fyke net.}
  \label{fig:spindle:segment}
\end{figure}

\begin{proof}
  Property~\ref{spindle:segment:boundary} is part of the definition
  of a segment and~\ref{spindle:segment:paths} is immediate by the
  definition of the sources and targets. Property~\ref{spindle:segment:base} is clear by the way the source
  $v_{j+1}^i$ has been chosen.

 Property~\ref{spindle:segment:rightmost}
  follows from the existence of a face having the edge
  $v_{j+1}^iw_{j+1}^{i+1}$ on its boundary and the fact that its
  third vertex $x$ cannot be
  \begin{itemize}
  \item a vertex in $P_{i+1}$, since this would contradict the
    choice of $w_{j+1}^{i+1}$ as the leftmost neighbour of
    $v_{j+1}^i$;
  \item an internal vertex of $N_j^i$, since then $x$ would have been in
    the $(i+1)$-st level of $T$, again contradicting the choice of
    $w_{j+1}^{i+1}$;
  \item a vertex on the subpath of $P_i$ from $u_{j+1}^i$ to
    $v_{j+1}^i$, since by the choice of $v_{j+1}^i$ no vertex on this
    path has a neighbour in the $(i+1)$-st level of $T$.
  \end{itemize}

  In order to prove~\ref{spindle:segment:faces}, let $e$ be an edge
  of $P_{i+1}$. It is part of the boundary of a unique face of $N_j^i$ and
  by the definition of $F_{i+1}$ it is also part of the boundary of a
  face of $T$ whose third vertex is in $F_i$. We will show that this
  latter face is also a face of $N_j^i$, thus
  showing~\ref{spindle:segment:faces}.

  We will prove this for the edges in $P_{i+1}$ one by one, starting
  from the edge at $w_j^{i+1}$. Let $x_1$ be the last neighbour of
  $w_j^{i+1}$ on $P_i$ (starting from $v_j^i$). The edge
  $x_1w_j^{i+1}$ divides $N_j^i$ into two parts, let $N_1$ be the part
  which contains all of $P_{i+1}$. The edge $x_1w_j^{i+1}$ is part of
  the boundary of a unique face of $N_1$, denote the third vertex of
  this face by $y_1$ (see Figure~\ref{fig:spindle:faces}). If $y_1$ is
  the neighbour of $w_j^{i+1}$ on $P_{i+1}$, then we have found the
  desired face. Otherwise, it cannot be a vertex of $P_{i+1}$ since
  the edge $w_j^{i+1}y_1$ is in $F_{i+1}$ and would thus also have
  been in $H_{i+1}$. Since $x_1$ was the last neighbour of $w_j^{i+1}$
  on $P_i$, $y_1$ has to be an internal vertex of $N_j^i$. Now repeat the
  construction with $y_1$ instead of $w_j^{i+1}$ to obtain a vertex
  $x_2$ on $P_i$ (possibly $x_2=x_1$), a near-triangulation
  $N_2\subseteq N_1$ and a vertex $y_2$. As before, $y_2$ cannot lie
  on $P_i$ by the definition of $x_2$ and not in
  $P_{i+1}\setminus\{w_j^{i+1}\}$ by the definition of $H_{i+1}$. It
  also cannot be $w_j^{i+1}$, since then the edge $x_2w_j^{i+1}$ would
  either contradict the choice of $x_1$ as the last neighbour of
  $w_j^{i+1}$ on $P_i$ (if $x_1\not=x_2$) or it would yield a double
  edge (if $x_1=x_2$), also a contradiction. We can thus continue the
  construction and will always obtain internal vertices of $N_j^i$ for
  $y_1,y_2,\dotsc$. Since these vertices are distinct and $N_j^i$ is
  finite, this is a contradiction, implying that $y_1$ must have been
  the neighbour of $w_j^{i+1}$ on $P_{i+1}$.

\begin{figure}[htbp]
  \centering
  \begin{tikzpicture}
    \draw[ultra thick] (0,0) rectangle (9,4);
    
    \foreach \x in {0,...,9}
      \filldraw (\x,0) circle (2pt);
    \foreach \x in {0,1.5,...,9}
      \filldraw (\x,4) circle (2pt);
    \filldraw (4,2) circle (1.5pt);
    \filldraw (5,2) circle (1.5pt);
    \filldraw (6,2) circle (1.5pt);
      
    \draw (0,4) -- (3,0) to[out=140,in=40] (0,0);
    \draw (0,4) -- (4,2) -- (3,0);
    \draw (4,2) -- (6,2) -- (5,0) -- cycle;
    \draw (5,2) -- (5,0) to[out=150,in=30] (3,0);
      
    \draw (0,0) node[anchor=0] {$v_j^i$};
    \draw (9,0) node[anchor=180] {$v_{j+1}^i$};
    \draw (0,4) node[anchor=0] {$w_j^{i+1}$};
    \draw (9,4) node[anchor=180] {$w_{j+1}^{i+1}$};
    \draw (3,0) node[anchor=90] {$x_1$};
    \draw (5,0) node[anchor=90] {$x_2=x_3$};
    \draw (4,2) node[anchor=270] {$y_1$};
    \draw (5,2) node[anchor=270] {$y_2$};
    \draw (6,2) node[anchor=270] {$y_3$};
    \draw (4.5,-.75) node {\large $\mathbf{P_i}$};
    \draw (4.5,4.5) node {\large $\mathbf{P_{i+1}}$};
  \end{tikzpicture}
  \caption{The construction proving Lemma~\ref{lem:spindle:segments}\ref{spindle:segment:faces}.
    Note that the vertices $x_1,x_2,\dotsc$ are not necessarily
    distinct. The vertices $y_1,y_2,\dotsc$, however, are mutually
    distinct.}
  \label{fig:spindle:faces}
\end{figure}
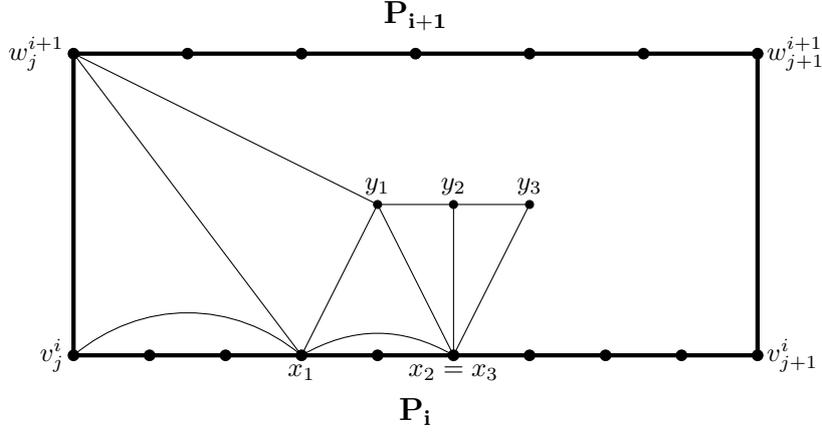

  The same construction for every later edge of $P_{i+1}$
  proves~\ref{spindle:segment:faces}.
  Property~\ref{spindle:segment:nochords} follows immediately
  from~\ref{spindle:segment:faces}.

  Finally, note that~\ref{spindle:segment:order2} is immediate since
  otherwise there would be a double edge in $T$ between $v_j^i$ and
  $v_{j+1}^i$.
\end{proof}

The near-triangulations pasted into leaves or pseudo-antarctic faces, however, do not have any
restrictions. Indeed, chords do neither contradict the construction
of the layers by Lemma~\ref{lem:cactus} nor can they result in double
edges.

A triangulation with rotative symmetry can thus be constructed by
first choosing a fyke net, then choosing, for every
isomorphism class of leaves or pseudo-antarctic faces, any near-triangulation to be pasted into
each of these leaves, and finally, for every isomorphism class of
segments, choosing a near-triangulation with
properties~\ref{spindle:segment:rightmost} and~\ref{spindle:segment:faces}
(and~\ref{spindle:segment:order2}, in the case of $m=2$) above (recall
that \ref{spindle:segment:nochords} follows immediately). More details
about this construction will be given in Section~\ref{sec:constr:rotative}.

\section{Reflective and rotative symmetries}\label{sec:both}

In this section, we assume that $\Aut(c_0,T)$ has a subgroup $H$ that contains both
reflective and rotative automorphisms. By Theorem~\ref{thm:aut},
$H$ is isomorphic to $D_n$ where $n\ge2$ is a divisor of
$d(c_0)$, i.e., there are $n$ reflections and $n-1$ rotations (and
the identity).

Since the rotations and the identity form a cyclic group, the results
of Section~\ref{sec:rotative} can be applied. In particular, there is
a unique cell $c_1\not= c_0$ that is invariant under all rotations.
Again, we call $c_0$ the north pole and $c_1$ the south pole of $T$.
For each reflection $\varphi$, there is a girdle $G_{\varphi}$ by the
results of Section~\ref{sec:reflective}.

Clearly, no two girdles are the same by Corollary~\ref{cor:incident}
and every girdle contains the north pole $c_0$ by definition. Thus, there
are $2n$ cells incident with $c_0$ that are invariant under some
reflection; denote them by $a_0,\dotsc,a_{2n-1}$, enumerated in the
same order they lie around $c_0$ (in clockwise direction, say). Then
for every reflection, there is an $i\in\{0,\dotsc,n-1\}$ such that the
invariant cells incident with $c_0$ are $a_i$ and $a_{n+i}$; denote
this automorphism by $\varphi_i$ and its girdle by $G_i$.

\begin{lemma}\label{lem:girdlesmeet}
  The girdles $G_0,\dotsc,G_{n-1}$ have the following properties.
  \begin{enumerate}
  \item\label{girdles:south}
    North and south pole are central cells of every girdle.
  \item\label{girdles:central}
    The two poles are the only cells that are central cells of more
    than one girdle.
  \end{enumerate}
\end{lemma}

\begin{proof}
  The north pole is a central cell of every girdle by definition
  (Definition~\ref{def:girdle}). Let
  $G_i,G_j$ be two distinct girdles. We first show that there is
  another cell that is central in both of them and then prove that
  this cell is the south pole. This will prove
  both~\ref{girdles:south} and~\ref{girdles:central}.

  Since for each of $\varphi_i,\varphi_j$, the invariant cells
  incident with $c_0$ lie opposite, the central cells of $G_i$
  incident with the north pole lie in different sides (or on the
  boundaries of different sides) of $G_j$. Since the central cells of
  a girdle separate its sides, $G_i$ and $G_j$ meet in at least one
  central cell apart from the north pole. Let $c$ be such a cell.

  Consider the automorphism $\varphi_i\circ\varphi_j$. Since $c$ is
  invariant both under $\varphi_i$ and under $\varphi_j$, it is also
  invariant under $\varphi_i\circ\varphi_j$. But the composition of
  two distinct reflections is always a rotation and thus, the only
  cells invariant under $\varphi_i\circ\varphi_j$ are the north and
  south pole, implying that $c$ is the south pole.
\end{proof}

Since the cells $a_0,\dotsc,a_{2n-1}$ form a cyclic sequence around
$c_0$, we will also consider their indices modulo $2n$, similarly to
the previous section. For simplicity, we will again write $a_i$
instead of $a_{i\pmod{2n}}$. The same kind of notation will be used
for the girdles $G_0,\dotsc,G_{n-1}$ (modulo $n$ instead of
modulo $2n$).

The rotations can be enumerated as $\rho_1,\dotsc,\rho_{n-1}$ so that
every $\rho_i$ satisfies
\begin{equation*}
  \rho_i(a_j) = a_{j+2i}
\end{equation*}
for all $j=0,\dotsc,2n-1$. With this notation, we have
$\rho_1^i=\rho_i$ for all $i=1,\dotsc,n-1$ (and $\rho_1^n=\id$).

Corollary~\ref{cor:incident} and the automorphism $\rho_1$ show that
$G_0$ is isomorphic to $G_2$, $G_1$ is isomorphic to $G_3$, and so
on. If $n$ is odd, this implies that all girdles are isomorphic; if
$n$ is even, all $G_i$ with even $i$ are isomorphic as well as the
ones with odd $i$. Moreover, in the latter case every girdle is
mapped to itself by the rotation $\rho_{\frac{n}2}$. In that case we
call the girdles \emph{symmetric} and $\rho_{\frac{n}2}$ a
\emph{symmetry} of each girdle. We thus have proved the following.

\begin{lemma}\label{lem:girdlesisom}
  For every $i=1,\dotsc,n-1$, the following holds.
  \begin{enumerate}
  \item\label{girdles:isom}
    For every $j=0,\dotsc,n-1$, the rotation $\rho_i$ induces an
    isomorphism between the girdles $G_j$ and $G_{j+2i}$.
  \item\label{girdles:odd}
    If $n$ is odd, all girdles are isomorphic.
  \item\label{girdles:even}
    If $n$ is even, $\rho_{\frac{n}2}$ is a symmetry of each girdle
    and every two girdles $G_i, G_j$ with $i-j$ even are
    isomorphic.
  \end{enumerate}
\end{lemma}

Recall that Lemma~\ref{lem:girdlesmeet} tells us that any two girdles cross
precisely twice: once at each of the poles. However, while a central
cell of a girdle cannot be a central cell of another girdle (unless
it is one of the poles), it might well be an \emph{outer} cell of
another girdle.

Since every girdle $G_i$ has both poles as central cells, they divide
it into two parts in a natural way: if $(x_j)_{j\in\ZZ_m}$ is the
cyclic sequence from Lemma~\ref{lem:pre-girdle} with $x_0=c_0$ (note
that by Lemmas~\ref{lem:pre-girdle}
and~\ref{lem:girdle}\ref{girdle:invariant} this sequence is unique up
to orientation), then $x_k=c_1$ for some $k$ and we can consider the
sequences $x_0,x_1,\dotsc,x_k$ and $x_k,x_{k+1},\dotsc,x_{m-1},x_0$.
One of the sequences contains $a_i$, so we denote the union of its
elements and their boundaries by $M_i$. The other sequence contains
$a_{n+i}$, we denote the union of its elements and their boundaries by
$M_{n+i}$. We call $M_i$ and $M_{n+i}$ \emph{meridians}, the cells
from the respective sequence of $x_j$'s are the \emph{central cells}
of $M_i$ and $M_{n+i}$, respectively. The other cells are \emph{outer
cells}, as before. Note that a central cell of $G_i$ that lies on the
boundary of one of the poles will be contained in both $M_i$ and
$M_{n+i}$. However, it will only be a central cell in one of them.
Clearly, $G_i = M_i\cup M_{n+i}$ and thus $\bigcup_{i=0}^{n-1}G_i =
\bigcup_{i=0}^{2n-1}M_i$.

Like the girdles, the meridians form a cyclic sequence; for
simplicity, we will write $M_i$ instead of $M_{i\pmod{2n}}$.

\begin{definition}[Skeleton]\label{def:skeleton}
  The union $S := \bigcup_{i=0}^{2n-1}M_i$ is the \emph{skeleton} of
  $T$. For every $i=0,\dotsc,2n-1$, we say that the meridians $M_i$
  and $M_{i+1}$ are \emph{adjacent}. Every face of $S$ that is
  not a central cell of at least one of the meridians (equivalently:
  of one of the girdles) is a \emph{segment} of $S$.
\end{definition}

Note that the skeleton of $T$ is unique since all the girdles are.

\begin{lemma}\label{lem:skeleton1}
  The skeleton $S$ has the following properties.
  \begin{enumerate}
  \item\label{skeleton:reflections}
    Every reflection $\varphi_i$, $0\le i\le n-1$, induces an
    isomorphism between $M_{i-j}$ and $M_{i+j}$ for every
    $j=1,\dotsc,n-1$.
  \item\label{skeleton:rotations}
    Every rotation $\rho_i$, $1\le i\le n-1$, induces an isomorphism
    between $M_j$ and $M_{j+2i}$ for every $j=0,\dotsc,2n-1$.
  \item\label{skeleton:meridians}
    There is an isomorphism in $H$ that maps $M_i$ to $M_j$ if and
    only if $i-j$ is even.
  \item\label{skeleton:central}
    For every central cell $c$ of a meridian $M_i$, $0\le i\le 2n-1$,
    exactly one of the following holds.
    \renewcommand{\theenumii}{\rm (C\arabic{enumii})}
    \begin{enumerate}
    \item\label{central:pole}
      $c$ is a pole;
    \item\label{central:incident}
      $c$ lies on the boundary of a pole;
    \item\label{central:disjoint}
      $c$ is not contained in any other meridian;
    \item\label{central:adjacent}
      $c$ is an outer cell of both meridians adjacent to $M_i$ and
      not contained in any other meridian.
    \end{enumerate}
  \item\label{skeleton:boundaries}
    Every segment of $S$ is bounded by a cycle that is contained in
    the union of two adjacent meridians.
  \item\label{skeleton:number}
    There is a non-negative integer $s$ such that for every pair
    $(M_i,M_{i+1})$ of adjacent meridians there are precisely $s$
    such segments.
  \end{enumerate}
\end{lemma}

\begin{proof}
  Claims~\ref{skeleton:reflections} and~\ref{skeleton:rotations}
  follow from Corollary~\ref{cor:incident} and the way $\varphi_i$
  and $\rho_i$ act on $a_1,\dotsc,a_{2n}$.
  Claim~\ref{skeleton:meridians} is an immediate corollary
  of~\ref{skeleton:reflections} or~\ref{skeleton:rotations}.

  To prove~\ref{skeleton:central}, let $c$ be a central cell of
  $M_i$. Note first that only one of the
  cases~\ref{central:pole}--\ref{central:adjacent} can hold. Now
  assume that~\ref{central:pole}--\ref{central:disjoint} do not hold,
  i.e., $c$ is neither a pole nor lies on the boundary of a pole and
  there is at least one meridian $M_j$ with $j\not=i$ that contains
  $c$. By Lemma~\ref{lem:girdlesmeet}\ref{girdles:central}, $c$ is an
  outer cell of every such meridian $M_j$.

  The central cells of $M_{i-1}$ and $M_{i+1}$ separate the sphere
  into two parts, one of which contains the central cells of $M_i$
  (apart from the poles) while the other contains the central cells
  (apart from the poles) of all other meridians. This implies that
  $c$ is an outer cell of at least one of $M_{i-1},M_{i+1}$ and not
  contained in any other meridian; it remains to show that $c$ is an
  outer cell of both $M_{i-1}$ and $M_{i+1}$.
  By~\ref{skeleton:reflections}, $\varphi_i$ (or $\varphi_{i-n}$ if
  $i>n$) induces an isomorphism between $M_{i-1}$ and $M_{i+1}$ and
  since $c$ is invariant under $\varphi_i$, it is an outer cell of
  both meridians adjacent to $M_i$. This
  proves~\ref{skeleton:central}.

  For~\ref{skeleton:boundaries}, note first that every segment of $S$
  is bounded by a cycle since the graph $S$ is 2-connected. To prove
  the other half of the statement, choose $2n$ arcs (injective
  topological paths) from one pole to the other, one in the union of
  the central cells of each meridian. By~\ref{skeleton:central}, these
  arcs only meet in the two poles and thus divide the sphere into $2n$
  discs, each having a boundary that is contained in the union of two
  of the arcs, with the corresponding meridians being adjacent. Since
  every segment of $S$ is contained in such a disc and no other
  meridian contains a point in this disc,~\ref{skeleton:boundaries}
  follows.

  Finally,~\ref{skeleton:number} follows by applying rotations and
  reflections to the segments with their boundaries in $M_1\cup M_2$.
\end{proof}

By Lemma~\ref{lem:skeleton1}\ref{skeleton:number}, we can denote
the segments of $S$ whose boundaries are contained in the union of
$M_i$ and $M_{i+1}$ by $f_1^i,\dotsc,f_s^i$. Note that the cycle
from Lemma~\ref{lem:skeleton1}\ref{skeleton:boundaries} bounding
$f_j^i$ is the union of a subpath of $M_i$ and a subpath of $M_{i+1}$.
These paths meet in their end vertices; denote by $v_j^i$ their end
vertex closer to the north pole and by $w_j^i$ the one closer to the
south pole. Without loss of generality, we assume that the enumeration
of $f_1^i,\dotsc,f_s^i$ is chosen so that $v_j^i$ is closer to the
north pole than $v_{j'}^i$ whenever $j<j'$. Finally, let $e_j^i$ be the
edge on the boundary of $f_j^i$ that is incident with $v_j^i$ and
\begin{enumerate}
\item contained in $M_i$ if $i$ is even, or
\item contained in $M_{i+1}$ if $i$ is odd.
\end{enumerate}
With this notation and Lemma~\ref{lem:skeleton1}\ref{skeleton:reflections}
and~\ref{skeleton:rotations}, we obtain the following.

\begin{lemma}\label{lem:skeleton2}
  Let $j\in\{1,\dotsc,s\}$.
  \begin{enumerate}
  \item\label{skeleton:insegment}
    The pair $(T,f_j^i)$ induces a near-triangulation $N_j^i$ at $v_j^i$ and
    $e_j^i$ for every $i$.
  \item\label{skeleton:isom}
    The near-triangulations $N_j^0,\dotsc,N_j^{2n-1}$ are isomorphic.
  \end{enumerate}
\end{lemma}

For a complete description of all possible skeletons, we need to
characterise their structure at the poles and at other points where
two adjacent meridians meet.

\begin{lemma}\label{lem:skeleton:poles}
  Let $S$ be a skeleton and $c$ be one of the poles of $T$. Then the
  structure of $S$ at $c$ is the following.
  \begin{enumerate}
  \item\label{poles:vertex}
    If $c$ is a vertex, then either
    \begin{enumerate}
    \item\label{poles:v:disjoint}
      no two meridians meet in a cell incident with $c$ or
    \item\label{poles:v:edge}
      there is a number $k\ge 1$ such that every two adjacent meridians
      meet in their first $k$ edges starting from $c$.
    \end{enumerate}
  \item\label{poles:edge}
    If $c=uv$ is an edge, then two non-adjacent meridians, say
    $M_0$ and $M_2$, have $u$ respectively $v$ as a central cell and
    the other two have its incident faces as central cells.
    No two meridians meet in an edge $e\not=c$ incident with $u$ or
    $v$.
  \item\label{poles:face}
    If $c$ is a face $f$ with vertices $u,v,w$ on its boundary,
    then three mutually non-adjacent meridians, say $M_0,M_2,M_4$, have $u$,
    $v$, respectively $w$ as a central cell and the other three have
    its incident edges as central cells. Either
    \begin{enumerate}
    \item\label{poles:f:disjoint}
      no two meridians meet in a cell incident with exactly one of
      $u,v,w$ or
    \item\label{poles:f:edge}
      there is a number $k\ge 1$ such that every two adjacent meridians
      meet in their first $k$ edges starting from $c$.
    \end{enumerate}
  \end{enumerate}
\end{lemma}

\begin{proof}
  Statements~\ref{poles:vertex} and~\ref{poles:face} follow from
  Lemma~\ref{lem:skeleton1}\ref{skeleton:reflections}
  and~\ref{skeleton:rotations}. The first claim in~\ref{poles:edge} is
  immediate, since each of the four meridians contains a different
  cell incident with $c$ as a central cell. For $i\in\{1,3\}$, denote
  by $f_i$ the face incident with $c$ that is a central cell of $M_i$
  (see Figure~\ref{fig:skeleton:poles}\ref{poles:edge}). Since $u,v$
  are incident with $f_1$ and $f_3$, there are unique vertices
  $w_1,w_3$ different from $u$ and $v$ that are incident with $f_1$
  and $f_3$, respectively. Note that $w_1$ and $w_3$ are distinct,
  since otherwise $f_1$ and $f_3$ would have the same set of incident
  vertices, which is not possible as our triangulations are simple and
  non-trivial. Suppose that $M_0$ meets $M_1$ in an edge $e\not=c$
  incident with $u$; this has to be the edge $uw_1$. By applying
  $\varphi_0$, we see that $M_0$ meets $M_3$ in the edge $uw_3$. Thus,
  $w_1$ and $w_3$ are connected by an edge $e_0$ that is central in
  $M_0$. Applying $\varphi_1$ shows that $M_2$ also has a central edge
  $e_2$ that connects $w_1$ and $w_3$. Since our triangulations are
  simple, the edges $e_0$ and $e_2$ are identical and $T$ is a $K_4$.
  Since we assume all triangulations to be non-trivial, this is a
  contradiction. We have thus shown~\ref{poles:edge}.
\end{proof}

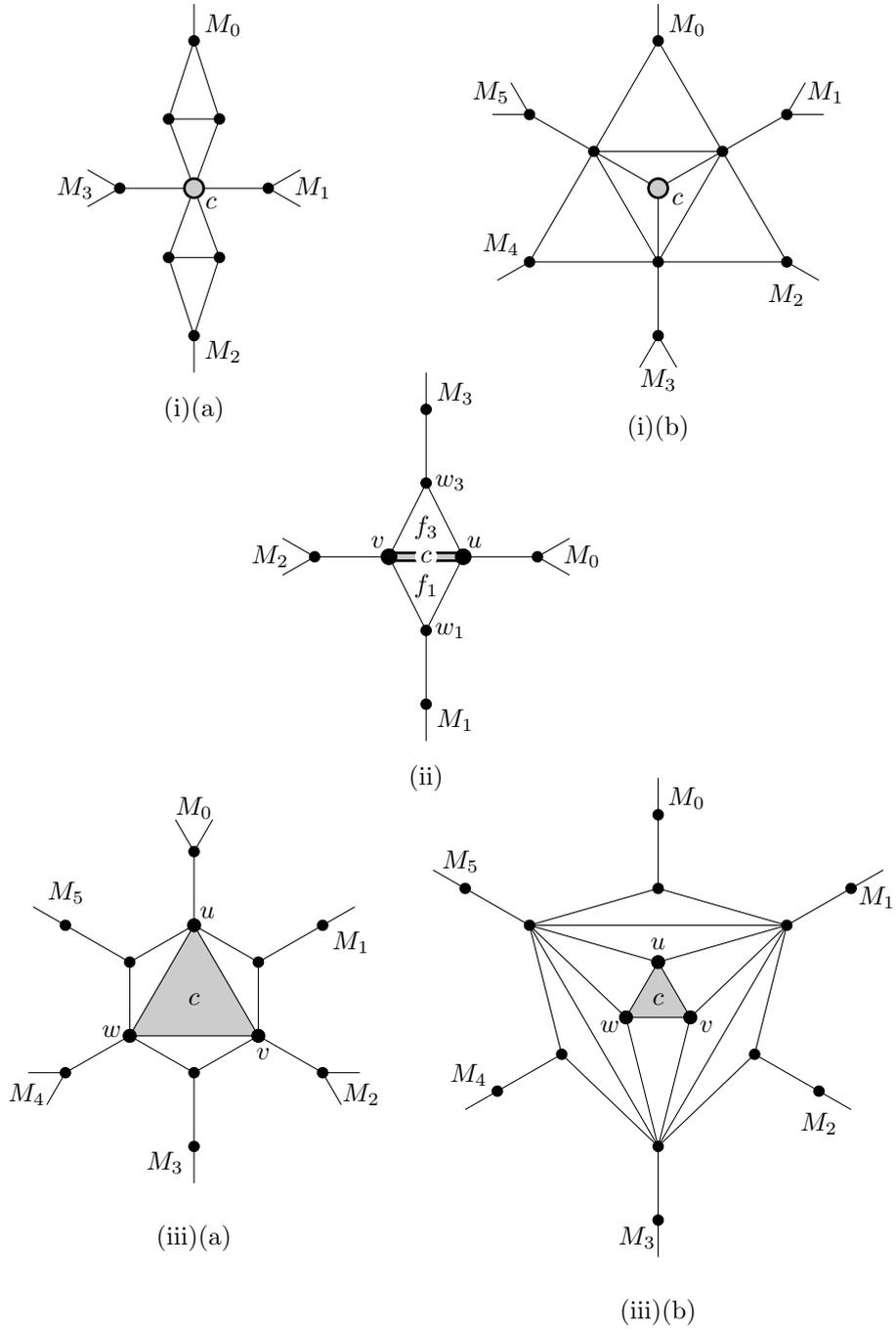
\begin{figure}[htbp]
  \centering
  \begin{tikzpicture}
    \begin{scope}
      \draw (70:1) -- (0,0) -- (110:1) -- (90:2) -- (70:1) -- (110:1);
      \draw (250:1) -- (0,0) -- (290:1) -- (270:2) -- (250:1) -- (290:1);
      \draw (0:1) -- (0,0) -- (180:1);
      \draw (0:1) +(30:.5) -- +(0,0) -- +(330:.5);
      \draw (180:1) +(150:.5) -- +(0,0) -- +(210:.5);
      \draw (90:2) -- (90:2.5);
      \draw (270:2) -- (270:2.5);

      \filldraw (0,0) circle (4pt);
      \filldraw[black!20] (0,0) circle (3pt);
      \filldraw (0:1) circle (2pt);
      \filldraw (70:1) circle (2pt);
      \filldraw (110:1) circle (2pt);
      \filldraw (180:1) circle (2pt);
      \filldraw (250:1) circle (2pt);
      \filldraw (290:1) circle (2pt);
      \filldraw (90:2) circle (2pt);
      \filldraw (270:2) circle (2pt);
      
      \draw (320:.3) node {$c$};
      \draw (80:2.25) node {$M_0$};
      \draw (0:1.6) node {$M_1$};
      \draw (280:2.25) node {$M_2$};
      \draw (180:1.6) node {$M_3$};
      
      \draw (0,-3) node {(i)(a)};
    \end{scope}

    \begin{scope}[xshift=6.25cm]
      \draw (0,0) -- (30:1) -- (30:2);
      \draw (0,0) -- (150:1) -- (150:2);
      \draw (0,0) -- (270:1) -- (270:2);
      \draw (30:1) -- (150:1) -- (270:1) -- (30:1) -- (90:2) -- (150:1) -- (210:2) -- (270:1) -- (330:2) -- cycle;
      \draw (90:2) -- (90:2.5);
      \draw (210:2) -- (210:2.5);
      \draw (330:2) -- (330:2.5);
      \draw (30:2) +(0:.5) -- +(0,0) -- +(60:.5);
      \draw (150:2) +(120:.5) -- +(0,0) -- +(180:.5);
      \draw (270:2) +(240:.5) -- +(0,0) -- +(300:.5);

      \filldraw (0,0) circle (4pt);
      \filldraw[black!20] (0,0) circle (3pt);
      \filldraw (30:1) circle (2pt);
      \filldraw (150:1) circle (2pt);
      \filldraw (270:1) circle (2pt);
      \filldraw (30:2) circle (2pt);
      \filldraw (90:2) circle (2pt);
      \filldraw (150:2) circle (2pt);
      \filldraw (210:2) circle (2pt);
      \filldraw (270:2) circle (2pt);
      \filldraw (330:2) circle (2pt);
      
      \draw (330:.3) node {$c$};
      \draw (80:2.25) node {$M_0$};
      \draw (30:2.6) node {$M_1$};
      \draw (320:2.25) node {$M_2$};
      \draw (270:2.6) node {$M_3$};
      \draw (200:2.25) node {$M_4$};
      \draw (150:2.6) node {$M_5$};
      
      \draw (0,-3.25) node {(i)(b)};
    \end{scope}
    
    \begin{scope}[xshift=3.125cm,yshift=-5cm]
      \draw (-1.5,0) -- (1.5,0);
      \draw[line width=4pt] (-.5,0) -- (.5,0);
      \draw[black!20,line width=2pt] (-.5,0) -- (.5,0);
      \draw (-.5,0) -- (0,1) -- (.5,0) -- (0,-1) -- cycle;
      \draw (0,1) -- (0,2.5);
      \draw (0,-1) -- (0,-2.5);
      \draw (1.5,0) +(30:.5) -- +(0,0) -- +(330:.5);
      \draw (-1.5,0) +(150:.5) -- +(0,0) -- +(210:.5);
    
      \filldraw (-.5,0) circle (3pt);
      \filldraw (.5,0) circle (3pt);
      \filldraw (-1.5,0) circle (2pt);
      \filldraw (1.5,0) circle (2pt);
      \filldraw (0,1) circle (2pt);
      \filldraw (0,-1) circle (2pt);
      \filldraw (0,2) circle (2pt);
      \filldraw (0,-2) circle (2pt);
      
      \filldraw[white] (0,0) circle (4pt);
      \draw (0,0) node {$c$};
      \draw (.65,.2) node {$u$};
      \draw (-.65,.2) node {$v$};
      \draw (0,-.4) node {$f_1$};
      \draw (0,.4) node {$f_3$};
      \draw (0,-1) node[anchor=west] {$w_1$};
      \draw (0,1) node[anchor=west] {$w_3$};
      \draw (0:2.1) node {$M_0$};
      \draw (280:2.25) node {$M_1$};
      \draw (180:2.1) node {$M_2$};
      \draw (80:2.25) node {$M_3$};
      
      \draw (0,-3) node {(ii)};
    \end{scope}
%
%
%
%
%
%
%
%
    
    \begin{scope}[yshift=-11cm]
      \filldraw[black!20] (90:1) -- (210:1) -- (330:1) -- cycle;

      \draw (90:1) -- (210:1) -- (330:1) -- cycle;
      \draw (30:1) -- (90:1) -- (150:1) -- (210:1) -- (270:1) -- (330:1) -- cycle;
      \draw (90:1) -- (90:2);
      \draw (30:1) -- (30:2.5);
      \draw (330:1) -- (330:2);
      \draw (270:1) -- (270:2.5);
      \draw (210:1) -- (210:2);
      \draw (150:1) -- (150:2.5);
      \draw (90:2) +(60:.5) -- +(0,0) -- +(120:.5);
      \draw (210:2) +(180:.5) -- +(0,0) -- +(240:.5);
      \draw (330:2) +(300:.5) -- +(0,0) -- +(0:.5);
      
      \filldraw (90:1) circle (2.5pt);
      \filldraw (210:1) circle (2.5pt);
      \filldraw (330:1) circle (2.5pt);
      \filldraw (30:1) circle (2pt);
      \filldraw (150:1) circle (2pt);
      \filldraw (270:1) circle (2pt);
      \filldraw (90:2) circle (2pt);
      \filldraw (210:2) circle (2pt);
      \filldraw (330:2) circle (2pt);
      \filldraw (30:2) circle (2pt);
      \filldraw (150:2) circle (2pt);
      \filldraw (270:2) circle (2pt);
      
      \draw (0,0) node {$c$};
      \draw (90:1) +(45:.25) node {$u$};
      \draw (330:1) +(285:.25) node {$v$};
      \draw (210:1) +(165:.25) node {$w$};
      \draw (90:2.6) node {$M_0$};
      \draw (20:2.25) node {$M_1$};
      \draw (330:2.6) node {$M_2$};
      \draw (260:2.25) node {$M_3$};
      \draw (210:2.6) node {$M_4$};
      \draw (140:2.25) node {$M_5$};
      
      \draw (0,-3.25) node {(iii)(a)};
    \end{scope}
    
    \begin{scope}[xshift=6.25cm,yshift=-11cm]
      \filldraw[black!20] (90:.5) -- (210:.5) -- (330:.5) -- cycle;

      \draw (90:.5) -- (210:.5) -- (330:.5) -- cycle;
      \draw (30:2) -- (150:2) -- (270:2) -- cycle;
      \draw (30:2) -- (90:.5) -- (150:2) -- (210:.5) -- (270:2) -- (330:.5) -- cycle;
      \draw (30:2) -- (90:1.5) -- (150:2) -- (210:1.5) -- (270:2) -- (330:1.5) -- cycle;
      \draw (30:2) -- (30:3.5);
      \draw (90:1.5) -- (90:3);
      \draw (150:2) -- (150:3.5);
      \draw (210:1.5) -- (210:3);
      \draw (270:2) -- (270:3.5);
      \draw (330:1.5) -- (330:3);
      
      \filldraw (90:.5) circle (2.5pt);
      \filldraw (210:.5) circle (2.5pt);
      \filldraw (330:.5) circle (2.5pt);
      \filldraw (30:2) circle (2pt);
      \filldraw (150:2) circle (2pt);
      \filldraw (270:2) circle (2pt);
      \filldraw (90:1.5) circle (2pt);
      \filldraw (210:1.5) circle (2pt);
      \filldraw (330:1.5) circle (2pt);
      \filldraw (30:3) circle (2pt);
      \filldraw (150:3) circle (2pt);
      \filldraw (270:3) circle (2pt);
      \filldraw (90:2.5) circle (2pt);
      \filldraw (210:2.5) circle (2pt);
      \filldraw (330:2.5) circle (2pt);
      
      \draw (0,0) node {$c$};
      \draw (90:.75) node {$u$};
      \draw (330:.75) node {$v$};
      \draw (210:.75) node {$w$};
      \draw (82:2.75) node {$M_0$};
      \draw (25:3.25) node {$M_1$};
      \draw (322:2.75) node {$M_2$};
      \draw (265:3.25) node {$M_3$};
      \draw (202:2.75) node {$M_4$};
      \draw (145:3.25) node {$M_5$};
      
      \draw (0,-4.25) node {(iii)(b)};
    \end{scope}
  \end{tikzpicture}
  \caption{The possible structures of a skeleton at a pole as stated in
    Lemma~\ref{lem:skeleton:poles}, with $k=1$ in the
    cases~\ref{poles:vertex}\ref{poles:v:edge}
    and~\ref{poles:face}\ref{poles:f:edge}.}
  \label{fig:skeleton:poles}
\end{figure}

Lemma~\ref{lem:skeleton:poles} describes the structure of the skeleton
at the poles. The following lemma deals with the intersections of
adjacent meridians between two segments.

\begin{lemma}\label{lem:skeleton:touching}
  Let $j\in\{1,\dotsc,s-1\}$ be fixed. Then there is a number $k_j\ge
  0$, such that for every $i$, the intersection of $M_i$ and $M_{i+1}$
  has a component that is a path of length $k_j$ from $w_j^i$ to
  $v_{j+1}^i$.
\end{lemma}

\begin{proof}
  This follows immediately from
  Lemma~\ref{lem:skeleton1}\ref{skeleton:reflections}.
\end{proof}

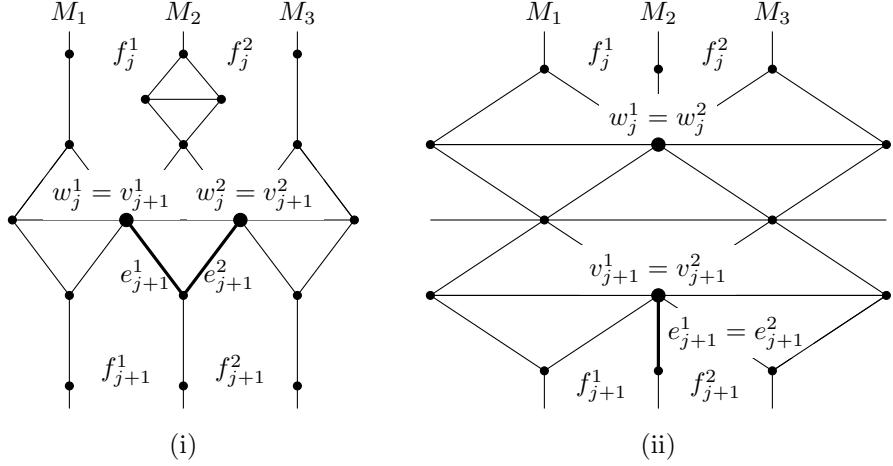
\begin{figure}[htbp]
  \centering
  \begin{tikzpicture}
    \begin{scope}
      \foreach \x in {0,1.5,3}
      {
        \draw (\x,0) -- (\x+.75,-1) -- (\x+1.5,0) -- (\x+.75,1) -- (\x,0) -- (\x+1.5,0);
      }
      \draw (.75,1) -- (.75,2.5);
      \draw (.75,-1) -- (.75,-2.5);
      \draw (2.25,-1) -- (2.25,-2.5);
      \draw (3.75,1) -- (3.75,2.5);
      \draw (3.75,-1) -- (3.75,-2.5);
      \draw (1.75,1.6) -- (2.25,1) -- (2.75,1.6) -- (2.25,2.2) -- (1.75,1.6) -- (2.75,1.6);
      \draw (2.25,2.2) -- (2.25,2.5);

      \draw (.75,2.75) node {$M_1$};
      \draw (2.25,2.75) node {$M_2$};
      \draw (3.75,2.75) node {$M_3$};
      \draw (1.5,2.2) node {$f_j^1$};
      \draw (3,2.2) node {$f_j^2$};
      \draw (1.5,-2) node {$f_{j+1}^1$};
      \draw (3,-2) node {$f_{j+1}^2$};
      \draw (1.5,0) node[rectangle,fill=white,anchor=300] {$w_j^1=v_{j+1}^1$};
      \draw (3,0) node[rectangle,fill=white,anchor=240] {$w_j^2=v_{j+1}^2$};
      \draw (1.75,-.75) node {$e_{j+1}^1$};
      \draw (2.85,-.75) node {$e_{j+1}^2$};
      
      \foreach \x in {0,1.5,3}
      {
        \filldraw (\x,0) circle (1.5pt);
        \filldraw (\x+1.5,0) circle (1.5pt);
        \filldraw (\x+.75,-1) circle (1.5pt);
        \filldraw (\x+.75,1) circle (1.5pt);
        \filldraw (\x+.75,-2.2) circle (1.5pt);
        \filldraw (\x+.75,2.2) circle (1.5pt);
      }
      \filldraw (1.75,1.6) circle (1.5pt);
      \filldraw (2.75,1.6) circle (1.5pt);

      \draw (0,0) -- (.75,1);
      \draw (4.5,0) -- (3.75,1);
      \draw[very thick] (1.5,0) -- (2.25,-1) -- (3,0);
      \filldraw (1.5,0) circle (2.5pt);
      \filldraw (3,0) circle (2.5pt);

      \draw (2.25,-3) node {(i)};
    \end{scope}

    \begin{scope}[xshift=5.5cm]
      \foreach \x in {1.5,4.5}
      {
        \draw (\x,2) -- (\x,2.5);
        \draw (\x,-2) -- (\x,-2.5);
        \foreach \y in {-1,1}
          \draw (\x-1.5,\y) -- (\x,\y-1) -- (\x+1.5,\y) -- (\x,\y+1) -- (\x-1.5,\y) -- (\x+1.5,\y);
        \filldraw (\x,2) circle (1.5pt);
        \filldraw (\x-1.5,1) circle (1.5pt);
        \filldraw (\x+1.5,1) circle (1.5pt);
        \filldraw (\x,0) circle (1.5pt);
        \filldraw (\x-1.5,-1) circle (1.5pt);
        \filldraw (\x+1.5,-1) circle (1.5pt);
        \filldraw (\x,-2) circle (1.5pt);
      }
      \draw (0,0) -- (6,0);
      \draw (3,1) -- (3,2.5);
      \draw (3,-1) -- (3,-2.5);
      \filldraw (3,2) circle (1.5pt);
      \filldraw (3,-2) circle (1.5pt);

      \draw (1.5,2.75) node {$M_1$};
      \draw (3,2.75) node {$M_2$};
      \draw (4.5,2.75) node {$M_3$};
      \draw (2.25,2.2) node {$f_j^1$};
      \draw (3.75,2.2) node {$f_j^2$};
      \draw (2.25,-2.2) node {$f_{j+1}^1$};
      \draw (3.75,-2.2) node {$f_{j+1}^2$};
      \draw (3,1) node[rectangle,fill=white,anchor=270] {$w_j^1=w_j^2$};
      \draw (3,-1) node[rectangle,fill=white,anchor=270] {$v_{j+1}^1=v_{j+1}^2$};
      \draw (3,-1.5) node[rectangle,fill=white,anchor=180] {$e_{j+1}^1=e_{j+1}^2$};

      \draw (0,1) -- (6,1);
      \draw (0,-1) -- (6,-1) -- (4.5,-2);
      \draw[very thick] (3,-1) -- (3,-2);
      \filldraw (3,1) circle (2.5pt);
      \filldraw (3,-1) circle (2.5pt);

      \draw (3,-3) node {(ii)};
    \end{scope}
  \end{tikzpicture}
  \caption{The structure of a skeleton between two segments as
    described in Lemma~\ref{lem:skeleton:touching}. In Case~(i), we
    have $k_j=0$, in Case~(ii) $k_j=2$.}
  \label{fig:skeleton:touching}
\end{figure}

All possible skeletons can thus be constructed by first choosing the
numbers $n\ge 2$ and $s$ and the dimensions of the poles. Note that a
pole can only be an edge if $n=2$ and it can only be a face if $n=3$.
Then choose the structure at the poles according to
Lemma~\ref{lem:skeleton:poles} and between the segments according to
Lemma~\ref{lem:skeleton:touching}.

All triangulations with both reflective and rotative symmetry can
be obtained by first taking a skeleton and then inserting the
same near-triangulation in each type of segment according to
Lemma~\ref{lem:skeleton2}. Similarly to the case of reflective symmetries,
the near-triangulation inserted into a segment is only allowed to
have chords that do not produce double edges by reflecting. In this
case, this means that for every chord of the near-triangulation and
each meridian bounding the corresponding segment, not both end
vertices of the chord are contained in the meridian and central in
it. More details about this construction will be given in
Section~\ref{sec:constr:both}.

\section{Constructions}\label{sec:constructions}

In this section we formalise the constructive decompositions developed
in the previous sections and show how to construct the basic graphs
arising in them: girdles, fyke nets, and skeletons.

\subsection{Reflective symmetries}\label{sec:constr:reflective}

The construction of all possible girdles is rather easy. Once the
length $\ell$ of the girdle and the number $d$ of diamonds are fixed,
all that is left is to consider all arrangements of $d$ diamonds on a
girdle of length $\ell$. Note that $d\le\frac{\ell}{2}$ is necessary;
in the case of $c_0$ being a face, we furthermore have $d\ge 1$.

Let a girdle $G$ be given. The near-triangulations that can be
inserted into the sides of $G$ in order to give rise to a
triangulation with reflective symmetry have to satisfy the conditions
of Lemma~\ref{lem:girdle:sides}. In particular, the distribution of
chords is restricted.

\begin{definition}
  Let $N$ be a near-triangulation and let $D$ be a subset of its set
  of outer vertices. We call $N$ \emph{chordless outside $D$} if every
  chord of $N$ has at least one end vertex in $D$.

  More generally, let a cycle
  $C$ with a root vertex $v_C$ and a root edge $e_C$ incident with
  $v_C$ be given and let $D_C$ be a set of vertices in $C$. Suppose
  that the length of $C$ is the same as the number of outer vertices
  of $N$ and let $\alpha$ be the unique isomorphism from $C$ to the
  boundary $C_N$ of the outer face of $N$ that maps $v_C$ to the root
  vertex $v_N$ of $N$ and $e_C$ to the root edge $e_N$ of $N$. We call
  $N$ \emph{chordless outside $D_C$} if it is chordless outside
  $\alpha(D_C)$.
\end{definition}

Recall that $j(G)$ is the smallest index for which $c_{j(G)}$ is a
vertex and let $v_1,v_2,e_1,e_2$ be given as in
Lemma~\ref{lem:girdle}\ref{girdle:inside}. Denote by $C_G$ the cycle
in $G$ bounding $f_1$ and let $D_G$ be the set of outer vertices of
$G$ in $C$. With this notation, Lemmas~\ref{lem:girdle}
and~\ref{lem:girdle:sides} give rise to the following.

\begin{theorem}\label{thm:reflective}
  The triangulations $T$ with a reflective symmetry in $\Aut(c_0,T)$
  are precisely the ones that can be constructed by choosing
  \begin{itemize}
  \item a girdle $G$ that contains $c_0$ as a central cell and
  \item a near-triangulation $N$ that is chordless outside $D_G$
  \end{itemize}
  and inserting a copy of $N$ into each side $f_1$ (respectively $f_2$)
  of $G$ at $v_1$ and $e_1$ (respectively at $v_2$ and $e_2$).
\end{theorem}

\begin{remark}
  Since the girdle $G_{\varphi}$ of a triangulation $T$ with respect
  to a given reflection $\varphi$ is only unique up to orientation,
  some triangulations have two ways of constructing them by inserting
  near-triangulations into the sides of a girdle. More precisely, the
  construction in Theorem~\ref{thm:reflective} is a 1-1 correspondence
  if and only if there is another reflection $\psi\not=\varphi$ in
  $\Aut(c_0,T)$ that fixes $G_{\varphi}$. By Theorem~\ref{thm:aut},
  this is equivalent to $\Aut(c_0,T)$ being isomorphic to $D_n$ with
  $n$ even. For all other triangulations, the construction is a
  1-2 correspondence, i.e.\ all triangulations $T$ with
  $\Aut(T,c_0)\simeq \ZZ_2$ can be constructed in precisely two
  different ways.
\end{remark}

\subsection{Rotative symmetries}\label{sec:constr:rotative}

As opposed to girdles, the construction of a graph $F$ that can serve
as a fyke net requires several steps. Suppose that the
desired order $m$ of the automorphism group is already given, as well
as the dimensions of the poles and the number $k+1$ of layers
$H_0,\dotsc,H_k$. We construct $F$ in the following steps.
\begin{itemize}
\item Choose the layer $H_0$ to be a cycle depending on the dimension
  of the north pole $c_0$ like in Figure~\ref{fig:invariantcycles}.
  If $c_0$ is a vertex, let $am$ be the length of $H_0$, otherwise we
  put $a=1$.
\item For $i=1,\dotsc,k-1$, let $C_i$ be a cycle whose length is a
  multiple of $m$. These cycles will serve as the centres of the
  layers. The choice of $C_k$ depends on the dimension of the south
  pole $c_1$: if $c_1$ is a vertex, then $C_k$ only consists of $c_1$
  and the layer $H_k$ will be antarctic. If $c_1$ is an edge, $C_k$
  can either be a cycle of even length, in which case $H_k$ will be
  pseudo-antarctic, or the edge $c_1$ itself, in which case $H_k$ will
  be antarctic. Finally, if $c_1$ is a face, then $C_k$ has to be a
  cycle whose length is divisible by $3$. In that case, $H_k$ will be
  antarctic if $H_k$ is a triangle and pseudo-antarctic otherwise.
\item Choose the bases $u_0^0,\dotsc,u_{am-1}^0$ in $H_0$ in a
  clockwise order like in Definition~\ref{def:extendedspindle}: if
  $c_0$ is an edge, then choose two opposite vertices as
  $u_0^0,u_1^0$ (the other two will then be the end vertices of
  $c_0$), otherwise choose all vertices of $H_0$.
\item Choose the sources $v_0^0,\dotsc,v_{am-1}^0$ as follows: For
  $j=0,\dotsc,a-1$, choose $v_j^0$ to be a vertex on the path
  starting at $u_j^0$ and running along $H_0$ in clockwise direction
  around the north pole to the predecessor of $u_{j+a}^0$ so that $v_1^0,\dotsc,v_a^0$
  appear in clockwise order on $H_0$, but are not necessarily
  distinct. The remaining sources $v_a^0,\dotsc,v_{am-1}^0$ are obtained by recursively applying
  the rotative symmetry $\varphi$. The set of sources in $H_0$ is
  denoted by $S_0$.
\item Recursively for $i=1,\dotsc,k$, choose the bases
  $u_0^i,\dotsc,u_{am-1}^i$ and the sources $v_0^i,\dotsc,v_{am-1}^i$ on
  $C_i$ as follows: if $C_i$ has length $a_im$, pick a subpath
  consisting of $a_i$ vertices and choose $u_0^i,\dotsc,u_{a-1}^i$ from
  this subpath so that they appear in clockwise order on $C_i$. Again,
  the bases do not have to be distinct. Furthermore, if the sources
  $v_l^{i-1}$ and $v_{l+1}^{i-1}$ were identical, the corresponding
  bases $u_l^i$ and $u_{l+1}^i$ should also be identical. The other bases follow again by
  applying symmetry. After choosing the bases, we can pick the sources
  like in the previous step (but note that we do not need any sources
  for $i=k$). We denote the sets of bases and sources in $C_i$ by
  $B_i$ and $S_i$, respectively.
\item Having fixed the bases and sources for every $i$, we can now
  extend the $C_i$ to layers $H_i$. To that end, we need to add
  a suitable plane cactus at every base in $C_i$. For every $i$,
  denote by $\pi_i\colon S_{i-1} \to B_i$ the function that maps $v$
  to $u$ if there is a $j$ with $v=v_j^{i-1}$ and $u=u_j^i$. For every
  base $u$ in $C_i$, we now attach a plane cactus at $u$ that
  has at most as many maximal blocks in its natural order as $u$ has
  preimages under $\pi_i$. Again, it is sufficient
  to choose the cacti for the first $a$ bases, the others are isomorphic
  by symmetry.
\item Finally, we choose the targets and the \liaisonp. For every
  $u\in B_i$ and every $v_j^{i-1} \in (\pi_i)^{-1}(u)$, we choose a vertex in
  the cactus at $u$ to be the target $w_j^i$ for the \liaisons\ $v_j^{i-1}w_j^i$
  according to the following rules:
  \begin{itemize}
  \item The targets are arranged in clockwise order for increasing index
    $j$ and
  \item every maximal block has at least one vertex that does not belong
    to any other block and is chosen as a target.
  \end{itemize}
\end{itemize}

The near-triangulations contained
in a segment and its boundary are characterised by
Lemma~\ref{lem:spindle:segments}.

\begin{definition}
  Let $C$ be a cycle that consists of a path $P_v$ from $v_1$ to
  $v_2$, a path $P_w$ from $w_1$ to $w_2$, and two edges $v_1w_1,
  v_2w_2$. Denote the length of $C$ by $\ell$ and let $u$ be a vertex
  on $P_v$. We choose $v_1$ as the root vertex of $C$ and $v_1w_1$ as
  the root edge. If $N$ is a near-triangulation with $\ell$ outer
  edges, root vertex $v_N$, and root edge $e_N$, let us denote by
  $\alpha$ the isomorphism from $C$ to the boundary of the outer face
  of $N$ that respects the rooting. We call $N$ \emph{$2$-layered with
  respect to $v_1$, $v_2$, $w_1$, $w_2$, and $u$} if it satisfies
  Conditions~\ref{spindle:segment:nochords}--\ref{spindle:segment:rightmost}
  of Lemma~\ref{lem:spindle:segments} with $v_j^i:=\alpha(v_1)$,
  $v_{j+1}^i:=\alpha(v_2)$, $w_j^{i+1}:=\alpha(w_1)$, $w_{j+1}^{i+1}:=\alpha(w_2)$,
  and $u_{j+1}^i:=\alpha(u)$.
\end{definition}

If $T$ is a triangulation with rotative symmetry and $f$ is a segment
of its fyke net $F$, then by Lemma~\ref{lem:spindle:segments} we
have $f=f_j^i$ for some $i,j$ and $N_j^i$ is $2$-layered with respect
to $v_j^i$, $v_{j+1}^i$, $w_j^{i+1}$, $w_{j+1}^{i+1}$, and
$u_{j+1}^i$. Here the boundary of $f_j^i$ is rooted at $v_j^i$ and
$v_j^iw_j^{i+1}$. For every leaf $f$ of the fyke net we choose
the root vertex $v_f$ of its boundary to be its vertex closest to the
base of the branch it is contained in. As root edge $e_f$ we choose
the left of the two edges at $v_f$ on this boundary. Finally, if $F$
has pseudo-antarctic faces, we choose for each such face $f$ the root
edge $e_f$ to be either the south pole $c_1$ (if it is an edge) or the
unique edge incident to both $f$ and $c_1$ (if $c_1$ is a face). As
the root vertex $v_f$ we choose the end vertex of $e_f$ that lies in
clockwise direction from $e_f$ around $f$.

\begin{theorem}\label{thm:rotative}
  The triangulations $T$ with rotative symmetries in $\Aut(c_0,T)$
  are precisely the ones that can be constructed by choosing
  \begin{itemize}
  \item a fyke net $F$,
  \item for every isomorphism class $[f]$ (under rotation) of leaves
    of $F$ a strongly rooted near-triangulation whose boundary is
    isomorphic to the boundaries of those leaves with respect to the
    rooting,
  \item for every isomorphism class $[f]$ of pseudo-antarctic faces
    of $F$ a strongly rooted near-triangulation whose boundary is
    isomorphic to the boundaries of those faces with respect to the
    rooting, and
  \item for every isomorphism class $\{f_j^i,f_{j+a}^i,\dotsc,f_{j+(m-1)a}^i\}$
    of segments of $F$ a $2$-layered near-triangulation with respect to
    $v_j^i$, $v_{j+1}^i$, $w_j^{i+1}$, $w_{j+1}^{i+1}$, and
    $u_{j+1}^i$
  \end{itemize}
  and inserting a copy of each near-triangulation into the corresponding
  faces of $F$ at their root vertices and edges.
  This construction is a 1-1 correspondence.
\end{theorem}

\subsection{Reflective and rotative symmetries}\label{sec:constr:both}

The graphs that can serve as a skeleton of a triangulation can be
constructed as follows. Suppose that the number $n$ of reflections is
given. Then we can choose
\begin{itemize}
\item the number $s$ of isomorphism classes of segments of the
  skeleton,
\item the structure of the skeleton at the poles according to
  Lemma~\ref{lem:skeleton:poles},
\item the numbers $k_1,\dotsc,k_{s-1}$ from
  Lemma~\ref{lem:skeleton:touching}, and
\item the distances of $v_j^1$ and $w_j^1$ on $M_1$ and on $M_2$ for
  every $j=1,\dotsc,s$ as well as the number and distribution of
  diamonds on these meridians between this two vertices. For arbitrary
  $i$, the structure of $M_i$ at the boundaries of the segments is
  identical to that of $M_1$ or $M_2$, depending of the parity of $i$,
  by Lemma~\ref{lem:skeleton1}\ref{skeleton:meridians}.
\end{itemize}

The near-triangulations that can be inserted into a segment are
similar to those that can be inserted into a side of a girdle: if
such a near-triangulation had a chord both of whose end vertices are
central cells of the same meridian, then applying the reflection that
corresponds to that meridian shows that there is a double edge, a
contradiction. In other words, a chord is only allowed if its end
vertices are not in the same meridian or if they are in the same
meridian, but at least one of them is an outer vertex of that
meridian.

\begin{definition}
  Let $N$ be a near-triangulation with root vertex $v_N$ and root edge
  $e_N$. Suppose that a vertex $w_N\not=v_N$ is fixed, then the
  boundary of $N$ is the union of two paths from $v_N$ to $w_N$;
  denote the path that contains $e_N$ by $R$ and the other path by
  $L$. We call $L$ and $R$ the \emph{sides} of the boundary. If vertex
  sets $D_L$ and $D_R$ on $L$ and $R$ are given, we call $N$
  \emph{$2$-sided chordless outside $D_L$ and $D_R$} if every chord of
  $N$ whose end vertices both lie on $L$ or both lie on $R$ has least
  one end vertex in $D_L$ or in $D_R$, respectively.

  More generally, let a cycle $C$ with a root vertex $v_C$ and a root
  edge $e_C$ incident with $v_C$ be given. If $w_C\not=v_C$ is given,
  let us define subpaths $L_C$ and $R_C$ as before. Suppose that
  $D_{L_C}$ and $D_{R_C}$ are sets of vertices on $L_C$ and $R_C$,
  respectively. If there is an isomorphism $\alpha$ from $C$ to the
  boundary of $N$ that respects the rooting and maps $w_C$ to $w_N$,
  then we call $N$ \emph{$2$-sided chordless outside $D_{L_C}$ and
  $D_{R_C}$} if it is $2$-sided chordless outside $\alpha(D_{L_C})$
  and $\alpha(D_{R_C})$.
\end{definition}

Consider a segment $f_j^0$ of the skeleton, let $C_j$ be its boundary.
The two sides of $C_j$ are its intersections $R_j$ with $M_0$ and
$L_j$ with $M_1$, the set $D_{R_j}$ (respectively $D_{L_j}$) is the
set of outer vertices of $M_0$ on $R_j$ (respectively of $M_1$ on $L_j$).
The near-triangulations that can be inserted into $f_j^0$ are
precisely those that are $2$-sided chordless outside $D_{L_j}$ and
$D_{R_j}$.
We thus have the following characterisation of triangulations with
both reflective and rotative symmetries.

\begin{theorem}\label{thm:both}
  The triangulations $T$ for which $\Aut(c_0,T)$ has a subgroup $H$
  isomorphic to $D_n$ are precisely those that can be
  constructed by choosing
  \begin{itemize}
  \item a skeleton $S_H$ and
  \item for every $j=1,\dotsc,s$, a near-triangulation $N_j$ that is $2$-sided
    chordless outside $D_{L_j}$ and $D_{R_j}$,
  \end{itemize}
  and inserting a copy of $N_j$ into $f_j^i$ at $v_j^i$ and $e_j^i$
  for every $j=1,\dotsc,s$ and $i=0,\dotsc,2n-1$.
\end{theorem}

\begin{remark}
  The construction from Theorem~\ref{thm:both} is a 1-2
  correspondence. Indeed, the skeleton $S$ with respect to
  $\Aut(c_0,T)$ is unique up to the enumeration of its meridians. By
  Lemma~\ref{lem:skeleton1}\ref{skeleton:meridians}, the
  $|\Aut(c_0,T)|$ many choices for the enumeration result in two
  different decompositions. Since the choice of $S_H$ only depends
  on which meridian of $S$ is chosen as $M_0$ for $S_H$, we have
  shown that all triangulations $T$ with $\Aut(T,c_0)\supseteq
  H\simeq D_n$ can be constructed in precisely two different ways.
\end{remark}

\section{Discussion and outlook}\label{sec:discussion}

The constructive decomposition presented in this paper is the key to
enumerate triangulations with specific symmetries (and to sample them
uniformly at random, based on a recursive method~\cite{recursive} or
on Boltzmann sampler~\cite{PolyaBoltzmann,boltzmann}). For
this end, it will be necessary to translate the decomposition into
functional equations for the cycle index sums~\cite{Robinson70,Walsh82} that enumerate these
triangulations and the basic structures arising in their
decomposition. This will be done in another paper~\cite{cubicplanar}.

In Section~\ref{sec:constructions} we showed how to construct the
basic structures of the decomposition: girdles, fyke nets, and
skeletons. These constructions will be enough to determine their cycle
index sums, but for a complete set of functional equations we still
need to provide a construction for the different types of
near-triangulations that are to be inserted into the faces of the
basic structures. In~\cite{cubicplanar}, we will present such a
construction, thus completing the constructive decomposition as well
as its interpretation as functional equations.

Our final aim, however, is not to enumerate triangulations, but cubic
planar graphs. This can be achieved along the following lines. From
the enumeration of triangulations, we can obtain an enumeration of
their duals: cubic planar \emph{maps}. More precisely, since the
triangulations considered in this paper are simple, their duals are
precisely the 3-connected cubic planar maps. From 3-connected cubic
planar maps, we can go to 3-connected cubic planar \emph{graphs},
since by Whitney's Theorem, every such graph has a unique embedding up
to orientation. However, the correspondence between maps and graphs
implied by Whitney's Theorem is not a 2-1 correspondence. Indeed, if a
3-connected cubic planar graph has a reflective symmetry, then it has
only one embedding. Since we distinguished triangulations with
reflective symmetries and triangulations without reflective
symmetries, we will be able to obtain relations between graphs and
maps separately for each of the two cases, thus resulting in an
enumeration of all 3-connected cubic planar graphs. Using the grammar
developed in~\cite{grammar}, we will then obtain an enumeration of all
cubic planar graphs.


\vfill

\noindent
Mihyun Kang, Philipp Sprüssel\\
Institute of Optimization and Discrete Mathematics\\
Graz University of Technology\\
8010 Graz, Austria\\
email: {\tt \{kang,spruessel\}@math.tugraz.at}

\end{document}